\documentclass[11pt]{amsart}

\usepackage{graphics}
\usepackage{color}
\usepackage[a4paper, margin=3cm]{geometry}
\usepackage{array}
\usepackage{amssymb,enumerate}
\usepackage{amsmath}
\usepackage{bbm}           
\usepackage{bm}
\usepackage{eso-pic,graphicx}
\usepackage{tikz}
\usepackage{cite}
\usepackage{esint}
\usepackage[pdftex, bookmarksnumbered, bookmarksopen, colorlinks, citecolor=blue, linkcolor=blue]{hyperref}

\usepackage{graphicx}
\usepackage{bbding}
\usepackage{makecell}
\usepackage{amsmath,amsfonts}
\usepackage{rotating}
\usepackage{amssymb}
\usepackage{verbatim}
\usepackage{rotating}
\usepackage{mathrsfs}
\DeclareSymbolFont{largesymbols}{OMX}{cmex}{m}{n}
\makeatletter
\def\Ddots{\mathinner{\mkern1mu\raise\p@
\vbox{\kern7\p@\hbox{.}}\mkern2mu
\raise4\p@\hbox{.}\mkern2mu\raise7\p@\hbox{.}\mkern1mu}}
\makeatother

\def\XXint#1#2#3{{\setbox0=\hbox{$#1{#2#3}{\int}$}
\vcenter{\hbox{$#2#3$}}\kern-.5\wd0}}

\begin{document}

\newtheorem{hyp}{Hypothesis}

\newtheorem{hyp2}[hyp]{Hypothesis}

\newtheorem{definition}{Definition}
\newtheorem{theorem}[definition]{Theorem}
\newtheorem{proposition}[definition]{Proposition}
\newtheorem{conjecture}[definition]{Conjecture}
\def\theconjecture{\unskip}
\newtheorem{corollary}[definition]{Corollary}
\newtheorem{lemma}[definition]{Lemma}
\newtheorem{claim}[definition]{Claim}
\newtheorem{sublemma}[definition]{Sublemma}
\newtheorem{observation}[definition]{Observation}
\theoremstyle{definition}

\newtheorem{notation}[definition]{Notation}
\newtheorem{remark}[definition]{Remark}
\newtheorem{question}[definition]{Question}

\newtheorem{example}[definition]{Example}
\newtheorem{problem}[definition]{Problem}
\newtheorem{exercise}[definition]{Exercise}
 \newtheorem{thm}{Theorem}
 \newtheorem{cor}[thm]{Corollary}
 \newtheorem{lem}{Lemma}[section]
 \newtheorem{prop}[thm]{Proposition}
 \theoremstyle{definition}
 \newtheorem{dfn}[thm]{Definition}
 \theoremstyle{remark}
 \newtheorem{rem}{Remark}
 \newtheorem{ex}{Example}
 \numberwithin{equation}{section}

\def\C{\mathbb{C}}
\def\R{\mathbb{R}}
\def\Rn{{\mathbb{R}^n}}
\def\Rns{{\mathbb{R}^{n+1}}}
\def\Sn{{{S}^{n-1}}}
\def\M{\mathbb{M}}
\def\N{\mathbb{N}}
\def\Q{{\mathbb{Q}}}
\def\Z{\mathbb{Z}}
\def\X{\mathbb{X}}
\def\Y{\mathbb{Y}}
\def\F{\mathcal{F}}
\def\L{\mathcal{L}}
\def\S{\mathcal{S}}
\def\supp{\operatorname{supp}}
\def\essi{\operatornamewithlimits{ess\,inf}}
\def\esss{\operatornamewithlimits{ess\,sup}}

\numberwithin{equation}{section}
\numberwithin{thm}{section}
\numberwithin{definition}{section}
\numberwithin{equation}{section}

\def\earrow{{\mathbf e}}
\def\rarrow{{\mathbf r}}
\def\uarrow{{\mathbf u}}
\def\varrow{{\mathbf V}}
\def\tpar{T_{\rm par}}
\def\apar{A_{\rm par}}

\def\reals{{\mathbb R}}
\def\torus{{\mathbb T}}
\def\t{{\mathcal T}}
\def\heis{{\mathbb H}}
\def\integers{{\mathbb Z}}
\def\z{{\mathbb Z}}
\def\naturals{{\mathbb N}}
\def\complex{{\mathbb C}\/}
\def\distance{\operatorname{distance}\,}
\def\support{\operatorname{support}\,}
\def\dist{\operatorname{dist}\,}
\def\Span{\operatorname{span}\,}
\def\degree{\operatorname{degree}\,}
\def\kernel{\operatorname{kernel}\,}
\def\dim{\operatorname{dim}\,}
\def\codim{\operatorname{codim}}
\def\trace{\operatorname{trace\,}}
\def\Span{\operatorname{span}\,}
\def\dimension{\operatorname{dimension}\,}
\def\codimension{\operatorname{codimension}\,}
\def\nullspace{\scriptk}
\def\kernel{\operatorname{Ker}}
\def\ZZ{ {\mathbb Z} }
\def\p{\partial}
\def\rp{{ ^{-1} }}
\def\Re{\operatorname{Re\,} }
\def\Im{\operatorname{Im\,} }
\def\ov{\overline}
\def\eps{\varepsilon}
\def\lt{L^2}
\def\diver{\operatorname{div}}
\def\curl{\operatorname{curl}}
\def\etta{\eta}
\newcommand{\norm}[1]{ \|  #1 \|}
\def\expect{\mathbb E}
\def\bull{$\bullet$\ }

\def\blue{\color{blue}}
\def\red{\color{red}}

\def\xone{x_1}
\def\xtwo{x_2}
\def\xq{x_2+x_1^2}
\newcommand{\abr}[1]{ \langle  #1 \rangle}

\newcommand{\Norm}[1]{ \left\|  #1 \right\| }
\newcommand{\set}[1]{ \left\{ #1 \right\} }
\newcommand{\ifou}{\raisebox{-1ex}{$\check{}$}}
\def\one{\mathbf 1}
\def\whole{\mathbf V}
\newcommand{\modulo}[2]{[#1]_{#2}}
\def \essinf{\mathop{\rm essinf}}
\def\scriptf{{\mathcal F}}
\def\scriptg{{\mathcal G}}
\def\m{{\mathcal M}}
\def\scriptb{{\mathcal B}}
\def\scriptc{{\mathcal C}}
\def\scriptt{{\mathcal T}}
\def\scripti{{\mathcal I}}
\def\scripte{{\mathcal E}}
\def\V{{\mathcal V}}
\def\scriptw{{\mathcal W}}
\def\scriptu{{\mathcal U}}
\def\scriptS{{\mathcal S}}
\def\scripta{{\mathcal A}}
\def\scriptr{{\mathcal R}}
\def\scripto{{\mathcal O}}
\def\scripth{{\mathcal H}}
\def\scriptd{{\mathcal D}}
\def\scriptl{{\mathcal L}}
\def\scriptn{{\mathcal N}}
\def\scriptp{{\mathcal P}}
\def\scriptk{{\mathcal K}}
\def\frakv{{\mathfrak V}}
\def\v{{\mathcal V}}
\def\C{\mathbb{C}}
\def\D{\mathcal{D}}
\def\R{\mathbb{R}}
\def\Rn{{\mathbb{R}^n}}
\def\rn{{\mathbb{R}^n}}
\def\Rm{{\mathbb{R}^{2n}}}
\def\r2n{{\mathbb{R}^{2n}}}
\def\Sn{{{S}^{n-1}}}
\def\bbM{\mathbb{M}}
\def\N{\mathbb{N}}
\def\Q{{\mathcal{Q}}}
\def\Z{\mathbb{Z}}
\def\F{\mathcal{F}}
\def\L{\mathcal{L}}
\def\G{\mathscr{G}}
\def\ch{\operatorname{ch}}
\def\supp{\operatorname{supp}}
\def\dist{\operatorname{dist}}
\def\essi{\operatornamewithlimits{ess\,inf}}
\def\esss{\operatornamewithlimits{ess\,sup}}
\def\dis{\displaystyle}
\def\dsum{\displaystyle\sum}
\def\dint{\displaystyle\int}
\def\dfrac{\displaystyle\frac}
\def\dsup{\displaystyle\sup}
\def\dlim{\displaystyle\lim}
\def\bom{\Omega}
\def\om{\omega}

\author[J. Tan]{Jiawei Tan}
\address{Jiawei Tan:
School of Mathematical Sciences \\
Beijing Normal University \\
Laboratory of Mathematics and Complex Systems \\
Ministry of Education \\
Beijing 100875 \\
People's Republic of China}
\email{jwtan@mail.bnu.edu.cn}

\author[Q. Xue]{Qingying Xue$^{*}$}
\address{Qingying Xue:
	School of Mathematical Sciences \\
	Beijing Normal University \\
	Laboratory of Mathematics and Complex Systems \\
	Ministry of Education \\
	Beijing 100875 \\
	People's Republic of China}
\email{qyxue@bnu.edu.cn}


\thanks{The authors were partly supported by the National Key R\&D Program of China (No. 2020YFA0712900) and NNSF of China (No. 12271041).
\thanks{$^{*}$ Corresponding author, e-mail address: qyxue@bnu.edu.cn}}

\date{\today}
\title[ SHARP WEIGHTED INEQUALITIES FOR ITERATED COMMUTATORS  ]
{\bf Sharp weighted inequalities for iterated commutators of a class of multilinear operators}

\begin{abstract}
In this paper, the sharp quantitative weighted bounds for the iterated commutators of a class of multilinear operators were systematically studied. This class of operators contains multilinear Calder\'{o}n-Zygmund  operators, multilinear Fourier integral operators, and multilinear Littlewood-Paley square operators as its typical examples. These were done only under two pretty much general assumptions of pointwise sparse domination estimates. We first established local decay estimates and quantitative weak $A_\infty$ decay estimates for iterated commutators of this class of operators. Then, we considered the corresponding Coifman-Fefferman inequalities and the mixed weak type estimates associated with Sawyer's conjecture.  Beyond that, the Fefferman-Stein inequalities with respect to arbitrary weights and weighted modular inequalities were also given. As applications, it was shown that all the conclusions aforementioned can be applied to multilinear $\omega$-Calder\'{o}n-Zygmund  operators, multilinear maximal singular integral operators, multilinear pseudo-differential operators, Stein's square functions, and higher order Calder\'{o}n commutators.
\end{abstract}\maketitle

\section{Introduction and main results}
\subsection{Motivation}\
\par
The main purpose of this paper is to develop a systematic sharp quantitative weighted theory for iterated commutators of a class of multilinear operators, which includes the classical Calder\'{o}n-Zygmund operators as well as numerous operators beyond the multilinear Calder\'{o}n-Zygmund theory. Our motivation lies in three aspects: \par
\textbf{(1).} During the past two decades, the theory of sparse domination has been developing rapidly. It was well known that, the sparse domination approach can significantly simplify the proof of the $A_2$ conjecture \cite{hyt0}, which states that if $T$ is a Calder\'{o}n-Zygmund operator with a H\"{o}lder-Lipschitz kernel, then whether it holds that
$$
\|T f\|_{L^2(w)} \leq c_{n, T}[w]_{A_2}\|f\|_{L^2(w)},
$$
where the definition of $ [w]_{A_2}$ is listed in Section \ref{sub2}. Let $\mathcal{X}$ be a Banach functions space and define
$$
\mathcal{A}_{\mathcal{S}} f(x)=\sum_{Q \in \mathcal{S}} \frac{1}{|Q|} \int_Q|f(x)|dx \chi_Q(x),
$$where each $Q$ is a cube with its sides parallel to the axis and $\mathcal{S}$ is a sparse family of such cubes.
In \cite{ler3}, Lerner proved that any standard Calder\'{o}n-Zygmund operator can be controlled in norm by a family of sparse operators in the way that
$$
\|T f\|_\mathcal{X} \leq \sup _{\mathcal{S}}\left\|\mathcal{A}_{\mathcal{S}} f\right\|_\mathcal{X}.
$$
This estimate combined with the following inequality obtained in \cite{cru2}
$$
\left\|\mathcal{A}_{\mathcal{S}}\right\|_{L^2(w) \rightarrow L^2(w)} \leq c_n[w]_{A_2}
$$
gives an simple proof of the $A_2$ conjecture.

On the other hand, Lerner and  Nazarov \cite{ler2}, Conde-Alonso and Rey \cite{con} independently showed that the Calder\'{o}n-Zygmund operator $T$ can be dominated pointwisely by a finite number of sparse operators, 
$$
|T f(x)| \leq c_{n,T} \sum_{j=1}^{3^n} \mathcal{A}_{\mathcal{S}_j} f(x).
$$ With this estimate in hand, Conde-Alonso and Rey \cite{con} answered an open question originally posed by Lerner \cite{ler3}.

These two examples illustrate that the method of sparse domination plays an important role in modern analysis. In fact, this
method was widely used in the study of several important operators in Harmonic analysis, such as Bochner-Riesz multipliers \cite{ben}, singular integrals satisfying the $L^r$-H\"{o}rmander condition \cite{li2}, rough singular integrals \cite{con1}, as well as singular non-integral type operators \cite{ber}.\par
Due to the importance of pointwise sparse domination and the fact that all the operators mentioned enjoy some kind of sparse domination estimates.  It is quite natural to ask if only with the assumation of pointwise sparse domination estimates, what kind of properties could be obtained  in general for these operators. This is the first motivation and the starting point of this paper.\par
\textbf{(2).} Secondly, multilinear theory is an essential extension for linear theory. Analogous to multivariable functional calculus, multivariable calculus provides a robust approach to the study of functions of more than one variable, and it broadens the narrow perspective of studying a single variable by freezing other variables. Multilinear analysis focuses on the investigation of operators linearly related to more than one function, treating all inputs as variables rather than just dealing with certain parameters. The study of multilinear theories, often based on the simultaneous decomposition of multiple variables, is naturally more complex than linear analyses, but it is more far-reaching and the results are more flexible. In Harmonic analysis, there are numerous examples of linear operators with fixed parameters that can be considered as multilinear operators: multiplier operators, Littlewood-Paley operators, Calder\'{o}n commutators, and Cauchy integrals along Lipschitz curves (see \cite{gra2}).\par
The development of multilinear Calder\'{o}n-Zygmund theory has made great progress in recent decades, see for example  \cite{ler4,gra1}. At the same time, however, it is noted that some important operators beyond the multilinear Calder\'{o}n-Zygmund theory, including multilinear singular integrals with non-smooth kernels \cite{duo}, multilinear pseudo-differential operators \cite{cao}, Calder\'{o}n commutators \cite{DL2019}, Stein's square functions\cite{CJSK}, etc., have various properties that imitate the Calder\'{o}n-Zygmund operators. For example, they enjoy the same sparse domination, $L^p$ boundedness, end-point weak-type estimate, and so on.\par

Based on certain assumptions of sparse domination, our attention has been drawn to the question of how to extract and analyze commonalities among various multilinear operators and their iterated commutators. These properties encompass the weighted local exponential decay property, the Coifman-Fefferman inequality, mixed weighted estimation, and the weighted modular inequality. \par

\textbf{(3).} We note that local decay estimates and mixed weak type inequalities of the commutators of a class of multilinear bounded oscillation operators have been established in \cite{cao1}. The method used to prove the local decay estimates in \cite{cao1} relies heavily on the local Coifman-Fefferman inequality as well as on the Rubio de Francia algorithm, under which the dependence of the constants on the weights cannot be obtained. Then one may wonder whether it is possible to give quantitatively weighted local decay estimates for the multilinear operators and their commutators, and obtain a multilinear version of the mixed weighted inequality with optimized or refined constants. These are the sources of our third motivation.

\subsection{Two fundamental hypotheses}\
\par
We need to present some definitions.
Let us start with the definition of general commutators. Let $\mathcal{T}$ be a $m$-linear operator from $\mathscr{X}_1 \times \cdots \times$ $\mathscr{X}_m$ into $\mathscr{Y}$, where $\mathscr{X}_1, \ldots, \mathscr{X}_m$ are some normed spaces and $\mathscr{Y}$ is a quasi-normed space. In our following statements, $\mathscr{X}_1, \ldots, \mathscr{X}_m$ and $\mathscr{Y}$ will be appropriately weighted Lebesgue spaces.

\begin{definition}[\textbf{$k$-th order commutators}]\label{def1.1}
Given $\vec{f}:=\left(f_1, \ldots, f_m\right) \in \mathscr{X}_1 \times \cdots \times \mathscr{X}_m, \vec{b}=\left(b_{i_1}, \ldots, b_{i_l}\right)$ of measurable functions with $I:=\{i_1,\ldots,i_l\}\subseteq \{1,\ldots,m\}$, and $k \in \mathbb{N}$, we define, whenever it makes sense, the $k$-th order commutator of $\mathcal{T}$ in the $i$-th entry of $\mathcal{T}$ as
$$
[\mathcal{T}, \vec{b}]_{k e_i}(\vec{f})(x):=\mathcal{T}\left(f_1, \ldots,\left(b_i(x)-b_i\right)^k f_i, \ldots, f_m\right)(x),
$$
where $e_i$ is the basis of $\mathbb{R}^n$ with the $i$-th component being 1 and other components being 0.\\
Furthermore, if $k=1$, we write
$$
[\mathcal{T}, \vec{b}]_{e_i}(\vec{f})(x)=b_i(x)\mathcal{T}\left(f_1, \ldots, f_m\right)(x)-\mathcal{T}\left(f_1, \ldots, b_if_i, \ldots, f_m\right)(x).
$$
\end{definition}

 Then, for a multi-index $\vec{\alpha}=\left(\alpha_1, \ldots, \alpha_m\right) \in \mathbb{N}^m$, we denote
$$
[\mathcal{T}, \vec{b}]_{\vec{\alpha}}:=\left[\cdots\left[[\mathcal{T}, \vec{b}]_{\alpha_1 e_1}, \vec{b}\right]_{\alpha_2 e_2} \cdots, \vec{b}\right]_{\alpha_m e_m}.
$$
Using this notation, the iterated commutator of $\mathcal{T}$ is defined as follows.
\begin{definition}[\textbf{iterated commutators}]\label{def1.2}
Given $l\leq m, \vec{f}:=\left(f_1, \ldots, f_m\right) \in \mathscr{X}_1 \times \cdots \times \mathscr{X}_m, \vec{b}=\left(b_{i_1}, \ldots, b_{i_l}\right)$ of measurable functions with $I:=\{i_1,\ldots,i_l\}\subseteq \{1,\ldots,m\}$. The $m$-linear iterated commutator of $\mathcal{T}$ is given by
$$
\mathcal{T}_{\vec{b}}(\vec{f})(x):=\left[\cdots\left[[\mathcal{T}, \vec{b}]_{e_{i_1}}, \vec{b}\right]_{e_{i_2}} \cdots, \vec{b}\right]_{e_{i_l}}.
$$
\end{definition}

When $\vec{b}=(b, \ldots, b)$, we denote $\mathcal{T}_{\vec{b}}:=\mathcal{T}_{b}$. In particular, if $\mathcal{T}$ is an $m$-linear operator with a kernel representation of the form
$$
\mathcal{T}(\vec{f})(x)=\int_{\mathbb{R}^{nm}} K(x, \vec{y}) f_1\left(y_1\right) \cdots f_m\left(y_m\right) d\vec{y}
$$
where $d\vec{y}=dy_1\cdots dy_m,$
then one can easily verify that $\mathcal{T}_{\vec{b}}$ has the following expression:
\begin{equation*}
\mathcal{T}_{\vec{b}}(\vec{f})(x)=\int_{\mathbb{R}^{nm}} \prod_{s=1}^l\left(b_{i_s}(x)-b_{i_s}(y_{i_s})\right) K(x, \vec{y}) \prod_{s=1}^m f_s\left(y_s\right) d\vec{y}.
\end{equation*}\par
We should mention that the commutators given in Definition \ref{def1.1} were originally introduced by P\'{e}rez and Torres
\cite{per4} in the study of the $m$-linear Calder\'{o}n-Zygmund operators. Weighted strong as well as weak type endpoint estimates for the iterated commutators of the Calder\'{o}n-Zygmund operators in Definition \ref{def1.2} were proved in \cite{per5} by P\'{e}rez et al. Recently, using the method of extrapolation, the weighted boundedness results for the general commutator $[\mathcal{T}, \vec{b}]_\alpha$ have also been proved in \cite{beny}. We refer to \cite{cao,cao1,ler1} for more information about commutators.\par
To introduce our hypotheses, we need to fix some notation. Given $r>0,$ we set
$$\langle|f|^{r}\rangle_{Q}=\frac{1}{|Q|} \int_{Q}|f(y)|^{r} d y,$$
particularly, $\langle f\rangle_{Q}=\frac{1}{|Q|} \int_{Q}f(y) d y.$
 Let $l(Q)$ be the side length of a cube $Q$ and $rQ$ be the unique cube with sides parallel to the axes having the same center as $Q$ and having side length $l(rQ)=rl(Q)$.

The following two hypotheses are crucial to our forthcoming discussion of $\mathcal{T}_{\vec{b}}.$
\begin{hyp}\label{hyp1}
Let $I =\{i_1,\ldots,i_l\}\subseteq \{1,\ldots,m\}$ and $\vec{b}=\left(b_{i_1}, \ldots, b_{i_l}\right)$ be locally integrable functions defined on $\mathbb{R}^n$. Let $\mathcal{T}$ be an $m$-linear operator and $\mathcal{T}_{\vec{b}}$ be its commutator given in Definition \ref{def1.2}. Suppose that for all $Q_0$ cubes in $\mathbb{R}^n$ and for any bounded functions $\vec{f}=\left(f_1, \ldots, f_m\right)$ with compact support, there exists a sparse collection $\mathcal{F}\subseteq \mathcal{D}(Q_0)$ (see Section \ref{sub1}) such that
for a.e. $x \in Q_0$,

\begin{equation*}
\begin{aligned}
{\left|\mathcal{T}_{\vec{b}} (f_1\chi_{3Q_{0}},\ldots,f_m\chi_{3Q_{0}} )(x)\right|}&
\leq C \sum_{Q \in {\mathcal{F}}}\left(\sum_{\vec{\gamma} \in {\{1,2\}^{l}}}  \prod_{s=1}^l \mathcal{R}(b_{i_s},f_{i_s},Q,\gamma_{i_s}) \right)\prod_{s\notin I} {\langle |f_s |\rangle}_{3Q}\chi_{Q}(x),\\
\end{aligned}
\end{equation*}
where
\begin{equation*}
\mathcal{R}(b,f,Q,\gamma)= \begin{cases}|b-{\langle b\rangle}_{3Q}|{\langle |f |\rangle}_{3Q}, & \text { if } \gamma=1, \\ {\langle |(b-{\langle b \rangle}_{3Q} )f  |\rangle}_{3Q}, & \text { if } \gamma=2.\end{cases}
\end{equation*}
\end{hyp}
\begin{hyp2}\label{hyp2}
 Let $I =\{i_1,\ldots,i_l\}\subseteq \{1,\ldots,m\}$ and $\vec{b}=\left(b_{i_1}, \ldots, b_{i_l}\right)$ be locally integrable functions defined on $\mathbb{R}^n$. Let $\mathcal{T}$ be an $m$-linear operator and $\mathcal{T}_{\vec{b}}$ be its commutator given in Definition \ref{def1.2}. Suppose that for any bounded functions $\vec{f}=\left(f_1, \ldots, f_m\right)$ with compact support, there exist $3^n$ sparse collections $\left\{\mathcal{S}_j\right\}_{j=1}^{3^n}$ such that
 $$
|\mathcal{T}_{\vec{b}}(\vec{f})(x)| \leq C\left(\sum_{j=1}^{3^n} \sum_{\vec{\gamma} \in\{1,2\}^l} \mathcal{A}_{\mathcal{S}_j, b}^{\vec{\gamma}}(\vec{f})(x)\right), \quad \text { a.e. } x \in \mathbb{R}^n,
$$
where
$$
\begin{gathered}
\mathcal{A}_{\mathcal{S}_j, b}^{\vec{\gamma}}(\vec{f})(x):=\sum_{Q \in \mathcal{S}_j}\left(\prod_{s=1}^l \mathcal{U}\left(b_{i_s}, f_{i_s}, Q, \gamma_{i_s}\right)(x)\right)\left(\prod_{s \notin I}\left\langle\left|f_s\right|\right\rangle_Q\right) \chi_Q(x), \quad \text { with } \\
\mathcal{U}(b, f, Q, \gamma)(x)= \begin{cases}\left|b(x)-\langle b\rangle_Q\right|\langle|f|\rangle_Q & \text { if } \gamma=1, \\
\left\langle\left|\left(b-\langle b\rangle_Q\right) f\right|\right\rangle_Q & \text { if } \gamma=2 .\end{cases}
\end{gathered}
$$

\end{hyp2}
 \par
 \begin{remark}
 We now make some comments on these two hypotheses. First of all, the estimate in the form of Hypothesis \ref{hyp1} holds for many operators, e.g.,  multilinear $w$-Calder\'{o}n-Zygmund  operators, multilinear pseudo-differential operators, etc., for more details see \cite[p. 166]{iba}. Secondly, if $I=\emptyset$, then Hypothesis \ref{hyp2} can be rewritten as
 $$
|\mathcal{T}(\vec{f})(x)| \leq C\sum_{j=1}^{3^n} \sum_{Q \in \mathcal{S}_j}\prod_{s =1 }^m\left\langle\left|f_s\right|\right\rangle_Q\chi_Q(x), \quad \text { a.e. } x \in \mathbb{R}^n,
$$which is consistent with the sparse domination obtained for classical multilinear Calder\'{o}n-Zygmund operators \cite[Theorem 1.4]{dam} and multilinear pseudo-differential operators \cite[Proposition 4.1]{cao}, and so on.
 \end{remark}
 \subsection{Main results}\label{sub1.3}\
\par
The main contributions of this paper are as follows:
\begin{enumerate}
	\item  [	{{\color{black}$\bullet$}}]
Our general framework gives a unified approach to study the quantitatively weighted estimations of commutators for a class of multilinear operators. This class of operators includes the multilinear Calder\'{o}n-Zygmund operators, multilinear Littlewood-Paley square operators, as well as other operators beyond multilinear Calder\'{o}n-Zygmund theory, such as Fourier integral operators and Calder\'{o}n commutators, etc., (cf. Sect. 8). We only assume that some sparse domination estimates holds for this class of operators. It should also be pointed out that all the results in this paper still hold for this class of multilinear operators itself (with no commutators), but one has to modify some places if needed.
	\item  [	{{\color{black}$\bullet$}}]The first main result, Theorem \ref{thm1.1}, gives a sharp weighted local sub-exponential decay estimate for the iterated commutators of a class of multilinear operators which essentially improves the result in \cite{cao1} and it is sharp for sub-exponential decay. These results accurately reflect the extent that an operator is locally controlled by certain maximal operator in the weighted case, thus improving the corresponding good-$\lambda$ inequalities. A quick comparison with \cite{cao1} reveals that the methodology and the whole proof scheme differ in a number of key points. For instance, we take a more direct approach and obtain a quantitatively weighted exponential decay estimate directly via pointwise sparse domination, see Remark \ref{remmark2}.
	\item  [	{{\color{black}$\bullet$}}]The second main result, Theorem \ref{thm1.3}, presents a weighted mixed weak type inequality, which improves the classical endpoint weighted inequality. In order to obtain the exact constant estimate, we used endpoint extrapolation techniques from \cite{li1} and gave a quantitative weighted Coifman-Fefferman inequality (Theorem \ref{thm1.2}).

	\item  [	{{\color{black}$\bullet$}}]Theorem \ref{thm1.4} focuses on the multilinear Fefferman-Stein inequalities with respect to arbitrary weights and Theorem \ref{thm1.5} establishes two weighted modular inequalities. The dependence between the constants and the weight functions is given, respectively.
\end{enumerate}

We are now in a position to state our main theorems. The first one is the local decay estimate of $\mathcal{T}_{\vec{b}}$ as follows:
 \begin{theorem}\label{thm1.1}
Let $I=\{i_1,\ldots,i_l\}=\{1,\ldots,l\}\subseteq \{1,\ldots,m\}.$ Let $Q$ be a cube and $f_s \in L_c^{\infty}\left(\mathbb{R}^n\right)$ such that $\operatorname{supp}\left(f_s\right) \subset Q$ for $1 \leq s\leq m$. If $\vec{b} \in \mathrm{BMO}^l$ and $\mathcal{T}_{\vec{b}}$ satisfies the Hypothesis $\ref{hyp1}$, then there are constants $\alpha_1, c_1>0$ such that
\begin{equation}\label{ie7}
\begin{aligned}
&\left|\left\{x \in Q:\big|\mathcal{T}_{\vec{b}}(\vec{f})(x)\big|>t \min{\{\mathcal{M}_{L(\log L)}^{(1,l)}(\vec{f})(x), \mathcal{M}(\vec{f_0})(x)\}}\right\}\right| \\
&\hspace{5cm} \qquad\leq c_1 e^{-\alpha_1(\frac{t}{\prod_{s=1}^l\|b_s\|_{\mathrm{BMO}}})^{1/(l+1)}}|Q|, \quad t>0
\end{aligned}
\end{equation}
where $\vec{f_0}=(\mathcal{A}_{{\mathcal{S}}^*}f_{1},\ldots,\mathcal{A}_{{\mathcal{S}}^*}f_{l}, f_{l+1},\ldots,f_m)$ with a sparse family ${\mathcal{S}}^*$ and  $$\mathcal{M}_{L(\log L)}^{(1,l)}(\vec{f})(x)=\sup\limits_{x\in Q}\prod\limits_{s =1}\limits^{l} \|f_s\|_{L(\log L),Q}\prod\limits_{s =l+1}\limits^{m}{\langle |f_s|\rangle}_{Q}.$$
 Moreover, the local decay estimate in $(\ref{ie7})$ is sharp in the sense that it does not hold for any $l_0> \frac{1}{1+l}.$
\end{theorem}
 \begin{remark}\label{remmark2}
	 We now make some comments on Theorem \ref{thm1.1}. Note that $$\min{\{\mathcal{M}_{L(\log L)}^{(1,l)}(\vec{f})(x), \mathcal{M}(\vec{f_0})(x)\}} \leq \mathcal{M}_{L(\log L)}(\vec{f})(x)\leq \mathcal{M}(Mf_1, \ldots, Mf_m),$$
 then this shows that Theorem \ref{thm1.1} improves Theorem 1.6 in \cite{cao1} substantially. Furthermore, when $m=1$ and $\mathcal{T}$ is a Calder\'{o}n-Zygmund operator, Theorem \ref{thm1.1} coincides with the main conclusion in \cite{ort}.
 \end{remark}
As a corollary of Theorem \ref{thm1.1}, we have the following weighted decay estimates of $\mathcal{T}_{\vec{b}}$.
 \begin{corollary}\label{cor1.1}
Let $w\in A_{\infty}^{\text{weak}},I=\{i_1,\ldots,i_l\}=\{1,\ldots,l\}\subseteq \{1,\ldots,m\}.$ Let $Q$ be a cube and $f_s \in L_c^{\infty}\left(\mathbb{R}^n\right)$ such that $\operatorname{supp}\left(f_s\right) \subset Q$ for $1 \leq s\leq m$. If $\vec{b} \in \mathrm{BMO}^l$ and $\mathcal{T}_{\vec{b}}$ satisfies the Hypothesis $\ref{hyp1}$, then there are constants $\alpha_2, c_2>0$ independent of $w$ such that 
\begin{equation}\label{ie6}
\begin{aligned}
&w\left(\left\{x \in Q:\big|\mathcal{T}_{\vec{b}}(\vec{f})(x)\big|>t \mathcal{M}_{L(\log L)}(\vec{f})(x)\right\}\right) \\
&\hspace{3cm} \qquad\leq c_2 e^{-{\frac{\alpha_2}{[w]_{A_\infty}^{weak}+1}}\left({\frac{t}{\prod_{s =1}^{l} \|b_s\|_{\mathrm{BMO}}}}\right)^{\frac{1}{l+1}}} w(2Q), \quad t>0,
\end{aligned}
\end{equation}
\end{corollary}
 \begin{remark}
 Clearly, when $m=1$ and $w\equiv 1,$ Theorem 2.8 in \cite{iba} is just a special case of Corollary \ref{cor1.1}. In addition, since $A_\infty \subsetneq A_{\infty}^{\text{weak}},$  Corollary \ref{cor1.1} is also valid for any $w\in A_\infty$, see Corollary \ref{cor1.1(1)}.
 \end{remark}
 For the Coifman-Fefferman inequality with the multilinear form of the iterated commutator of $\mathcal{T}$, we have
 \begin{theorem}\label{thm1.2}
Let $I=\{i_1,\ldots,i_l\}=\{1,\ldots,l\}\subseteq \{1,\ldots,m\}.$  If $\vec{b} \in \mathrm{BMO}^l$ and $\mathcal{T}_{\vec{b}}$ satisfies the Hypothesis $\ref{hyp2}$, then for any $0<p<\infty, w\in A_\infty,$

\begin{equation}\label{iethm1.2}
\int_{\mathbb{R}^n}\big|\mathcal{T}_{\vec{b}}(\vec{f})(x)\big|^pw(x)dx \lesssim \prod_{s =1}^{l} \|b_s\|_{\mathrm{BMO}}[w]_{A_\infty}^{pl}
[w]_{A_\infty}^{\max\{2,p\}}\int_{\mathbb{R}^n}\left(\mathcal{M}_{L(\log L)}(\vec{f})(x)\right)^pw(x)dx.
\end{equation}

\end{theorem}
 \begin{remark}
 In general, the Coifman-Fefferman inequalities are obtained by extrapolation (see for example \cite[Proposition 5.1]{cao}), here we use sparse domination to obtain better estimates of the weight constants. This provides convenience for the following quantitatively weighted mixed weak type estimation.
 \end{remark}
 For the endpoint case, we establish weighted mixed weak type inequalities with a precisely weighted constant.
 \begin{theorem}\label{thm1.3}
Let $I=\{i_1,\ldots,i_l\}=\{1,\ldots,l\}\subseteq \{1,\ldots,m\}.$  If $\vec{b} \in \mathrm{BMO}^l$ and $\mathcal{T}_{\vec{b}}$ satisfies the Hypothesis $\ref{hyp2}.$ Let $\vec{w}=\left(w_1, \ldots, w_m\right)$ and $u=\prod_{i=1}^m w_i^{1 / m}$. If $\vec{w} \in A_{\vec{1}}$ and $v \in A_{\infty},$ then there exists $t>1 $ depending only on $v$, such that
$$
\left\|\frac{\mathcal{T}_{\vec{b}}(\vec{f})}{v}\right\|_{L^{\frac{1}{m}, \infty}(u v^{\frac{1}{m}})} \lesssim K_0^{2l+6m}[v^{\frac{1}{m}}]_{A_t}^{2l+4m}\prod_{s =1}^{l} \|b_s\|_{\mathrm{BMO}}\left\|\frac{ \mathcal{M}_{L(\log L)} (\vec{f})}{v}\right\|_{L^{\frac{1}{m}, \infty}(u v^{\frac{1}{m}})},
$$
where $K_0=4C_np_0p_0^{\prime}([u]_{A_1}+2^{p_0-1}C_n^t[v^{\frac{1}{m}}]_{A_t}^2[u]_{A_1}^{p_0-1})+1$ with $p_0=2^{n+3}(t-1)[u]_{A_1}+1.$

\end{theorem}
\begin{remark}
	Very recently, in the linear case, Ib\'{a}\~{n}ez-Firnkorn and Rivera-R\'{\i}os \cite{iba1} established mixed-weighted endpoint estimates for the commutators of a class of linear operators and gave some
	 weighted constant estimates. 
We would like to point out that Theorem \ref{thm1.3} improves the results in \cite{cao1,li1} in two aspects.  It yields a more accurate norm constant for the weights $u$ and $v,$ and reduces the condition $v^{\frac{1}{m}} \in A_{\infty}$ to $v \in A_{\infty}.$
\end{remark}

\begin{remark}\label{remark1}
In particular, under the above theorem conditions, if  $\vec{w} \in A_{\vec{1}}$ and $v \in A_{p}(1<p<\infty),$ then we have
$$
\left\|\frac{\mathcal{T}_{\vec{b}}(\vec{f})}{v}\right\|_{L^{\frac{1}{m}, \infty}(u v^{\frac{1}{m}})} \lesssim \tilde{K}_0^{2l+6m}[v^{\frac{1}{m}}]_{A_p}^{2l+4m}\prod_{s =1}^{l} \|b_s\|_{\mathrm{BMO}}\left\|\frac{ \mathcal{M}_{L(\log L)} (\vec{f})}{v}\right\|_{L^{\frac{1}{m}, \infty}(u v^{\frac{1}{m}})},
$$
where $\tilde{K}_0=C_n\tilde{p}_0\tilde{p}_0^{\prime}2^{\tilde{p}_0-1}([v^{\frac{1}{m}}]_{A_p}^2[u]_{A_1}^{\tilde{p}_0})$ with $\tilde{p}_0=2^{n+3}(p-1)[u]_{A_1}+1.$
\end{remark}
 As a corollary of Theorem \ref{thm1.3}, we can easily obtain the following weak type estimates for iterated commutators with the type of Coifman-Fefferman inequalities.
 \begin{corollary}\label{cor1.2}
Let $I=\{i_1,\ldots,i_l\}=\{1,\ldots,l\}\subseteq \{1,\ldots,m\}.$  If $\vec{b} \in \mathrm{BMO}^l$ and $\mathcal{T}_{\vec{b}}$ satisfies the Hypothesis $\ref{hyp2}$. Let $\vec{w}=\left(w_1, \ldots, w_m\right)$ and $u=\prod_{i=1}^m w_i^{1 / m}$. If $\vec{w} \in A_{\vec{1}},$   then
$$
\left\|\mathcal{T}_{\vec{b}}(\vec{f})\right\|_{L^{\frac{1}{m}, \infty}(u )} \lesssim (2[u]_{A_1})^{2^{n+7}m[u]_{A_1}}\prod_{s =1}^{l} \|b_s\|_{\mathrm{BMO}}\left\| \mathcal{M}_{L(\log L)} (\vec{f})\right\|_{L^{\frac{1}{m}, \infty}(u )}.
$$

\end{corollary}\par
 In order to present the Fefferman-Stein inequalities with arbitrary weights, given a weight $w$ and $0<p<\infty,$ we define a class of weighted BMO spaces $\mathrm{BMO}_{p}(w)$  by
$$
\mathrm{BMO}_{p}(w):=\left\{f \in L_{l o c}^1\left(\mathbb{R}^n\right):\|f\|_{\mathrm{BMO}_{p}(w)}<\infty\right\}
$$
where
$$
\|f\|_{\mathrm{BMO}_{p}(w)}:=\sup _Q \left(\frac{1}{w(Q)} \int_Q\left|f(x)-\langle f \rangle_Q\right|^p w(x)dx\right)^{\frac{1}{p}}.
$$
When $p=1,$ we denote
$
\|f\|_{\mathrm{BMO}(w)}:=\sup _Q \frac{1}{w(Q)} \int_Q\left|f(x)-\langle f \rangle_Q\right| w(x)dx.
$
\begin{theorem}\label{thm1.4}
Let $m\geq 2, I=\{i_1,\ldots,i_l\}=\{1,\ldots,l\}\subseteq \{1,\ldots,m\}, \vec{b}=(b_1,\ldots, b_l).$  Let $1<p_1, \ldots, p_m<\infty$ and $\frac{1}{p}=\frac{1}{p_1}+\cdots+\frac{1}{p_m}.$ Assume that for all weights $\vec{w}=\left(w_1, \ldots, w_m\right), \nu_{\vec{w}}=\prod_{s=1}^m w_s^{p / p_s},$ $b_s \in \mathrm{BMO}_{p_s}(w_s)\cap \mathrm{BMO}$ with $1\leq s\leq l,$ and $\mathcal{T}_{\vec{b}}$ satisfies the Hypothesis $\ref{hyp2}$. If $0<p\leq1,$ then

$$
\left\|\mathcal{T}_{\vec{b}}(\vec{f})\right\|_{L^p\left(\nu_{\vec{w}}\right)} \leq C \|\vec{b}\|_{\mathrm{BMO}}^*\prod_{s=1}^m\left\|f_s\right\|_{L^{p_s}(M w_s)},
$$
where $C$ is independent of $\vec{w}$ and $\vec{b},$ and $$\|\vec{b}\|_{\mathrm{BMO}}^*:=\max_{\vec{\gamma}\in \{1,2\}^l}\{\prod_{s:\gamma_s=1}\|b_s\|_{\mathrm{BMO}_{p_s}(w_s)}\prod_{s:\gamma_s=2}\|b_s\|_{\mathrm{BMO}}\}.$$
\end{theorem}
\begin{remark}
Since the weights $\vec{w}=\left(w_1, \ldots, w_m\right)$ are arbitrary, the conditions $b_s \in \mathrm{BMO}_{p_s}(w_s)\cap \mathrm{BMO} (1\leq s\leq l)$ are required. However, the following two corollaries show that when the weights $\vec{w}$ satisfies some restrictive conditions ( e.g., $w_s \in A_\infty^{\text{weak}}$ with $1\leq s\leq l$ ), the conditions of the Theorem \ref{thm1.4} can be reduced to $b_s \in \mathrm{BMO} (1\leq s\leq l).$
\end{remark}

\begin{corollary}\label{cor1.3}
Let $m\geq 2, I=\{i_1,\ldots,i_l\}=\{1,\ldots,l\}\subseteq \{1,\ldots,m\}, \vec{b}=(b_1,\ldots, b_l).$  Let $0<p\leq1,$ $1<p_1, \ldots, p_m<\infty$ and $\frac{1}{p}=\frac{1}{p_1}+\cdots+\frac{1}{p_m}.$ Assume that $\vec{w}=\left(w_1, \ldots, w_m\right)$, $\nu_{\vec{w}}=\prod_{s=1}^m w_s^{p / p_s},$ $\vec{b}\in \mathrm{BMO}^l$ and $\mathcal{T}_{\vec{b}}$ satisfies the Hypothesis $\ref{hyp2}$. Then it holds that
\begin{enumerate}\item Suppose that $w_s\in A_\infty$ for any $1\leq s\leq m.$ Then 
$$
\left\|\mathcal{T}_{\vec{b}}(\vec{f})\right\|_{L^p\left(\nu_{\vec{w}}\right)} \leq C \prod_{s=1}^l[w_s]_{A_\infty}\prod_{s=1}^l\|b_s\|_{\mathrm{BMO}}\prod_{s=1}^m\left\|f_s\right\|_{L^{p_s}(M w_s)},
$$

\item  Suppose that $w_s\in A_\infty^{\text{weak}}$ for any $1\leq s\leq m.$ Then
\end{enumerate}
\begin{equation}\label{cor1.4.1}
\left\|\mathcal{T}_{\vec{b}}(\vec{f})\right\|_{L^p\left(\nu_{\vec{w}}\right)} \leq C \sum_{\vec{\gamma}\in \{1,2\}^l}\big(\prod_{s:\gamma_s=1}[w_s]_{A_\infty}^{\text{weak}}\big)\prod_{s=1}^l\|b_s\|_{\mathrm{BMO}}
\prod_{s=1}^m\left\|f_s\right\|_{L^{p_s}(M w_s)},
\end{equation}
where the constant $C$ is independent of $\vec{w}$ and $\vec{b}.$
\end{corollary}
\begin{remark}
According to \cite[Example 3.2]{and}, when $w(x)=e^x$, $[w]_{A_\infty}^{\text{weak}}$ is less than one. Thus $\prod_{s=1}^{k}[w_s]_{A_\infty}^{\text{weak}}$  is not monotonically increasing with $0\leq k\leq l$. This is the reason why the term $\sum_{\vec{\gamma}\in \{1,2\}^l}(\prod_{s:\gamma_s=1}[w_s]_{A_\infty}^{\text{weak}})$ appears in  (\ref{cor1.4.1}).
\end{remark}

Finally, we present the weighted modular inequalities for multilinear iterated commutators $\mathcal{T}_{\vec{b}}$, which are completely new even in the unweighted case.
\begin{theorem}\label{thm1.5}
Let $I=\{i_1,\ldots,i_l\}=\{1,\ldots,l\}\subseteq \{1,\ldots,m\}, \vec{b}=(b_1,\ldots, b_l)\in \mathrm{BMO}^l,$ and $\mathcal{T}_{\vec{b}}$ satisfies the Hypothesis $\ref{hyp2}$.  Let~$\phi$ be an~$N$-function with sub-multiplicative property. Then for any $1<r<\infty,$ we have
\begin{enumerate}[(1)]
	\item if $r<i_{\phi}<\infty,$ then there exists a constant~$\alpha$ such that for
any ~$1<q<\frac{i_{\phi}}{r}$ and ~$w\in A_q ,$
	 \begin{equation*}
\begin{aligned}
		\int_{\mathbb{R}^n} \phi\left(\mathcal{T}_{\vec{b}}(\vec{f})(x)\right) w(x) d x \lesssim & [w]_{A_\infty}^{(l+1)(\alpha C_1+1)}\prod_{s=1}^l\|b_s\|_{\mathrm{BMO}}^{1+\alpha C_1}\\
&\times\left(\prod_{i=1}^m\int_{\mathbb{R}^n} \phi^m\left([w]_{A_q}^{\frac{1}{qr}}\left|f_i(x)\right|\right) w(x) d x\right)^{\frac{1}{m}};
\end{aligned}
	\end{equation*}
	\item if $1<i_{\phi}\leq r,$ then there exists  a constant~$\alpha$ such that for
any~$1<q<i_{\phi}$ and ~$w\in A_q ,$
	 \begin{equation*}
\begin{aligned}
		\int_{\mathbb{R}^n} \phi\left(\mathcal{T}_{\vec{b}}(\vec{f})(x)\right) w(x) d x \lesssim & [w]_{A_\infty}^{(l+1)(\alpha C_1+1)+1+mC_1}\prod_{s=1}^l\|b_s\|_{\mathrm{BMO}}^{1+\alpha C_1}\\
&\times\left(\prod_{i=1}^m\int_{\mathbb{R}^n} \phi^m\left([w]_{A_q}^{\frac{2}{q}}\left|f_i(x)\right|\right) w(x) d x\right)^{\frac{1}{m}},
\end{aligned}
	\end{equation*}
where the definitions of the $N$-function, $i_{\phi}$ and $C_1$ are listed in Section \ref{sub4}.
\end{enumerate}
\end{theorem}
 \subsection{Historical background}\
\par

{\bf{Local decay estimates.}}
A local decay estimate is an inequality of the form:
\begin{equation}\label{b2}
\varphi(t):=\frac{1}{|Q|}\big|\left\{x \in Q:\left|T_1 f(x)\right|>t\left|T_2 f(x)\right|\right\}\big|\leq c_1 e^{-c_2t}, \quad t>0,
\end{equation}
where $T_1$ is a linear or sublinear operator and $T_2$ is an appropriate maximal function. Inequality (\ref{b2})  reflects accurately the extent that an operator is locally controlled by cerain maximal operator and provides us with enough information to measure the size of $T_1$ and $T_2.$
More precisely, local decay estimation greatly improves the Coifman-Fefferman inequality as follows:
$$
\left\|T_1 f\right\|_{L^p(w)} \leq c\|T_2 f\|_{L^p(w)}
$$
with $0<p<\infty $ and $w\in A_\infty.$
 In \cite{coi}, it is shown that the above inequality holds for maximal singular integral operator $T_1=T^*$ and Hardy-Littlewood maximal operator $T_2=M$. But its proof depends on the good-$\lambda$ technique, so there is no access to the dependence of the constant $c$ on the weight $w$ or $p.$ In 1993, Buckley \cite{buk} obtained an improved good-$\lambda$ inequality with a local exponential decay in $t,$ 
\begin{equation}\label{b1}
\left|\left\{x \in Q: T^* f(x)>2 \lambda, M f(x) \leq t \lambda\right\}\right| \leq c e^{-c / t}|Q|.
\end{equation}
As an application, the sharp constant dependence on weighted norm of $T^*$ was given in \cite{buk} by using inequality (\ref{b1})\par
The estimate of type  (\ref{b2}) is an improved version of inequality (\ref{b1}) due to Karagulyan \cite{kar}. Unfortunately, it is unknown whether Karagulyan's method could be applied to operators besides maximal singular integral operator $T^*.$ In 2013, Ortiz-Caraballo et al. \cite{ort} used a different approach to obtain estimates in the form of (\ref{b2}), and their approach is generalized enough to obtain local decay estimates for operators such as Calder\'{o}n-Zygmund operators, vector-valued extensions of the maximal function or Calder\'{o}n-Zygmund operators, commutators of singular integrals with BMO functions and higher order commutators. For the most recent results on local decay estimations, we refer to \cite{cao,cao1,wen} and the references therein.\par
{\bf{Mixed weak type estimates.}}
Mixed weak type estimates, also known as Sawyer-type inequalities, mean that for the weight functions $u,v$ and the operator $T,$ the following inequality holds:
\begin{equation}\label{m1}
\left\|\frac{T(f v)}{v}\right\|_{L^{1, \infty}\left(\mathbb{R}^n, u v\right)} \lesssim\|f\|_{L^1\left(\mathbb{R}^n, u v\right)}.
\end{equation}
The study of this type of inequalities  has a long history. In 1977, Muckenhoupt and Wheeden \cite{muc} first introduced a new weak type inequality which takes into account the perturbation of the Hardy-Littlewood maximal operator with $A_1$ weights,
$$
|\{x \in \mathbb{R}: w(x) M f(x)>t\}| \leq c_w \frac{1}{t} \int_{\mathbb{R}}|f| w(x) d x.
$$
It is worth mentioning that $w\in A_1$ is no longer a necessary condition to guarantee the validity of weak endpoint results
 \cite[Section 5]{muc}. \par
Later on, in order to give a new proof of Muckenhoupt's theorem, Sawyer \cite{saw} obtained the following result:
\begin{equation}\label{m2}
u v\left(\left\{x \in \mathbb{R}: \frac{M(f v)(x)}{v(x)}>t\right\}\right) \leq c_{u, v} \frac{1}{t} \int_{\mathbb{R}}|f| u(x) v(x) d x,
\end{equation}
where $u, v \in A_1.$
Sawyer \cite{saw} conjectured that (\ref{m2}) should be true for the Hilbert transform. In 2005, Cruz-Uribe et al. \cite{cru} extended (\ref{m2}) to higher dimensions and actually showed that Sawyer's conjecture also  holds for Calder\'{o}n-Zygmund operators. For $u \in A_1$, and $v \in A_1$ or $u v \in A_{\infty},$ it was demonstrated in \cite{cru} that the following estimate is valid for either the Calder\'{o}n-Zygmund operator or the Hardy-Littlewood maximal operator,
\begin{equation}\label{m3}
u v\left(\left\{x \in \mathbb{R}^n: \frac{|T(f v)(x)|}{v(x)}>t\right\}\right) \leq \frac{C}{t} \int_{\mathbb{R}^n}|f(x)| u(x) v(x) d x, \quad t>0.
\end{equation}
The left side of this inequality has no difference from the standard weak-type estimates except a weight function appearing in the level set of the operator $T$.
However, it would be extraordinarily difficult to deal with the left side of (\ref{m3}). There are two main obstacles. The first one is that the product of $uv$ may become more singular. For example, if one chooses $u(x)=v(x)=|x|^{-1 / 2}$ in $\mathbb{R}$, then, $u,v\in A_1$. But the product $u v$ is not locally integrable. The second drawback is that the structure of such sets in the left side of (\ref{m3}) makes it impossible or very difficult to measure them using classical tools such as Vitali covering lemma or interpolation theorem. But once (\ref{m3}) holds for some operator $T$, one may obtain  immediately a weak-type endpoint estimate for it by taking $u \in A_1$ and $v \equiv 1.$ \par
Recently, the study of mixed weak type estimates have attracted much attention. Among them are the works for multilinear Hardy-Littlewood maximal operators in \cite{li1},  the multilinear bounded oscillation operators in \cite{cao1}, Lorentz spaces extensions in \cite{per6}, the multilinear maximal operators and bilinear sparse operators in homogeneous spaces in  \cite{iba1}.\par
{\bf{Fefferman-Stein inequalities.}}
 For an operator $T$ and any weight $w,$ recall that the Fefferman-Stein type inequalities are the estimates of the form:
\begin{equation}\label{f1}
w\left(\left\{x \in \mathbb{R}^n:|T f(x)|>\lambda\right\}\right) \leq \frac{C}{\lambda} \int_{\mathbb{R}^n}|f(x)| M w(x) d x
\end{equation}
and
\begin{equation}\label{f2}
\int_{\mathbb{R}^n}|Tf(x)|^p w(x) d x \leq C\int_{\mathbb{R}^n}|f(x)|^p M w(x) d x, \qquad \hbox{for $1<p<\infty.$}
\end{equation}
In 1971, Fefferman and Stein \cite{fef} showed that (\ref{f1}) and (\ref{f2}) hold for Hardy-Littlewood maximal operator $T=M.$ This type of inequalities was extended to many operators, and (\ref{f1}) was shown to be true for square functions \cite{cha}, but false for fractional integral operators \cite{car}. \par
It was shown by P\'{e}rez \cite{per7} that (\ref{f1}) holds if $M$ is replaced by iterated operator $M^2$ or even by the operator $M_{L(\log L)^{\varepsilon}}$ with $\varepsilon>0,$ and
(\ref{f2}) is true if $M$ is replaced by $M^{\lfloor p\rfloor+1}$ where $\lfloor p\rfloor$ is the integer part of $p$. However, in 2012, Reguera and Thiele \cite{reg} gave an example to show that the estimate of (\ref{f1}) does not hold for the Hilbert transform (this disproves the so-called Muckenhoupt-Wheeden conjecture). For the recent progress of Fefferman-Stein inequalities, see \cite{hoa,rah}.\par
{\bf{Modular inequalities.}}
The modular inequalities concern the following estimates:
\begin{equation}\label{mo1}
\int_{\mathbb{R}^n} \phi(|T_1 f(x)|) w(x) d x \leq C \int_{\mathbb{R}^n} \phi(|T_2 f(x)|) w(x) d x
\end{equation}
and its corresponding weak version,
$$
\sup _{\lambda>0} \phi(\lambda) w\big(\{x \in \mathbb{R}^n:|T_1 f(x)|>\lambda\}\big) \leq C \sup _{\lambda>0} \phi(\lambda) w\big(\{x \in \mathbb{R}^n:|T_2 f(x)|>\lambda\}\big),
$$
where $T_1, T_2$ are linear or sublinear operators, $w\in A_\infty$ and $\phi \geq 0$ is an increasing function satisfying some very mild condition. \par
The modular inequality comes from the study of endpoint estimates for operators. It is well known that the commutators of the Calder\'{o}n-Zygmund operators $T$ with BMO function $b$ have the following Coifman-Fefferman inequality \cite{per8}
$$
\int_{\mathbb{R}^n}|[b, T] f(x)|^p w(x) d x \leq C\|b\|_{\mathrm{BMO}}^p \int_{\mathbb{R}^n} M^2 f(x)^p w(x) d x
$$
with any $0<p<\infty$ and any $w \in A_{\infty}$. Notice that the commutators of the Calder\'{o}n-Zygmund operators are not weak (1,1) type since the maximal operator used to control the commutator in the above estimate is $M^2$, but there is no weak (1,1) property for $M^2$ (see \cite{kok}), which is replaced by the following weak $L\log L$-type estimate:
$$
\left|\left\{x \in \mathbb{R}^n: M^2 f(x)>\lambda\right\}\right| \leq C \int_{\mathbb{R}^n} \phi\left(\frac{|f(x)|}{\lambda}\right) d x,
$$
where $\phi(t)=t\left(1+\log ^{+} t\right)$. This result, however, is not compatible with classical function spaces. Therefore this type of estimates is called modular inequality (see \cite{kok}), and it gives an appropriate endpoint result for $M^2$ and this type of estimates has good interpolation properties.
Based on this, one may wonder whether such estimates are also available for commutators or not. It was P\'{e}rez  \cite{per2} who considered a weak type of Coifman-Fefferman inequality with the form
\begin{equation}\label{mo2}
\sup _{\lambda>0} \varphi(\lambda) w\left(\{x \in \mathbb{R}^n:|[b, T] f(x)|>\lambda\}\right) \leq C \sup _{\lambda>0} \varphi(\lambda) w\left(\{x \in \mathbb{R}^n: M^2 f(x)>\lambda\}\right),
\end{equation}
where $T$ is Calder\'{o}n-Zygmund operator, $\varphi(\lambda)=\frac{\lambda}{1+\log ^{+} \lambda^{-1}}$ and $w \in A_{\infty}.$  Inequality (\ref{mo2}) is very important in illustrating the endpoint estimates of the commutators since as a consequence, it gives that
$$
\left|\left\{x \in \mathbb{R}^n:|[b, T] f(x)|>\lambda\right\}\right| \leq C_{\|b\|_{\mathrm{BMO}}} \int_{\mathbb{R}^n} \phi\left(\frac{|f(x)|}{\lambda}\right) d x.
$$

Notice that there is a function $\varphi$ on both sides of (\ref{mo2}) which is not homogeneous and hence each side of this inequality is not a norm or quasi-norm. But this type of inequality reflects the properties of the operator itself. Since then, many works have enriched the literature in this area. For example, the extrapolation theorem for modular inequality \cite{cur}, the modular inequalities of strong type for
maximal truncated Calder\'{o}n-Zygmund operators \cite{and}, and for variation operators of singular integrals and their
commutators \cite{tan}.

 \subsection{Structure of the paper}\
\par

The organization of the rest of this article is as follows: In Section \ref{Sect 2} we present some preliminaries,
including the properties of Muckenhoupt weights, weak $A_\infty$ weights, Young functions and Orlicz maximal operators.
Section \ref{Sect 3} contains the proof of local exponential decay estimates (\ref{ie7}) of iterated commutators, both in quantitatively weighted and unweighted versions. In Section \ref{Sect 4}, the proof of Coifman-Fefferman inequality (Theorem \ref{thm1.2}) will be given, which will be used later. Section \ref{Sect 5} is devoted to establishing mixed weak type estimates and the purpose of Section \ref{Sect 5} is to prove Theorem \ref{thm1.4}. The modular inequalities are proved in Section \ref{Sect 7} based on the sparse domination. Finally, some applications, including multilinear $\omega$-Calder\'{o}n-Zygmund operators, multilinear Fourier integral operators and Stein's square functions, will be given in Section \ref{Sect 8}.\par
Throughout this paper, we always use $C$ to denote a positive constant, which is independent of the main parameters, but it may change at each occurrence. Let $L_{\text {loc }}^1(\mathbb{R}^n)$ denote the set of all locally integrable functions on $\mathbb{R}^n,$ and $\mathbb{N}=\{0,1,2,\ldots\}.$ For any set $E$, we use $\chi_E$ to denote its characteristic function. Let $\mathscr{S}\left(\mathbb{R}^n\right)$ denote the collection of all Schwartz functions on $\mathbb{R}^n$, equipped with the classical well-known topology determined by a countable family of norms, and $\mathscr{S}^{\prime}\left(\mathbb{R}^n\right)$ its topological dual. If, for any $a, b \in \mathbb{R}, a \leq C b $ ($a \geq C b,$ respectively), we then denote $a \lesssim b $ where $C$ is independent of $a$ and $b$, $(a \gtrsim b,$ respectively). If $a \lesssim b \lesssim a$, we then denote $a \backsimeq b$.

\section{Preliminary}\label{Sect 2}
We begin by presenting some basic facts for sparse family, modular inequalities, Muckenhoupt weights and Orlicz maximal operators.
\subsection{ Sparse family}\label{sub1}
In this subsection, we will introduce a quite useful tool, dyadic calculus, which is taken from \cite{ler2,ler5}.\par
In the following, let $\mathcal{D}\left(Q\right)$ be the set of all dyadic cubes with respect to cube $Q$. These cubes obtained by repeated subdivision $Q$ and its descendants in $2^n$ cubes with the same side length.
\begin{definition}\label{def2.1}
	A collection, $\mathcal{D}$ of cubes is said to be a dyadic lattice if it satisfies the following properties:
	\begin{enumerate}[(i).]
		\item If $Q \in \mathcal{D}$, then each descendant of $Q$ is in $\mathcal{D}$ as well;
		\item For every cubes $Q_1, Q_2 \in \mathcal{D}$, we can find a common ancestor $Q \in \mathcal{D}$ such that $Q_1, Q_2 \in D(Q)$;
		\item For any compact set $K \subset \mathbb{R}^n$, there exists a cube $Q \in \mathcal{D}$ containing $K$.
	\end{enumerate}
\end{definition}\par
In dyadic calculus, the following Three Lattice Theorem (see \cite[Theorem 3.1]{ler2}) plays an important role, which provides a clearly understanding of the structure for dyadic lattics.
\begin{lemma}[\cite{ler2}]
Given a dyadic lattice $\mathcal{D}$, there exist $3^n$ dyadic lattices $\{\mathcal{D}_j\}_{j=1}^{3^n}$  such that
$$
\{3 Q: Q \in \mathcal{D}\}=\bigcup_{j=1}^{3^n} \mathcal{D}_j
$$
and for each cube $Q \in \mathcal{D}$, there is a cube $R_Q \in \mathcal{D}_{j}$ for some $j$ such that $Q \subseteq R_Q$ and $3 l_Q=l_{R_Q}$.
\end{lemma}
We need to introduce one more  definition.
\begin{definition}
	Given a dyadic lattice $\mathcal{D},$ a subset $\mathcal{S}$ of $ \mathcal{D}$ is said to be a $\eta$-sparse family with $\eta \in(0,1)$ if for every cube $Q \in \mathcal{S},$
	$$|\bigcup_{P \in \mathcal{S}, P \subsetneq Q} P|\leq (1-\eta) \left| Q \right|.$$
\end{definition}
Equivalently, if we define
$E(Q)=Q \backslash \bigcup_{P \in \mathcal{S}, P \subsetneq Q} P,$
then a simple calculation shows that the sets $E(Q)$ are pairwisely disjoint and $|E(Q)| \geq \eta |Q|$.\par
Let $\mathcal{D}$ be a dyadic lattice and $\mathcal{S}\subseteq \mathcal{D}$ be a $\eta$-sparse family, we define the sparse operator $\mathcal{A}_{r, \mathcal{S}}$ with $r>0$ as
$$\mathcal{A}_{r, \mathcal{S}} f(x)= \sum_{Q \in \mathcal{S}}\langle|f|^{r}\rangle_{Q}^{1 / r}\chi_{Q}(x)=\sum_{Q \in \mathcal{S}}\left(\frac{1}{|Q|} \int_{Q}|f(y)|^{r} d y\right)^{\frac{1}{r}} \chi_{Q}(x).$$\par
The following lemma in \cite[Lemma 5.1]{ler1} for the oscillation estimation of sparse families will play a crucial role in proving the local estimation of the commutators.
\begin{lemma}[\cite{ler1}]\label{lem4}
Let $\mathscr{D}$ be a dyadic lattice and let $\mathcal{S} \subset \mathscr{D}$ be a $\gamma$-sparse family. Assume that $b \in L_{\text {loc }}^1(\mathbb{R}^n)$. Then there exists a $\frac{\gamma}{2(1+\gamma)}$-sparse family $\widetilde{\mathcal{S}} \subset \mathscr{D}$ such that $\mathcal{S} \subset \widetilde{\mathcal{S}}$ and for every cube $Q \in \widetilde{\mathcal{S}}$,
$$
\left|b(x)-{\langle b \rangle}_{Q}\right| \leq 2^{n+2} \sum_{R \in \widetilde{\mathcal{S}}, R \subseteq Q} \langle\left|b-{\langle b \rangle}_{R}\right|\rangle_R \chi_R(x)
$$
for a.e. $x \in Q$.
\end{lemma}
\subsection{ The classical weights }\label{sub2}
In this subsection, we will present the relevant definitions of classical Muckenhoupt weights, multiple weights, and the weak $A_\infty$ class, as well as their main properties. We will start with the multilinear maximal function $\mathcal{M}$ defined by
$$
\mathcal{M}(\vec{f})(x):=\sup _{Q \ni x} \prod_{i=1}^m \frac{1}{|Q|} \int_Q\left|f_i(y)\right| d y,
$$
where the supremum is taken over all cubes $Q$ containing $x$.\\
The following multiple weights associated with $\mathcal{M}$ were introduced in \cite{ler4}.
\begin{definition}[\cite{ler4}]\label{def2.2}
Let $\frac{1}{p}=\frac{1}{p_1}+\cdots+\frac{1}{p_m}$ with $1\leq p_1, \ldots, p_m<\infty$, $\vec{w}=\left(w_1, \ldots,w_m\right)$, where each $w_i$ is a nonnegative and locally integrable function on $\mathbb{R}^n$, we say that $\vec{w}$ belongs to multiple weights $A_{\vec{p}}$  if
$$
[\vec{w}]_{A_{\vec{p}}}:=\sup _Q\left(\frac{1}{|Q|} \int_Q \nu_{\vec{w}}(x) d x\right) \prod_{j=1}^m\left(\frac{1}{|Q|} \int_Q w_i^{1-p_i^{\prime}}(x) d x\right)^{p / p_i^{\prime}}<\infty,
$$
where the supremum is taken over all cubes $Q \subset \mathbb{R}^n$ and $\nu_{\vec{w}}=\prod_{j=1}^m w_j^{p / p_j}$.
When $p_i=1$ for some $i,$ $\left(\frac{1}{|Q|} \int_Q w_i{ }^{1-p_i^{\prime}}\right)^{1 / p_i^{\prime}}$ is understood as $\left(\inf _Q w_i\right)^{-1}$.
\end{definition}

If $m=1$, the multiple $A_{\vec{p}}$ weights coincide with the classical Muckenhoupt $A_p$ weights. In the linear case, we say a weight $w$ belongs to the class $A_1$ if there is a constant $C$ such that
$$
\frac{1}{|Q|} \int_Q w(y) d y \leq C \inf _Q w,
$$
and the infimum of these constants $C$ is called the $A_1$ constant of $w.$ \par
The following characterization of multiple weights $A_{\vec{p}}$ is given in \cite[Theorem 3.6]{ler4}.
\begin{lemma}[\cite{ler4}] Let $\vec{w}=\left(w_1, \ldots, w_m\right)$ and $1 \leq p_1, \ldots, p_m<\infty$.
Then $\vec{w} \in A_{\vec{p}}$ if and only if
$$
\left\{\begin{array}{l}
w_j^{1-p_j^{\prime}} \in A_{m p_j^{\prime}}, \quad j=1, \ldots, m, \\
\nu_{\vec{w}} \in A_{m p},
\end{array}\right.
$$
where the condition $w_j^{1-p_j^{\prime}} \in A_{m p_j^{\prime}}$ in the case $p_j=1$ is understood as $w_j^{1 / m} \in A_1$.
\end{lemma}
We now introduce definitions of several other classes of weights which will be used later. For more information about them, see \cite{and,omb,per3}.  Since the $A_p$ classes are increasing with respect to $p$, the $A_{\infty}$ class of weights is defined in a natural way by $$A_{\infty}=\bigcup_{p>1} A_p.$$
A weight $w \in A_{\infty}$ if and only if
$$
[w]_{A_{\infty}}:=\sup _Q \frac{1}{w(Q)} \int_Q M\left(w \chi_Q\right)(x)dx<\infty .
$$
This form of $[w]_{A_{\infty}}$ is known as the Fujii-Wilson $A_{\infty}$ constant and was first introduced in \cite{fuj}.
Next we give a  class of weights that is more general than $A_{\infty}$. A weight $w$ belonging to weak $A_{\infty}$ class means that there exist $0<c,\delta<\infty$ such that for all cubes $Q$ and all measurable subsets $E$ of $Q$:
$$
w(E) \leq c\left(\frac{|E|}{|Q|}\right)^\delta w(2 Q).
$$
This class of weights is derived from \cite{saw1}, but is very interesting in its own way as it appears in many contexts, such as quasiregular mappings theory or the regularity for solutions of elliptic PDE's (see \cite{boj}).\par
Similar to the Fujii-Wilson $A_{\infty}$ constants, the weak $A_{\infty}$ constants can be expressed as follows
$$
[w]_{A_{\infty}}^{\text {weak }}:=\sup _Q \frac{1}{w(2 Q)} \int_Q M\left(w \chi_Q\right)(x)dx.
$$
It is shown in \cite{and} that the constant $2$ in the average could be replaced by any parameter $\kappa>1,$ and the following lemma holds.
\begin{lemma}[\cite{and}]\label{lem4(1)}
Let $w\in A_{\infty}^{\text{weak}},$ then for all $Q$ cubes in $\mathbb{R}^n$ with sides parallel to the axes,
$$\left(\frac{1}{|Q|}\int_Q w^r(x)dx\right)^{\frac{1}{r}}\leq \frac{2}{|2Q|}\int_{2Q}w(x)dx,$$
with $$1 < r \leq 1+\frac{1}{\tau_n[w]_{A_\infty}^{weak}},$$
where $\tau_n$ is a dimensional constant with the property $\tau_n \simeq 2^n.$
\end{lemma}

\subsection{  Young function and Orlicz maximal operators}\label{sub3}
 We need to recall some fundamental facts about Young functions and Orlicz spaces. For more information and a lively exposition about these spaces, we refer the readers to \cite{rao}.\par
Let $\Phi$ be the set of functions $\phi:[0, \infty) \longrightarrow[0, \infty)$ which are non-negative, increasing, $\lim _{t \rightarrow \infty} \phi(t)=\infty $  and $\lim _{t \rightarrow 0} \phi(t)=0 .$ $\phi$ is said to be a Young function If $\phi \in \Phi$ is convex.
Given a Young function $\phi,$ the Orlicz space $ L_{\phi}(\mu)$ with respect to the measure $\mu$ is defined to be the set of measurable functions $f$, such that for some $\lambda>0$,
$$
\int_{\mathbb{R}^n} \phi\left(\frac{|f(x)|}{\lambda}\right) d \mu<\infty.
$$
The Luxemburg norm of $f$ over a cube $Q$ is defined by
$$
\|f\|_{\phi(\mu), Q}:=\inf \left\{\lambda>0: \frac{1}{\mu(Q)} \int_{Q} \phi\left(\frac{|f(x)|}{\lambda}\right) d \mu \leq 1\right\}.
$$
For the sake of convenience, we denote $\|f\|_{\phi(\mu), Q}=\|f\|_{\phi, Q}$ if $\mu$ is the Lebesgue measure and  $\|f\|_{\phi(\mu), Q}=\|f\|_{\phi(w), Q}$ if $\mu=w d x$ is an absolutely continuous measure with respect to the Lebesgue measure.\par
A simple yet important observation in Orlicz space $ L_{\phi}(\mu)$ is that each Young function $\phi$ satisfies the generalized H\"{o}lder inequality:
$$\frac{1}{\mu(Q)} \int_{Q}|f g| d \mu \leq 2\|f\|_{\phi(\mu), Q}\|g\|_{\bar{\phi}(\mu), Q},$$
where $\bar{\phi}(t)=\sup _{s>0}\{s t-\phi(s)\}$ is the complementary function of $\phi$.

Let $\mathcal{D}$ be a dyadic grid and $M_{\phi}^{\mathcal{D}}$ be the dyadic Orlicz maximal operator defined by
$$
M_{\phi}^{\mathcal{D}} f(x):=\sup _{ Q\ni x, Q \in \mathcal{D}}\|f\|_{\phi, Q},
$$	where the supremum is taken over all the dyadic cubes containing $x$. Similarly, we denote the classical  Orlicz maximal operator  by $M_{\phi}$.\par
We will employ the following particular examples of maximal operators several times.

\begin{itemize}
	\item If $\phi(t)=t^{r}$ with $r>1,$ then $M_{\phi}=M_{r}$.
	\item If $\phi(t)=t \log (e+t)^{\alpha}$ with $\alpha>0 ,$ then $\bar{\phi}(t) \simeq e^{t^{1 / \alpha}}-1$ and we denote
	$M_{\phi}=M_{L \log L^{\alpha}}$. Then $M \leq M_{\phi} \lesssim M_{r}$ for all $1<r<\infty.$ Moreover, $M_{\phi} \backsimeq M^{l+1},$ where $\alpha=l \in \mathbb{N}$ and $M^{l+1}$ is $M$ iterated $l+1$ times.
\end{itemize}
We end this subsection by the definition of the multilinear $L(\log L)$-maximal operators
$$
\mathcal{M}_{L(\log L)}(\vec{f})(x):=\sup _{Q \ni x} \prod_{i=1}^m\left\|f_i\right\|_{L(\log L), Q}.
$$
\subsection{  Modular inequality}\label{sub4}
In this subsection, we will collect some  concepts related to Young functions and modular inequalities from \cite{cur}.\par
A function $\phi \in \Phi$ is said to be quasi-convex if there exist a convex function $\widetilde{\phi}$ and $a_{1} \geq 1$ such that
$$\widetilde{\phi}(t) \leq \phi(t) \leq a_{1} \widetilde{\phi}\left(a_{1} t\right), \quad t \geq 0.$$
Given a positive increasing function $\phi,$  we define the lower and upper dilation indices of $\phi,$ respectively, as follows:
$$i_{\phi}=\lim _{t \rightarrow 0^{+}} \frac{\log h_{\phi}(t)}{\log t}=\sup _{0<t<1} \frac{\log h_{\phi}(t)}{\log t}, \quad I_{\phi}=\lim _{t \rightarrow \infty} \frac{\log h_{\phi}(t)}{\log t}=\inf _{1<t<\infty} \frac{\log h_{\phi}(t)}{\log t},$$
where $h_{\phi}(t)=\sup _{s>0} \frac{\phi(s t)}{\phi(s)}, t>0.$\par
Now we turn to the $\Delta_{2}$ condition. A function $\phi \in \Phi$ satisfies the $\Delta_{2}$ condition (we denote $\phi \in \Delta_{2}$) if $\phi$ is doubling, that is, $\phi(2t)\leq C\phi(t)$.
A key fact is that if $\phi$ is quasi-convex, then $i_{\phi} \geq 1$ and that $\phi \in \Delta_{2}$ if and only if $I_{\phi}<\infty.$ Moreover, $\bar{\phi} \in \Delta_{2}$ if and only if $i_{\phi}>1,$ where $\bar{\phi}$ is the complementary function of $\phi$ defined in Section \ref{sub3}.\par
\vspace{0.1cm}
Given a weight $w \in A_{\infty}$ and $\phi \in \Phi$, the modular of $f$ is defined by
$$\rho_{w}^{\phi}(f)=\int_{\mathbb{R}^{n}} \phi(|f(x)|) w(x) d x.$$
The collection
$$
Z_{w}^{\phi}=\left\{f: \rho_{w}^{\phi}(f)<\infty\right\}
$$
is called as a modular space.
A multi(sub)linear operator $T$ is said to satisfy a modular inequality on $Z_{w}^{\phi}$ if there exist constants $c_{i}^{(1)}, c_{i}^{(2)}>0$ with $i=1,\ldots, m,$ such that
$$\rho_{w}^{\phi}(T \vec{f}) \leq \prod _{i=1}^m c_{i}^{(1)} \rho_{w}^{\phi}\left(c_{i}^{(2)} f_i\right).$$\par

\section{ Proofs of Theorems \ref{thm1.1} and Corollary \ref{cor1.1}}\label{Sect 3}
In this section, we will prove Theorem \ref{thm1.1} in the unweighted setting and Corollary \ref{cor1.1} in the weighted case. We begin with the proof of Theorem \ref{thm1.1}.
\begin{proof}[Proof of Theorem $\ref{thm1.1}$]
Let $I=\{1,\ldots,l\}\subseteq \{1,\ldots,m\}$ and fix a cube $Q_0$ such that $\operatorname{supp}\left(f_s\right) \subset Q_0$ for $1 \leq s\leq m.$ By the Hypothesis $\ref{hyp1}$, we can see that there exists a $\eta$-sparse family $\mathcal{F} \subset \mathcal{D}\left(Q_0\right)$ such that for a.e. $x \in Q_0$,

\begin{equation*}
\begin{aligned}
{\left|\mathcal{T}_{\vec{b}} \vec{f}(x)\right|}&= \left|\mathcal{T}_{\vec{b}} (\vec{f}\chi_{3Q_{0}})(x)\right|\\
&\leq C \sum_{Q \in {\mathcal{F}}}\sum_{\vec{\gamma} \in {\{1,2\}^{l}}} \left( \prod_{s=1}^l \mathcal{R}(b_s,f_s,Q,\gamma_s) \right)\prod_{s=l+1}^m {\langle |f_s |\rangle}_{3Q}\chi_{Q}(x),\\
\end{aligned}
\end{equation*}
where
\begin{equation*}
\mathcal{R}(b,f,Q,\gamma)= \begin{cases}|b-{\langle b\rangle}_{3Q}|{\langle |f |\rangle}_{3Q}, &  \text{if}\quad  \gamma=1, \\ {\langle |(b-{\langle b \rangle}_{3Q} )f  |\rangle}_{3Q}, & \text{if}\quad  \gamma=2.\end{cases}
\end{equation*}
Consider the pointwisely estimate of
\begin{equation*}
\begin{aligned}
\mathcal{T}_{\mathcal{F}}^{\vec{\gamma}}(\vec{f},\vec{b})(x):=\sum_{Q \in {\mathcal{F}}}\prod_{s =1}^{l}\mathcal{R}(b_s,f_s,Q,\gamma_s)\prod_{s=l+1}^m {\langle |f_s |\rangle}_{3Q}\chi_{Q}(x).
\end{aligned}
\end{equation*}
Without loss of generality, we may assume that $\vec{\gamma}=(\overbrace{1, \ldots ,1}^{l_1}, \overbrace{2, \ldots ,2}^{l-l_1})$  and write
\begin{equation*}
\begin{aligned}
\mathcal{T}_{\mathcal{F}}^{\vec{\gamma}}(\vec{f},\vec{b})(x)=&\sum_{Q \in {\mathcal{F}}}\prod_{s =1}^{l_1}\left|b_s(x)-{\langle b_s \rangle}_{3Q} \right|  {\langle |f_s|  \rangle}_{3Q}   \prod_{s =l_1+1}^{l}      {\langle |(b_s-{\langle b_s \rangle}_{3Q} )f_s  |\rangle}_{3Q}\\
&\times \prod_{s =l+1}^{m} {\langle |f_s|\rangle}_{3Q} \chi_{Q}(x).
\end{aligned}
\end{equation*}\par
First, we observe that
\begin{equation}\label{ie2}
|b_s-{\langle b_s \rangle}_{3Q} | \leq |b_s-{\langle b_s \rangle}_{Q}|+C_n \|b_s\|_{\mathrm{BMO}},  \quad 1\leq s \leq l_1.   \\
\end{equation}
By applying the Lemma \ref{lem4} to $b_1,$ we know that there exists a $\frac{\gamma}{2(1+\gamma)}$-sparse family $\tilde{{{\mathcal{S}}_1}} \subset \mathscr{D}(Q_0)$ such that $\mathcal{F} \subset \tilde{{{\mathcal{S}}_1}}$ and
$$
\left|b_1(x)-{\langle b_1 \rangle}_{Q} \right| \leq 2^{n+2} \sum_{R \in\tilde{{{\mathcal{S}}_1}} , R\subseteq Q} {\langle \left|b_1-{\langle b_1\rangle}_R\right|\rangle}_R \chi_R (x)
$$
for a.e. $x \in Q$.\par
Note that, if ${\mathcal{S}}_1$ is a $\eta_1$-sparse family and ${\mathcal{S}}_2$ is a $\eta_2$-sparse family, then
${\mathcal{S}}_1\bigcup {\mathcal{S}}_2$ is a $\frac{\eta_1 \eta_2}{\eta_1+\eta_2}$-sparse family. In fact, by the fact that ${\mathcal{S}}_i$ is $\frac{1}{\eta_i}$-Carleson ($i=1,2$), (see \cite[p. 22]{ler2}), for any $Q\in \mathscr{D}(Q_0),$ it holds that
\begin{equation}\label{ie34}
\sum_{P\in \mathcal{S}_1 \bigcup\mathcal{S}_2,P\subseteq Q} |P| \leq \sum_{P\in \mathcal{S}_1,P\subseteq Q}|P| +\sum_{P\in \mathcal{S}_2,P\subseteq Q}|P|\leq \frac{\eta_1+\eta_2}{\eta_1 \eta_2}|Q|,
\end{equation}
which implies that ${\mathcal{S}}_1\bigcup {\mathcal{S}}_2$ is a $\frac{\eta_1 \eta_2}{\eta_1+\eta_2}$-sparse family.\par
For $b_2$, there exists a $\frac{\gamma}{2(1+\gamma)}$-sparse family $\tilde{{{\mathcal{S}}_2}} \subset \mathscr{D}(Q_0)$ such that
\begin{equation*}
\left|b_2(x)-{\langle b_2 \rangle}_{Q} \right| \leq 2^{n+2} \sum_{R \in\tilde{{\mathcal{S}_1}}\bigcup \tilde{\mathcal{S}_2} , R\subseteq Q} {\langle \left|b_2-{\langle b_2\rangle}_R\right|\rangle}_R \chi_R (x).
\end{equation*}
This estimate is also valid for $b_1$ and $\tilde{{\mathcal{S}_1}}\bigcup \tilde{\mathcal{S}_2}$ is a $\frac{\eta }{4(1+\eta)}$-sparse family. It is also convenient to denote $\tilde{{\mathcal{S}_1}}\bigcup \tilde{\mathcal{S}_2}$ by $\tilde{\mathcal{S}_2}.$ We iterate this procedure $l$ times to obtain a $\frac{\eta }{2^l(1+\eta)}$-sparse family $\tilde{{\mathcal{S}_l}}=:\tilde{{\mathcal{S}}}$ satisfying the following property
\begin{equation}\label{ie1}
\left|b_t(x)-{\langle b_t \rangle}_{Q} \right| \leq 2^{n+2} \sum_{R \in\tilde{{\mathcal{S}}} , R\subseteq Q} {\langle \left|b_t-{\langle b_t\rangle}_R\right|\rangle}_R \chi_R (x),\quad  1\leq t \leq l.
\end{equation}
This indicates that there exists a sparse family $\tilde{{{\mathcal{S}}}} \subset \mathscr{D}(Q_0)$ such that for any $1\leq s \leq l_1$ and $Q\in \mathcal{F}\subseteq\tilde{{{\mathcal{S}}}},$
\begin{equation*}
|b_s(x)-{\langle b_s \rangle}_{Q} | \leq 2^{n+2}\|b_s\|_{\mathrm{BMO}}\sum_{R \in\tilde{{{\mathcal{S}}}} , R\subseteq Q} \chi_{R}(x),
\end{equation*}
which, together with (\ref{ie2}), gives that
\begin{equation}\label{ie3}
|b_s-{\langle b_s \rangle}_{3Q} | \leq C_n\|b_s\|_{\mathrm{BMO}}( 1+\sum_{R \in\tilde{{{\mathcal{S}}}} , R\subseteq Q} \chi_{R}(x)), \quad 1\leq s\leq l_1.
\end{equation}\par
We now turn to ${\langle|(b_s-{\langle b_s \rangle}_{3 Q}) f_s |  \rangle}_{3Q}$ with $l_1+1\leq s\leq l$, there are two different ways to deal with it. \\
Method one: By the generalized H\"{o}lder's inequality, we have
\begin{equation*}
\begin{aligned}
{\langle|(b_s-{\langle b_s \rangle}_{3 Q}) f_s |  \rangle}_{3Q} &\leq 2\left\|b_s-{\langle b_s \rangle}_{3Q}\right\|_{\exp L,3Q} \|f_s\|_{L(\log L),3Q}\\
&\lesssim \|b_s\|_{\mathrm{BMO}} \|f_s\|_{L(\log L),3Q}.
\end{aligned}
\end{equation*}
This fact together with (\ref{ie3}) easily yields
\begin{equation*}
\begin{aligned}
\mathcal{T}_{\mathcal{F}}^{\vec{\gamma}}(\vec{b},\vec{f})(x) &\lesssim \sum_{Q \in {\mathcal{F}}}\prod_{s =1}^{l_1} \left( \|b_s\|_{\mathrm{BMO}}\left( 1+\sum_{R \in\tilde{{{\mathcal{S}}}} , R\subseteq Q} \chi_{R}(x)\right){\langle|f_s|\rangle}_{3Q} \right) \\
&\quad \quad\times\prod_{s =l_1+1}^{l}\|b_s\|_{\mathrm{BMO}}\|f_s\|_{L(\log L),3Q}\prod_{s =l+1}^{m} {\langle |f_s|\rangle}_{3Q} \chi_{Q}(x)  \\
&\leq \left( 1+\sum_{R \in\tilde{{{\mathcal{S}}}} , R\subseteq Q_0} \chi_{R}(x)\right)^{l_1}\prod_{s =1}^{l} \|b_s\|_{\mathrm{BMO}}\sum_{Q \in {\mathcal{F}}} \prod_{s =1}^{l} \|f_s\|_{L(\log L),3Q} \\
&\quad \times \prod_{s =l+1}^{m}{\langle |f_s|\rangle}_{3Q} \chi_{Q}(x) \\
&\lesssim  \prod_{s =1}^{l} \|b_s\|_{\mathrm{BMO}}\sum_{Q \in {\mathcal{F}},Q\subseteq Q_0}\prod_{s =1}^{l} \|f_s\|_{L(\log L),3Q}\prod_{s =l+1}^{m}{\langle |f_s|\rangle}_{3Q} \chi_{Q}(x)\\
&\quad+ \prod_{s =1}^{l} \|b_s\|_{\mathrm{BMO}} \left( \sum_{Q \in {\mathcal{F}},Q\subseteq Q_0}\prod_{s =1}^{l} \|f_s\|_{L(\log L),3Q}\prod_{s =l+1}^{m}{\langle |f_s|\rangle}_{3Q}\chi_{Q}(x)\right) \\
&\quad  \times \left(\sum_{R \in \tilde{{\mathcal{S}}},R\subseteq Q_0} \chi_R(x)\right)^{l_1}.\\
\end{aligned}
\end{equation*}
By the definition of $M_{L(\log L)}^{(1,l)},$ we have
\begin{equation}\label{ie4}
\begin{aligned}
\mathcal{T}_{\mathcal{F}}^{\vec{\gamma}}(\vec{b},\vec{f})(x) &\lesssim \prod_{s =1}^{l} \|b_s\|_{\mathrm{BMO}}\mathcal{M}_{L(\log L)}^{(1,l)}(\vec{f})(x)\sum_{Q \in {\mathcal{F}},Q\subseteq Q_0}\chi_Q(x)\\
&\quad +\prod_{s =1}^{l} \|b_s\|_{\mathrm{BMO}}\mathcal{M}_{L(\log L)}^{(1,l)}(\vec{f})(x)\left(\sum_{Q \in {\tilde{{\mathcal{S}}}},Q\subseteq Q_0}\chi_Q(x)\right)^{l_1+1} \\
&\lesssim \prod_{s =1}^{l} \|b_s\|_{\mathrm{BMO}}\mathcal{M}_{L(\log L)}^{(1,l)}(\vec{f})(x)\left(\sum_{Q \in {\tilde{{\mathcal{S}}}},Q\subseteq Q_0}\chi_Q(x)\right)^{l_1+1}.
\end{aligned}
\end{equation}

Method two: For each $l_1+1\leq s\leq l$, using (\ref{ie1}) again, we obtain
\begin{equation*}
\begin{aligned}
{\langle|(b_s-{\langle b_s \rangle}_{3Q}) f_s |  \rangle}_{3Q} &\leq \frac{1}{|3Q|}\int_{3Q}\left|b_s(x)-{\langle b_s \rangle}_{Q}\right| |f_s(x) |  dx\\
& \quad +|{\langle b_s \rangle}_{Q}-{\langle b_s \rangle}_{3Q}| \frac{1}{|3Q|}\int_{3Q}|f_s(x)|dx\\
&\lesssim \|b_s\|_{\mathrm{BMO}}\frac{1}{|3Q|}\int_{3Q}\sum_{R \in {\mathcal{\tilde{S}}},R\subseteq Q}\chi_R(x)|f_s(x)|dx+\|b_s\|_{\mathrm{BMO}}{\langle|f_s | \rangle}_{3Q}  \\
&=\|b_s\|_{\mathrm{BMO}}\frac{1}{|3Q|}\Big( \sum_{R \in {\mathcal{\tilde{S}}},R\subseteq Q}\int_{R}|f_s(x) |dx+\int_{3Q}|f_s(x)|dx\Big).  \\
\end{aligned}
\end{equation*}

This estimate yields that
\begin{equation*}
\begin{aligned}
{\langle|(b_s-{\langle b_s \rangle}_{3Q}) f_s |  \rangle}_{3Q} &\lesssim \frac{2\|b_s\|_{\mathrm{BMO}}}{|3Q|} \sum_{R \in {\mathcal{\tilde{S}}},R\in Q}\int_{3R}|f_s(x) |dx \\
&=\frac{2\|b_s\|_{\mathrm{BMO}}}{|3Q|} \sum_{R \in {\mathcal{\tilde{S}}},R\subseteq Q}\frac{3^n}{|3R|}\int_{3R}|f_s(x) |dx \int_{Q}\chi_R(x) dx  \\
&=\frac{2\|b_s\|_{\mathrm{BMO}}}{|Q|}\int_{Q}\sum_{R \in {\mathcal{\tilde{S}}},R\subseteq Q}{\langle |f_s| \rangle}_{3R}\chi_{R}(x)dx.
\end{aligned}
\end{equation*}
Let $\mathcal{S}^*=\{3R:R\in \mathcal{\tilde{S}}\},$ then $\mathcal{S}^*$ is also a sparse family and
\begin{equation*}
\sum_{R \in {\mathcal{\tilde{S}}},R\subseteq Q}{\langle |f_s|  \rangle}_{3R}\chi_R(x)\leq \sum_{Q \in {\mathcal{S}^*}}{\langle |f_s|  \rangle}_{Q}\chi_Q(x)=\mathcal{A}_{{S}^*}(f_s)(x).
\end{equation*}Therefore
\begin{equation*}
{\langle|(b_s-{\langle b_s \rangle}_{3Q}) f_s |  \rangle}_{3Q}\lesssim \|b_s\|_{\mathrm{BMO}}\frac{1}{|Q|}\int_{Q}\mathcal{A}_{{\mathcal{S}}^*}(f_s)(x)dx, \\
\end{equation*}
which together with (\ref{ie3}) implies that
\begin{equation}\label{ie5}
\begin{aligned}
\mathcal{T}_{\mathcal{F}}^{\vec{\gamma}}(\vec{f},\vec{b})(x) &\lesssim  \left( 1+\sum_{R \in\tilde{{{\mathcal{S}}}} , R\subseteq Q_0} \chi_{R}(x)\right)^{l_1}\prod_{s =1}^{l} \|b_s\|_{\mathrm{BMO}}\\
&\quad \times \sum_{Q \in {\mathcal{F}}} \prod_{s =1}^{l_1} {\langle |f_s|\rangle}_{3Q} \prod_{s =l+1}^{m}{\langle |f_s|\rangle}_{3Q} \prod_{s =l_1+1}^{m}{\langle \mathcal{A}_{{{\mathcal{S}}}^*}f_s\rangle}_{3Q}\chi_Q(x)\\
& \leq \prod_{s =1}^{l} \|b_s\|_{\mathrm{BMO}} \left( 1+\sum_{R \in\tilde{{{\mathcal{S}}}} , R\subseteq Q_0} \chi_{R}(x)\right)^{l_1}\mathcal{M}(\vec{f_0})(x)\sum_{Q \in {\mathcal{F}}} \chi_Q(x),\\
\end{aligned}
\end{equation}
where $\vec{f_0}:=(\mathcal{A}_{{\mathcal{S}}^*}f_{1}, \cdots,\mathcal{A}_{{{\mathcal{S}}}^*}(f_{l}),f_{l+1},\cdots,f_m)$\par
Combining all the estimates obtained in (\ref{ie4}) and (\ref{ie5}), it yields that

\begin{equation}
\begin{aligned}
\mathcal{T}_{\mathcal{F}}^{\vec{\gamma}}(\vec{f},\vec{b})(x) &\leq C \prod_{s =1}^{l} \|b_s\|_{\mathrm{BMO}}   \min\left\{\mathcal{M}_{L(\log L)}^{(1,l)}(\vec{f}),\mathcal{M}(\vec{f_0})\right\}\left(\sum_{Q \in\tilde{{{\mathcal{S}}}} , Q\subseteq Q_0}\chi_Q(x)\right)^{l_1+1}. \\
\end{aligned}
\end{equation}
Recall that in \cite[Theorem 2]{ort}, it was shown that
\begin{equation}\label{ie35}
\left|\left\{x \in Q: \sum_{Q^{\prime} \in \mathcal{S}, Q^{\prime} \subseteq Q} \chi_{Q^{\prime}}(x)>t\right\}\right| \leq c e^{-\alpha t}|Q|, \quad \forall Q \ \text{and} \ t>0.
\end{equation}
Keeping this significant observation in mind, then we have
\begin{equation*}
\begin{aligned}
& \left| \left\{x \in Q_0:\mathcal{T}_{\mathcal{F}}^{\vec{\gamma}}(\vec{f},\vec{b})(x)>t \min\left\{\mathcal{M}_{L(\log L)}^{(1,l)}(\vec{f})(x),\mathcal{M}(\vec{f_0})(x)\right\}\right\} \right|   \\
&\hspace{4cm} \leq \left| \left\{x \in Q_0:\sum_{Q \in{\tilde{{\mathcal{S}}}} , Q\subseteq Q_0}\chi_Q(x)>(\frac{t}{C \prod_{s =1}^{l} \|b_s\|_{\mathrm{BMO}}})^{\frac{1}{l+1}}\right\} \right|   \\
&\hspace{4cm}\leq C e^{-\alpha\left(\frac{t}{ \prod_{s =1}^{l} \|b_s\|_{\mathrm{BMO}}}\right)^{\frac{1}{l+1}}}|Q_0|.
\end{aligned}
\end{equation*}
It then follows that
\begin{equation*}
\begin{aligned}
&\left| \left\{x \in Q_0:\left|\mathcal{T}_{\vec{b}}(\vec{f})(x)\right|>t \min\left\{\mathcal{M}_{L(\log L)}^{(1,l)}(\vec{f})(x),\mathcal{M}(\vec{f_0})(x)\right\}\right\} \right|  \\
 &\quad\leq \sum_{\vec{\gamma}\in (1,2)^l}\left| \left\{x \in Q_0:\mathcal{T}_{\mathcal{F}}^{\vec{\gamma}}(\vec{f},\vec{b})(x)> \frac{t}{C2^l}\min\left\{\mathcal{M}_{L(\log L)}^{(1,l)}(\vec{f})(x),\mathcal{M}(\vec{f_0})(x)\right\}\right\} \right|   \\
 &\quad\leq C_1   e^{\frac{-\alpha_1 t^{1/(l+1)}}{ \left(\prod_{s =1}^{l} \|b_s\|_{\mathrm{BMO}}\right)^{1/(l+1)}}}|Q_0|.
\end{aligned}
\end{equation*}
This finishes the proof of (\ref{ie7}).\par
Finally, we need to prove that the exponent of the result in Theorem \ref{thm1.1} is sharp. To see this, let $m=1, I=\{1\}$ and $T$ be an $\omega$-Calder\'{o}n-Zygmund operator with $\omega$ satisfying the Dini condition $[\omega]_{\text {Dini }}=\int_0^1 \omega(t) \frac{d t}{t}<\infty$. Note that for any $b\in \mathrm{BMO},$ $T_b$ satisfies Hypothesis $\ref{hyp1}$ (see \cite[Theorem 1.1]{ler1}) and $M_{L\log L}f\simeq M^2f$ where $M^2$ denotes the composition of Hardy-Littlewood maximal operator $M.$ \\
Applying Theorem \ref{thm1.1}  to $T_b,$  we have

\begin{equation}\label{ie8}
\begin{aligned}
&\left| \left\{x \in Q_0:\left|T_{b}(f)(x)\right|>t M^2f(x)\right\} \right| \leq C_1  e^{-\sqrt{\frac{\alpha_1 t}{\|b\|_{\mathrm{BMO}}}}}|Q_0|.
\end{aligned}
\end{equation}
In particular, if $n=1, Q_0=(0,1)$, and $f(x)=\chi_{(0,1)}(x),$ then for any $x\in Q_0,$ $Mf(x)=1$ which implies $M^2f(x)=1.$ Let $b=\log |x|\in \mathrm{BMO}$ and $T=H$ (the Hilbert transform). Then the following estimate was proved in
\cite[p. 6]{per1}:
$$\left|\left\{x \in Q_0:\left|T_b\left(\chi_{Q_0}\right)(x)\right|>t\right\}\right| \geq e^{-\sqrt{c_0 t}},$$
which holds for some absolute constant $c_0.$ \par
Comparing this result with (\ref{ie8}), we know that the exponent in local decay estimate is sharp and this completes the proof of Theorem \ref{thm1.1}.
\end{proof}\par
We now turn our attention to the proof of Corollary $\ref{cor1.1}$.
\begin{proof}[Proof of Corollary $\ref{cor1.1}$]
Set
$$E=\left\{x\in Q_0: |\mathcal{T}_{\vec{b}}(\vec{f})(x)|>t\mathcal{M}_{L(\log L)}(\vec{f})(x)\right\}\subseteq Q_0.$$
Then Theorem \ref{thm1.1} gives that
$$|E|\leq C_1  e^{-\alpha_1\left(\frac{t}{ \prod_{s =1}^{l} \|b_s\|_{\mathrm{BMO}}}\right)^{\frac{1}{l+1}}}|Q_0|.$$ \\
It suffices to show that for each $w\in A_{\infty}^{\text{weak}},$ there exist constants $C_0, \varepsilon_0$ which depend on
$w$ such that for every $A\subseteq Q,$
\begin{equation}\label{ie08}
\frac{w(A)}{w(2Q)}\leq C_0 \left(\frac{|A|}{|Q|} \right)^{\varepsilon_0}.
\end{equation}
In view of this, $w(E)$ automatically satisfies
$$w(E) \leq C_1^{\varepsilon_0} C_0  e^{-\alpha_1\varepsilon_0\left(\frac{t}{ \prod_{s =1}^{l} \|b_s\|_{\mathrm{BMO}}}\right)^{\frac{1}{l+1}}}w(2Q_0).$$\par
In order to obtain the quantitative weighted estimate (\ref{ie6}), it remains to prove (\ref{ie08}) and determine the dependence of $C_0, \varepsilon_0$ and $w.$
Let $r_w=1+\frac{1}{\tau_n[w]_{A_\infty}^{weak}},$ then $ r_w^{\prime}=\frac{r_w}{r_w-1}=\tau_n[w]_{A_\infty}^{weak}+1$ and Lemma \ref{lem4(1)} imply that
$$\left(\frac{1}{|Q|}\int_Q w^{r_w}(x)dx\right)^{\frac{1}{r_w}}\leq \frac{2}{|2Q|}\int_{2Q}w(x)dx.$$
 H\"{o}lder's inequality with exponents $r_w$ further gives
\begin{equation*}
\begin{aligned}
w(E) &\leq \left( \int_E w^{r_w}(x)dx\right)^{\frac{1}{r_w}}|E|^{\frac{1}{r_w^{\prime}}} \leq \frac{2}{|2Q_0|}\int_{2Q_0}w(x)dx |Q_0|^{\frac{1}{r_w}}|E|^{\frac{1}{r_w^{\prime}}}, \\
\end{aligned}
\end{equation*}
where we have used the fact that $E\subseteq Q_0.$\\
Therefore,
\begin{equation*}
\begin{aligned}
\frac{w(E)}{w(2Q_0)}\leq 2^{1-n} \left(\frac{|E|}{|Q_0|} \right)^{\frac{1}{r_w^{\prime}}}
\leq 2^{1-n} \left(\frac{|E|}{|Q_0|} \right)^{\frac{1}{\tau_n[w]_{A_\infty}^{weak}+1}}.
\end{aligned}
\end{equation*}
 Picking $C_0=2^{1-n}$ and $\varepsilon_0=\frac{1}{\tau_n[w]_{A_\infty}^{weak}+1},$ we have now proved that
 \begin{equation*}
\begin{aligned}
w\left( \left\{x \in Q_0:|\mathcal{T}_{\vec{b}}(\vec{f})(x)|>t\mathcal{M}_{L(\log L)}\vec{f}(x)\right\}\right) \leq C_2 e^{-{\frac{\alpha_2c_{*}}{[w]_{A_\infty}^{weak}+1}}t^{\frac{1}{l+1}}} w(2Q_0),
\end{aligned}
\end{equation*}
where $c_{*}^{-l-1}=\prod_{s =1}^{l} \|b_s\|_{\mathrm{BMO}}.$
This finishes the proof of the Corollary $\ref{cor1.1}.$
\end{proof}
\par
Using Theorem \ref{thm1.1} and Corollary \ref{cor1.1}, we can directly obtain the following result.

\begin{corollary}\label{cor1.1(1)}
Let $\mathcal{T}_{\vec{b}}$ be defined as in Theorem \ref{thm1.1}. If $w\in A_{\infty}$, then
\begin{equation*}
\begin{aligned}
&w\left(\left\{x \in Q_0:\left|\mathcal{T}_{\vec{b}}(\vec{f})(x)\right|>t \mathcal{M}_{L(\log L)}(\vec{f})(x)\right\}\right) \\
&\hspace{5cm} \qquad\leq c_2 e^{-{\frac{\alpha_2}{[w]_{A_\infty}}}\left({\frac{t}{\prod_{s =1}^{l} \|b_s\|_{\mathrm{BMO}}}}\right)^{\frac{1}{l+1}}} w(Q_0), \quad t>0.
\end{aligned}
\end{equation*}

\end{corollary}
\begin{proof}[Proof of Corollary $\ref{cor1.1(1)}$]
Following the definitions in the proof of Corollary \ref{cor1.1}, we have
\begin{equation}\label{iecor1.1}
|E|\leq C_1  e^{-\alpha_1c_{*}t^{\frac{1}{l+1}}}|Q_0|
\end{equation}
with $c_{*}^{-l-1}=\prod_{s =1}^{l} \|b_s\|_{\mathrm{BMO}}.$
Using the estimate
\begin{equation*}
\frac{w(E)}{w(Q_0)}\leq 2\left(\frac{|E|}{|Q_0|} \right)^{\frac{1}{c_n[w]_{A_\infty}}}
\end{equation*}
proved in Lemma 4.6 in \cite{iba} for every $w\in A_{\infty}$ and (\ref{iecor1.1}), it is easy to verfy that Corollary $\ref{cor1.1(1)}$ is valid.
\end{proof}
\begin{remark}
Without using Lemma 4.6 in \cite{iba}, Corollary $\ref{cor1.1(1)}$ may also be proved via the doubling property of $w\in A_\infty$. But the constant $c_2$ would depend on the doubling constant of the measure $wdx$ and thus on $[w]_{A_\infty}$, which yields the fact that $c_2$ has a worse dependence on $[w]_{A_\infty}$.
\end{remark}

\section{ Proofs of Theorems \ref{thm1.2}}\label{Sect 4}

In this section, we aim to establish the Coifman-Fefferman inequality for $\mathcal{T}_{\vec{b}}.$ Before doing it, we present a generalized H\"{o}lder's inequality of multilinear version, which is a generalization of \cite[Lemma 1]{per9} under the general measure.

\begin{lemma}\label{lem2}
Let $\Phi_0, \Phi_1, \Phi_2, \ldots, \Phi_m$ be Young functions. If

$$
\Phi_1^{-1}(t) \Phi_2^{-1}(t) \cdots \Phi_m^{-1}(t) \leq D \Phi_0^{-1}(t),
$$
then for all functions $f_1, \ldots, f_m$ and all cubes $Q$ we have that
\begin{equation}\label{1e22}
\left\|f_1 f_2 \cdots f_m\right\|_{\Phi_0(\mu), Q} \leq m D\left\|f_1\right\|_{\Phi_1(\mu), Q}\left\|f_2\right\|_{\Phi_2(\mu), Q} \cdots\left\|f_m\right\|_{\Phi_m(\mu), Q} .
\end{equation}
In particular, for any weight $w$  and $s_1, \ldots, s_m \geq 1.$ Let $\frac{1}{s}=\sum_{i=1}^m \frac{1}{s_i}$. Then we have
\begin{equation}\label{1e23}
\frac{1}{w(Q)} \int_Q\left|f_1(x) \cdots f_m(x) g(x)\right|w(x)dx \leq 2^{\frac{1}{s}}(1+\frac{1}{s})^{\frac{1}{s}}\prod_{i=1}^m\left\|f_i\right\|_{\exp L^{s_i}(w),Q} \|g\|_{L(\log L)^{\frac{1}{s}}(w),Q}.
\end{equation}
\end{lemma}

\begin{proof}
We first prove that if $\Phi_0, \ldots, \Phi_k$ are continuous, nonnegative, strictly increasing functions on $[0, \infty)$ with $\Phi_i(0)=0$ and $\lim _{t \rightarrow \infty} \Phi_i(t)=\infty$ $(0\leq i\leq m)$ such that
$$
\Phi_1^{-1}(t) \Phi_2^{-1}(t) \cdots \Phi_k^{-1}(t) \leq \Phi_0^{-1}(t), \quad t \geq 0,
$$
then for all $0 \leq x_1, x_2, \ldots, x_k<\infty$
$$
\Phi_0\left(x_1 x_2 \cdots x_k\right) \leq \Phi_1\left(x_1\right)+\Phi_2\left(x_2\right)+\cdots+\Phi_k\left(x_k\right) .
$$
To see this, fix any $\vec{x}=\left(x_1, \ldots, x_m\right)\in (\mathbb{R}^+)^m$  and let $t_0=\Phi_1\left(x_1\right)+\Phi_2\left(x_2\right)+\cdots+\Phi_m\left(x_m\right)$. Then the condition in Lemma \ref{lem2} gives
$$
\Phi_0\left(\frac{\Phi_1^{-1}\left(t_0\right) \Phi_2^{-1}\left(t_0\right) \cdots \Phi_m^{-1}\left(t_0\right)}{D}\right) \leq \Phi_0\left(\frac{D\Phi_0^{-1}(t_0)}{D} \right)=t_0.
$$
By the nonnegativity of $\Phi_i$, for any $i\in \{1,\ldots,m\},$ we have $t_0\geq \Phi_i(x_i).$ Then

$$
\Phi_i^{-1}\left(t_0\right) \geq \Phi_i^{-1}\left(\Phi_i\left(x_i\right)\right)=x_i,
$$
which indicates that
\begin{equation}\label{1e023}
\Phi_0\left(\frac{x_1 x_2 \cdots x_m}{D}\right) \leq t_0=\Phi_1\left(x_1\right)+\Phi_2\left(x_2\right)+\cdots+\Phi_m\left(x_m\right).
\end{equation}\par
 Now consider the proof of (\ref{1e22}). By using the convexity of $\Phi_0$, for any $\lambda \in (0,1),$ it holds that
$$\Phi_0(\lambda t_1+(1-\lambda)t_2)\leq \lambda \Phi_0(t_1)+(1-\lambda)\Phi_0(t_2).$$
Let $t_2 = 0$, then $\Phi_0(\lambda t)\leq \lambda \Phi_0(t) (t>0)$ since $\Phi_0(0)=0$.
Recalling the definition of $\|f\|_{A(\mu),Q}$, we get
$$
\|f\|_{A(\mu), Q} \leq 1 \Leftrightarrow \frac{1}{\mu(Q)} \int_Q A(|f(x)|) d \mu(x) \leq 1 .
$$
Using this fact and by (\ref{1e023}), for any $t_i>\left\|f_i\right\|_{\Phi_i, Q}$ with $1\leq i\leq m,$ we obtain
$$
\begin{aligned}
\frac{1}{\mu(Q)} \int_Q \Phi_0\left(\frac{\left|f_1 \cdots f_m\right|}{m D t_1 \cdots t_m}\right)d\mu & \leq \frac{1}{m\mu(Q)} \int_Q \Phi_0\left(\frac{\left|f_1 \cdots f_m\right|}{D t_1 \cdots t_m}\right)d\mu \\
& \leq \frac{1}{m}\sum_{i=1}^m\frac{1}{\mu(Q)} \int_Q \Phi_i\left(\frac{\left|f_i(x)\right|}{t_i}\right)d\mu(x)\\
& \leq1 .
\end{aligned}
$$
This inequality implies
$$
\left\|f_1 \cdots f_m\right\|_{\Phi_0(\mu), Q} \leq m D t_1 \cdots t_m.
$$
and it is enough to take the infimum on each $t_i$ to finish the proof of the (\ref{1e22}).\par
Finally, we give the proof of (\ref{1e23}). For $x\geq 0, t>0,$ we denote $\varphi_t(x)=e^{x^t}-1,$ and $\Phi_t(x)=x(\log (e+x))^t.$ It is easy to see that $\varphi_t^{-1}(x)=(\log (x+1))^{\frac{1}{t}}.$
In order to prove (\ref{1e23}), by (\ref{1e22}) with $d\mu =wdx$, it suffices to show that
\begin{equation}\label{ie024}
\varphi_{s_1}^{-1}(x) \varphi_{s_2}^{-1}(x)\cdots \varphi_{s_m}^{-1}(x) \Phi_{1/s}^{-1}(x) \leq Dx.
\end{equation}
First, we claim that
$$
 \Phi_{t}^{-1}(x)\simeq  \frac{x}{(\log (e+x))^{t}}.
$$
In fact, we only need to show $ x\simeq  \frac{\Phi_{t}(x)}{(\log (e+\Phi_{t}(x)))^{t}}.$
Note that $\Phi_{t}(x)\geq x,$ then
$$
\frac{\Phi_{t}(x)}{(\log (e+\Phi_{t}(x)))^{t}}\leq \frac{\Phi_{t}(x)}{(\log (e+x))^{t}}=x.
$$
On the other hand, since $\Phi_{t}(x)\leq (x+1)^{t+1}$, we have
$$
\begin{aligned}
x(\log (e+\Phi_{t}(x)))^{t}&\leq  x(\log (e+(x+1)^{t+1}))^{t}\leq x(\log (e+x+1)^{t+1})^{t}\\
& \leq 2^t(t+1)^t x(\log (e+x))^{t}=2^t(t+1)^t \Phi_{t}(x).
\end{aligned}
$$
We now continue with the proof of (\ref{ie024}). By the fact that $\varphi_{t}^{-1}(x)=(\log (x+1))^{\frac{1}{t}}\leq(\log (e+x))^{\frac{1}{t}},$ it may lead to
$$
\begin{aligned}
\varphi_{s_1}^{-1}(x) \varphi_{s_2}^{-1}(x)\cdots \varphi_{s_m}^{-1}(x) \Phi_{1/s}^{-1}(x) &\leq
2^{\frac{1}{s}}(1+\frac{1}{s})^{\frac{1}{s}}(\log (e+x))^{\frac{1}{s_1}+\cdots+\frac{1}{s_m}}\frac{x}{(\log (e+x))^{\frac{1}{s}}}\\
&=2^{\frac{1}{s}}(1+\frac{1}{s})^{\frac{1}{s}}x,
\end{aligned}
$$
which together with (\ref{1e22}) completes the proof of Lemma \ref{lem2}.
\end{proof}
The following weighted John-Nirenberg inequality for BMO functions provides a foundation for our analysis.
\begin{lemma}[\cite{iba}]\label{lem3}
Let $b \in \mathrm{B M O}$ and $w \in A_{\infty}$. Then we have
$$
\left\|b-b_Q\right\|_{\exp L(w), Q} \leq c_n[w]_{A_{\infty}}\|b\|_{\mathrm{B M O}} .
$$
Furthermore, if $j>0$ then
$$
\left\|\left|b-b_Q\right|^j\right\|_{\exp L^{\frac{1}{j}}(w), Q} \leq c_{n, j}[w]_{A_{\infty}}^j\|b\|_{\mathrm{B M O}}^j .
$$
\end{lemma}
We are now ready to prove Theorem $\ref{thm1.2}$.
\begin{proof}[Proof of Theorem $\ref{thm1.2}$]
According to the Hypothesis \ref{hyp2}, we have
\begin{equation*}
\begin{aligned}
\big|\mathcal{T}_{\vec{b}}(\vec{f})(x)\big| \leq C \cdot \sum_{j=1}^{3^n}\sum_{\vec{\gamma}\in\{1,2\}^l}\mathcal{A}_{\mathcal{S}_{j},\vec{b}}^{\vec{\gamma}}(\vec{f})(x) \quad a.e. \ x\in \mathbb{R}^n.\\
\end{aligned}
\end{equation*}
By symmetry, we may assume that $\vec{\gamma}=(\overbrace{1, \ldots ,1}^{l_1}, \overbrace{2, \ldots ,2}^{l-l_1}).$  In order to show inequality (\ref{iethm1.2}), by using the triangle inequality, it suffices to show that
\begin{equation}\label{iethm1.2(1)}
\|\mathcal{A}_{\mathcal{S}_{j},\vec{b}}^{\vec{\gamma}}(\vec{f})\|_{L^p(w)} \lesssim \prod_{s =1}^{l} \|b_s\|_{\mathrm{BMO}}^{1/p}[w]_{A_\infty}^{l}
[w]_{A_\infty}^{\max\{2/p,1\}}\|\mathcal{M}_{L(\log L)}(\vec{f})\|_{L^p(w)},
\end{equation}
for every $1\leq j \leq 3^n.$\par
To prove (\ref{iethm1.2(1)}), we first consider the case $p>1$. For any $w\in A_\infty,$ by duality, we obtain
\begin{equation}
\begin{aligned}
\left\|\mathcal{A}_{\mathcal{S}_{j},\vec{b}}^{\vec{\gamma}}(\vec{f})\right\|_{L^p(w)}=\sup_{\|g\|_{L^{p^\prime}(w)}\leq 1}\left|\int_{\mathbb{R}^{n}}\mathcal{A}_{\mathcal{S}_{j},\vec{b}}^{\vec{\gamma}}(\vec{f})(x)g(x)w(x)dx \right|.
\end{aligned}
\end{equation}
For any fixed nonnegative function $g\in L^{p^\prime}(w)$ with $\|g\|_{L^{p^\prime}(w)}\leq 1,$ Lemmas \ref{lem2} and \ref{lem3} give that
\begin{equation*}
\begin{aligned}
\left|\int_{\mathbb{R}^{n}}\mathcal{A}_{\mathcal{S}_{j},\vec{b}}^{\vec{\gamma}}(\vec{f})(x)g(x)w(x)dx \right| &\leq  \sum_{Q \in \mathcal{S}_{j}}\prod_{s =1}^{l_1} {\langle |f_s|\rangle}_{Q} \int_Q \prod_{s =1}^{l_1} \left|b_s(x)-{\langle b_s \rangle}_{Q} \right| g(x)w(x)dx  \\
& \quad \quad \times \prod_{s =l_1+1}^{m}{\langle|(b_s-{\langle b_s \rangle}_{Q} )f_s|\rangle}_Q\prod_{s =l+1}^{m}{\langle|f_s|\rangle}_Q \\
&\lesssim \sum_{Q \in \mathcal{S}_{j}} w(Q)\prod_{s =1}^{l_1} {\langle |f_s|\rangle}_{Q}\prod_{s =1}^{l_1}\left\|b_s-{\langle b_s \rangle}_{Q} \right\|_{\exp L(w),Q}  \\
& \quad \quad \times \|g\|_{L(\log L)^{l_1}(w),Q}\prod_{s =l_1+1}^{l}\left\|b_s-{\langle b_s \rangle}_{Q} \right\|_{\exp L,Q}\|f_s\|_{L(\log L),Q}\\
&\quad \quad \times  \prod_{s =l+1}^{m}{\langle|f_s|\rangle}_Q \\
&\lesssim \sum_{Q \in \mathcal{S}_{j}}\|g\|_{L(\log L)^{l_1}(w),Q} \left(\prod_{s =1}^{m}\|f_s\|_{L(\log L),Q} \right)w(Q)\\
& \quad \times [w]_{A_\infty}^{l_1}\prod_{s =1}^{l} \|b_s\|_{\mathrm{BMO}}.
\end{aligned}
\end{equation*}
Then the Carleson embedding theorem combining with H\"{o}lder's inequality yields that

\begin{equation*}
\begin{aligned}
\left|\int_{\mathbb{R}^{n}}\mathcal{A}_{\mathcal{S}_{j},\vec{b}}^{\vec{\gamma}}f(x)g(x)w(x)dx \right| &\lesssim [w]_{A_\infty}^{l_1} \prod_{s =1}^{l} \|b_s\|_{\mathrm{BMO}}\sum_{Q \in \mathcal{S}_{j}}w(Q) \\
&\quad \times \left(\frac{1 }{w(Q)}\int_Q\left(\mathcal{M}_{L(\log L)}(\vec{f})(x) M_{L(\log L)^{l_1}(w)}^{\mathcal{D}_j}g(x)\right)^{\frac{1}{2}}w(x)dx\right)^2 \\
&\lesssim  [w]_{A_\infty}^{l_1+1}\prod_{s =1}^{l} \|b_s\|_{\mathrm{BMO}}\\
&\quad \times \int_{\mathbb{R}^{n}}\mathcal{M}_{L(\log L)}(\vec{f})(x) M_{L(\log L)^{l_1}(w)}^{\mathcal{D}_j}g(x)w(x)dx \\
&\lesssim [w]_{A_\infty}^{l_1+1}\prod_{s =1}^{l} \|b_s\|_{\mathrm{BMO}} \left\|\mathcal{M}_{L(\log L)}(\vec{f}) \right\|_{L^p(w)}\left\| (M_{w}^{\mathcal{D}_j})^{l_1+1}g\right\|_{L^{p'}(w)},
\end{aligned}
\end{equation*}
here we used that $M_{L(\log L)^{k}(w)}^{\mathcal{D}}f \simeq (M_{w}^{\mathcal{D}})^{k+1}f$ $(k\in \mathbb{N}^*)$ (\cite[p. 179]{per2}) in the last step.\\
For any weight $w$, since $\left\|M_w^{\mathcal{D}}f \right\|_{L^p(w)}\leq C \left\|f\right\|_{L^p(w)}$ (\cite[Theorem 15.1]{ler2}), then 
\begin{equation}\label{ie9}
\begin{aligned}
\left\|\mathcal{A}_{\mathcal{S}_{j},\vec{b}}^{\vec{\gamma}}(\vec{f})\right\|_{L^p(w)}\lesssim [w]_{A_\infty}^{l_1+1}\prod_{s =1}^{l} \|b_s\|_{\mathrm{BMO}} \left\|\mathcal{M}_{L(\log L)}(\vec{f}) \right\|_{L^p(w)}.
\end{aligned}
\end{equation}\par
We now turn our attention to the case $0<p\leq 1.$ By duality, it follows that
\begin{equation}\label{ie10}
\begin{aligned}
\|\mathcal{A}_{\mathcal{S}_{j},\vec{b}}^{\vec{\gamma}}(\vec{f})\|_{L^p(w)}=\sup_{\|g\|_{L^{2}(w)}\leq 1}\Big|\int_{\mathbb{R}^{n}}\left(\mathcal{A}_{\mathcal{S}_{j},\vec{b}}^{\vec{\gamma}}(\vec{f})(x)\right)^{\frac{p}{2}}
g(x)w(x)dx \Big|^{\frac{2}{p}}.
\end{aligned}
\end{equation}
Therefore,
$$
\begin{aligned}
\Big|\int_{\mathbb{R}^{n}}\left(\mathcal{A}_{\mathcal{S}_{j},\vec{b}}^{\vec{\gamma}}(\vec{f})(x)\right)^{\frac{p}{2}}
g(x)w(x)dx \Big| \leq &\sum_{Q \in \mathcal{S}_{j}}\prod_{s =1}^{l_1} {\langle |f_s|\rangle}_{Q}^{\frac{p}{2}} \int_Q \prod_{s =1}^{l_1} \left|b_s(x)-{\langle b_s \rangle}_{Q} \right|^{\frac{p}{2}} |g(x)|w(x)dx  \\
& \quad \quad \times \prod_{s =l_1+1}^{m}{\langle|(b_s-{\langle b_s \rangle}_{Q} )f_s|\rangle}_Q^{\frac{p}{2}}\prod_{s =l+1}^{m}{\langle|f_s|\rangle}_Q^{\frac{p}{2}}. \\
\end{aligned}
$$
Using Lemma \ref{lem2} and the same argument as in the case of $p>1$, we deduce that
\begin{equation*}
\begin{aligned}
\int_{\mathbb{R}^{n}}\left(\mathcal{A}_{\mathcal{S}_{j},\vec{b}}^{\vec{\gamma}}(\vec{f})(x)\right)^{\frac{p}{2}}
|g(x)|w(x)dx &\leq \sum_{Q \in \mathcal{S}_{j}} (l_1+1)(1+\frac{pl_1}{2})^{\frac{pl_1}{2}}2^{\frac{pl_1}{2}}w(Q)\left(\prod_{s =1}^{l_1} {\langle |f_s|\rangle}_{Q}^{\frac{p}{2}}\right)\\
&\quad \times \prod_{s =1}^{l_1}\left\| |b_s-{\langle b_s \rangle}_{Q} |^{\frac{p}{2}} \right\|_{\exp L^{\frac{p}{2}}(w),Q}  \|g\|_{L(\log L)^{^{\frac{pl_1}{2}}}(w),Q}\\
& \quad \quad \times C_n^{\frac{pl}{2}} \prod_{s =l_1+1}^{l}\|b_s\|_{\mathrm{BMO}}^{\frac{p}{2}}\|f_s\|_{L(\log L),Q}^{\frac{p}{2}}\prod_{s =l+1}^{m}{\langle|f_s|\rangle}_Q \\
&\leq (l_1+1)(1+\frac{pl_1}{2})^{\frac{pl_1}{2}}2^{\frac{pl_1}{2}}C_n^{\frac{pl}{2}} [w]_{A_\infty}^{\frac{pl_1}{2}}\prod_{s =1}^{l} \|b_s\|_{\mathrm{BMO}}^{\frac{p}{2}}\\
& \quad \times \sum_{Q \in \mathcal{S}_{j}}\|g\|_{L(\log L)^{\frac{pl_1}{2}}(w),Q} \left(\prod_{s =1}^{m}\|f_s\|_{L(\log L),Q}^{\frac{p}{2}} \right)w(Q).\\
\end{aligned}
\end{equation*}
 Then the Carleson embedding theorem yields the following inequality
\begin{equation}\label{ie11}
\begin{aligned}
\int_{\mathbb{R}^{n}}\left(\mathcal{A}_{\mathcal{S}_{j},\vec{b}}^{\vec{\gamma}}(\vec{f})(x)\right)^{\frac{p}{2}}
|g(x)|w(x)dx
&\leq (l_1+1)(1+\frac{pl_1}{2})^{\frac{pl_1}{2}}2^{\frac{pl_1}{2}}C_n^{\frac{pl}{2}} [w]_{A_\infty}^{1+\frac{pl_1}{2}}\prod_{s =1}^{l} \|b_s\|_{\mathrm{BMO}}^{\frac{p}{2}}\\
&\quad \times \int_{\mathbb{R}^{n}}(\mathcal{M}_{L(\log L)}(\vec{f})(x))^{\frac{p}{2}} M_{L(\log L)^{(pl_1)/2}(w)}^{\mathcal{D}_j}g(x)w(x)dx \\
&\leq (l_1+1)(1+\frac{pl_1}{2})^{\frac{pl_1}{2}}2^{\frac{pl_1}{2}}C_n^{\frac{pl}{2}} [w]_{A_\infty}^{1+\frac{pl_1}{2}}\prod_{s =1}^{l} \|b_s\|_{\mathrm{BMO}}^{\frac{p}{2}}\\
&\quad \times \left\|(\mathcal{M}_{L(\log L)}(\vec{f}))^{\frac{p}{2}} \right\|_{L^2(w)}\left\| M_{L(\log L)^{(pl_1)/2}(w)}^{\mathcal{D}_j}g\right\|_{L^{2}(w)}.
\end{aligned}
\end{equation}
Note that $\| M_{L(\log L)^{(pl_1)/2}(w)}^{\mathcal{D}_j}g\|_{L^{2}(w)}\leq \| M_{L(\log L)^{l_1}(w)}^{\mathcal{D}_j}g\|_{L^{2}(w)}\leq C_{n,l}\|g\|_{L^{2}(w)},$  which together with (\ref{ie10}) and (\ref{ie11}) gives
\begin{equation}\label{ie12}
\begin{aligned}
\|\mathcal{A}_{\mathcal{S}_{j},\vec{b}}^{\vec{\gamma}}(\vec{f})\|_{L^p(w)}&\leq C_{n,l} C_n^{\frac{4}{p}} (l+1)^{\frac{2}{p}}(1+\frac{pl}{2})^{l}[w]_{A_\infty}^{l+\frac{2}{p}}\prod_{s =1}^{l} \|b_s\|_{\mathrm{BMO}} \left\|\mathcal{M}_{L(\log L)}(\vec{f}) \right\|_{L^p(w)}.
\end{aligned}
\end{equation}
Combining (\ref{ie9}) with (\ref{ie12}), we conclude that for any $0<p<\infty$ and $ w\in A_\infty,$

\begin{equation*}
\int_{\mathbb{R}^n}\left|\mathcal{T}_{\vec{b}}(\vec{f})(x)\right|^pw(x)dx \lesssim \prod_{s =1}^{l} \|b_s\|_{\mathrm{BMO}}[w]_{A_\infty}^{pl}
[w]_{A_\infty}^{\max\{2,p\}}\int_{\mathbb{R}^n}\left(\mathcal{M}_{L(\log L)}(\vec{f})(x)\right)^pw(x)dx,
\end{equation*}
which yields the required estimate (\ref{iethm1.2}).
\end{proof}
\begin{remark}
For the Coifman-Fefferman inequality, we focus only on the precise $A_\infty$ weight constant. But in proving the mixed weighted weak type estimates, such as Theorem \ref{thm1.3}, we need the dependence of the constants in (\ref{ie12}) with respect to $p$ .
\end{remark}

\section{ Proofs of Theorems \ref{thm1.3} and Corollary \ref{cor1.2}}\label{Sect 5}
To prove Theorem \ref{thm1.3}, we need the following lemma with a more precise constant estimate, a previous version of which can be found in \cite[Lemma 2.3]{cru}.
\begin{lemma}\label{lem5}
If $u\in A_1, v\in A_p, 1\leq p<\infty,$ then $uv^\varepsilon \in A_p$ for all $0<\varepsilon <\frac{1}{2^{n+2}[u]_{A_1}}.$
\end{lemma}
\begin{proof}
Since $u\in A_1,$ by \cite[Lemma 3.26]{per3}, for each cube $Q$ it follows that
$$
\left(\frac{1}{|Q|} \int_Q u^{r_0}(x)dx\right)^{1 / r_0} \leq \frac{2}{|Q|} \int_Q u(x)dx,
$$
where $r_0=1+\frac{1}{2^{n+1}[u]_{A_1}}.$
For any $0<\varepsilon <\frac{1}{2^{n+2}[w]_{A_1}},$ let $t=(\frac{1}{\varepsilon})^{\prime}.$ Then the H\"{o}lder's inequality yields that
$$
\left(\frac{1}{|Q|} \int_Q u^{t}(x)dx\right)^{1 / t} \leq \frac{2}{|Q|} \int_Q u(x)dx.
$$\par
Consider first the case $p=1$. Since $u, v \in A_1$, for any cube $Q$ and almost every $x \in Q$,
$$
\begin{aligned}
 \frac{1}{|Q|} \int_Q u(x) v^\varepsilon(x) dx
 &\leq\left(\frac{1}{|Q|} \int_{Q} u^t(x)dx\right)^{1 / t}\left(\frac{1}{|Q|} \int_{Q} v^{\varepsilon t^\prime}(x) dx\right)^{1 / t^{\prime}} \leq 2[u]_{A_1}[v]_{A_1}^\varepsilon u(x) v^\varepsilon(x), \\
&
\end{aligned}
$$
which implies that $u v^\varepsilon \in A_1$ and $\left[u v^\varepsilon\right]_{A_1} \leq2[u]_{A_1}[v]_{A_1}^\varepsilon$.\par
If $1<p<\infty$ and $v \in A_p.$ Then for any cube $Q$, the H\"{o}lder's inequality implies that
$$
\begin{aligned}
& \left(\frac{1}{|Q|} \int_{Q} u(x) v^\varepsilon(x) d x\right)\left(\frac{1}{|Q|} \int_{Q}\left(u(x) v^\varepsilon(x)\right)^{1-p^{\prime}} d x\right)^{p-1} \\
 & \hspace{3cm} \leq\left(\frac{1}{|Q|} \int_{Q} u^t(x) d x\right)^{1 / t}\left(\frac{1}{|Q|} \int_{\mathrm{Q}} v^{\varepsilon t^\prime}(x)d x\right)^{1 / t^{\prime}} \\
& \hspace{4cm} \times\left(\frac{1}{|Q|} \int_{Q} u^{t(1-p^{\prime})}(x)d x\right)^{\frac{p-1}{t}}\left(\frac{1}{|Q|} \int_{Q}v^{\varepsilon t^\prime (1-p^{\prime})}(x) d x\right)^{\frac{p-1}{t^\prime}} \\
& \hspace{3cm} \leq \frac{2}{|Q|} \int_Q u(x) d x\|u^{-1}\|_{L^\infty(Q)} \left[\frac{1}{|Q|} \int_{Q} v(x) d x\left(\frac{1}{|Q|} \int_{Q}v^{1-p^{\prime}}(x) d x\right)^{p-1}\right]^\varepsilon\\
&\hspace{3cm}\leq2[u]_{A_1}[v]_{A_{p}}^\varepsilon .
\end{aligned}
$$
Therefore, $u v^\varepsilon \in A_p$ with $\left[u v^\varepsilon\right]_{A_p} \leq2[u]_{A_1}[v]_{A_p}^\varepsilon$.
This finishes the proof of Lemma \ref{lem5}.
\end{proof}
Now we need to show that Lemma \ref{lem5} implies Theorem $\ref{thm1.3}$.
\begin{proof}[Proof of Theorem $\ref{thm1.3}$]
Some basic ideas will be taken from \cite[Theorem 1.9]{li1}, and these ideas have also been used in \cite[Theorem 1.7]{cru}. Note that $u=w_1^{\frac{1}{m}} \cdots w_m^{\frac{1}{m}} \in A_1$ and $v\in A_\infty.$ Let $u \in A_1,$ $S_u$ be the operator defined by
$$
S_u f(x)=\frac{M(f u)(x)}{u(x)}
$$
if $u(x) \neq 0, S_u f(x)=0$ otherwise.\par
For any $h \in L^{r^{\prime}, 1}(u v^{\frac{1}{m}})$ with $h \geq 0,$ applying the Rubio de Francia algorithm with
$$
\mathcal{R} h(x)=\sum_{j=0}^{\infty} \frac{S_u^j h(x)}{(2 K_0)^j},
$$
where $K_0>0$ is an absolute constant which will be chosen later. A simple calculation shows that

$$
\begin{aligned}
& 0\leq h(x) \leq \mathcal{R} h(x) ; \quad
 S_u(\mathcal{R} h)(x) \leq 2 K_0 \mathcal{R} h(x) .
\end{aligned}
$$
Then, it follows from the second estimate that
$\mathcal{R} h\cdot u\in A_1$ and $[\mathcal{R} h\cdot u]_{A_1}\leq 2 K_0.$
Furthermore, we claim that there exists some $r>1$ such that $\mathcal{R} h \cdot u v^{\frac{1}{m r^{\prime}}} \in A_{\infty}$ and
$$
\|\mathcal{R} h\|_{L^{r^{\prime}, 1}(u v^{\frac{1}{m}})} \leq 2\|h\|_{L^{r^{\prime}, 1}(u v ^\frac{1}{m})} .
$$
We postpone the proof of this claim to the end of this section.\par
Observe that
\begin{equation}\label{ie17}
\begin{aligned}
\bigg\|\frac{\mathcal{T}_{\vec{b}}(\vec{f})}{v}\bigg\|_{L^{\frac{1}{m}, \infty}(u v^{\frac{1}{m}})}^{\frac{1}{m r}} 
& =\bigg\| \Big|\frac{\mathcal{T}_{\vec{b}}(\vec{f})}{v}\Big|^{\frac{1}{m r}}\bigg\|_{L^{r, \infty}\left(u v^{\frac{1}{m}}\right)},
\end{aligned}
\end{equation}
then using the duality property of Lorentz spaces for $1<p<\infty$ (Exercise 1.4.12 in \cite{gra}), it follows that
$$
\|f\|_{L^{p, \infty}(\mu)} \backsimeq \sup _{\|g\|_{L^{p^{\prime}, 1}(\mu)} \leq 1}\left|\int_{\mathbb{R}^n} f(x) g(x) d \mu(x)\right|.
$$
This means that there exist $c_1,c_2 >0,$
\begin{equation}\label{ie16}
\begin{aligned}
c_1 \sup _{\|g\|_{L^{p^{\prime}, 1}(\mu)} \leq 1}\Big |\int_{\mathbb{R}^n} f(x) g(x) d \mu(x)\Big| \leq\|f\|_{L^{p, \infty}(\mu)}\leq c_2 \sup _{\|g\|_{L^{p^{\prime}, 1}(\mu)} \leq 1}\Big|\int_{\mathbb{R}^n} f(x) g(x) d \mu(x)\Big|.
\end{aligned}
\end{equation}
In the following, we will calculate the exact value of $c_1$ and $c_2$. \par
Suppose that $\mathcal{X}$ is a quasi-Banach space and let $\mathcal{X}^*$  be its dual space. Then for all $T\in \mathcal{X}^*$, we have
$$
\|T\|_{\mathcal{X}^*}=\sup _{\substack{x \in \mathcal{X} \\\|x\|_\mathcal{X} \leq 1}}|T(x)| .
$$
Let $\mathcal{X}=L^{p^{\prime},1}(\mu)$ and $\mathcal{X}^*=L^{p,\infty}(\mu).$  For a fixed $f\in L^{p,\infty}(\mu),$ we define $T_f(g)$ by
$$T_f(g):=\int_{\mathbb{R}^n}f(x)g(x)d\mu(x), \quad g\in L^{p^{\prime},1}(\mu).$$
Therefore
\begin{equation}\label{ie14}
\|T_f\|_{\mathcal{X}^*}=\sup _{\substack{x \in \mathcal{X} \\\|x\|_\mathcal{X} \leq 1}}|T_f(x)| =\sup _{\|g\|_{L^{p^{\prime}, 1}(\mu)} \leq 1}\Big|\int_{\mathbb{R}^n}f(x)g(x)d\mu(x)\Big|.
\end{equation}
In addition, the discussion in \cite[p. 59]{gra} gives
\begin{equation}\label{ie15}
\|T_f\|_{\mathcal{X}^*} \leq \|f\|_{L^{p, \infty}(\mu)} \leq p^{\prime}\|T_f\|_{\mathcal{X}^*}.
\end{equation}
 Inserting (\ref{ie14}) in (\ref{ie15}) and picking $c_1=1,c_2=p^{\prime}$ yields what we want to prove. \par
Let us continue to estimate (\ref{ie17}).
Applying (\ref{ie16}) with $c_2=r^{\prime},$ we obtain
\begin{equation}\label{ie13}
\begin{aligned}
\bigg\|\frac{\mathcal{T}_{\vec{b}}(\vec{f})}{v}\bigg\|_{L^{\frac{1}{m}, \infty}(u v^{\frac{1}{m}})}^{\frac{1}{m r}}
& \leq r^{\prime} \sup _{\substack{h \in L^{r^{\prime}, 1}(u v^{\frac{1}{m}})\\ \|h\|_{L^{r^{\prime}, 1}(u v^{\frac{1}{m}})}\leq1}}\int_{\mathbb{R}^n}\Big| \mathcal{T}_{\vec{b}}(\vec{f})(x)\Big|^{\frac{1}{m r}} \mathcal{R}h(x) u(x) v^{\frac{1}{mr^\prime}}(x) d x .
\end{aligned}
\end{equation}
Note that ${mr}>1$, where $r$ depends on the weights $u$ and $v.$ Then, using the weighted constant estimates in Theorem \ref{thm1.2} with  $0<p\leq 1$ gives
\begin{equation*}
\begin{aligned}
{}&\int_{\mathbb{R}^n}\Big| \mathcal{T}_{\vec{b}}(\vec{f})(x)\Big|^{\frac{1}{m r}} \mathcal{R}h(x) u(x) v^{\frac{1}{mr^\prime}}(x) d x\\ &\leq C_{n,l}^{\frac{l}{mr}} C_n (l+1)^2(1+\frac{1}{2mr})^{\frac{l}{mr}}[\mathcal{R}h\cdot uv^{\frac{1}{mr^\prime}}]_{A_\infty}^{2+\frac{l}{mr}}\prod_{s =1}^{l} \|b_s\|_{\mathrm{BMO}}^{\frac{1}{mr}}\\ &
\quad \times\int_{\mathbb{R}^n}\Big| \mathcal{M}_{L(\log L)}(\vec{f})(x)\Big|^{\frac{1}{m r}} \mathcal{R}h(x) u(x) v^{\frac{1}{mr^\prime}}(x) d x\\
& \leq C_{n,m,l}[\mathcal{R}h\cdot uv^{\frac{1}{mr^\prime}}]_{A_\infty}^{2+\frac{l}{m}}\prod_{s =1}^{l} \|b_s\|_{\mathrm{BMO}}^{\frac{1}{mr}}\\
&\quad \times \int_{\mathbb{R}^n}\Big| \frac{\mathcal{M}_{L(\log L)}(\vec{f})(x)}{v(x)}\Big|^{\frac{1}{m r}} \mathcal{R}h(x) u(x) v^{\frac{1}{m}}(x) d x.\\
\end{aligned}
\end{equation*}
This inequality, together with the H\"{o}lder's inequality (\ref{ie16}) in Lorentz spaces, implies that
$$
\begin{aligned}
\int_{\mathbb{R}^n}\Big| \mathcal{T}_{\vec{b}}(\vec{f})(x)\Big|^{\frac{1}{m r}} \mathcal{R}h(x) u(x) v^{\frac{1}{mr^\prime}}(x) d x&
\leq C_{n,m,l}[\mathcal{R}h\cdot uv^{\frac{1}{mr^\prime}}]_{A_\infty}^{2+\frac{l}{m}}\prod_{s =1}^{l} \|b_s\|_{\mathrm{BMO}}^{\frac{1}{mr}}\\
& \quad \times \bigg\|\frac{\mathcal{M}_{L(\log L)}(\vec{f})}{v}\bigg\|_{L^{\frac{1}{m}, \infty}(u v^{\frac{1}{m}})}^{\frac{1}{m r}}\|\mathcal{R}h\|_{L^{r^\prime, 1}(u v^{\frac{1}{m}})}^{\frac{1}{m r}}
\end{aligned}
$$
Recall that
$
\|\mathcal{R} h\|_{L^{r^{\prime}, 1}(u v^{\frac{1}{m}})} \leq 2\|h\|_{L^{r^{\prime}, 1}(u v ^\frac{1}{m})}.
$Then for $h\in {L^{r^{\prime}, 1}(u v ^\frac{1}{m})}$ in (\ref{ie13}), 
we have
\begin{equation}\label{ie18}
\begin{aligned}
\bigg\|\frac{\mathcal{T}_{\vec{b}}(\vec{f})}{v}\bigg\|_{L^{\frac{1}{m}, \infty}(u v^{\frac{1}{m}})}^{\frac{1}{ r}}
&\leq C_{n,m,l} (r^{\prime})^m [\mathcal{R}h\cdot uv^{\frac{1}{mr^\prime}}]_{A_\infty}^{2m+l}\prod_{s =1}^{l} \|b_s\|_{\mathrm{BMO}}^{\frac{1}{r}}\bigg\|\frac{\mathcal{M}_{L(\log L)}(\vec{f})}{v}\bigg\|_{L^{\frac{1}{m}, \infty}(u v^{\frac{1}{m}})}^{\frac{1}{ r}}.
\end{aligned}
\end{equation}\par
To finish the proof it remains to show our foregoing claim and get the value of $K_0$. The proof follows the same scheme of that in \cite{cru}, but we have a more precise estimate of the constants.
Since $\vec{w}=\left(w_1, \ldots, w_m\right) \in A_{\vec{1}}$ and $v\in A_\infty,$ it follows that $u=w_1^{\frac{1}{m}} \ldots w_m^{\frac{1}{m}} \in A_1$ and $v^{\frac{1}{m}}\in A_\infty.$ The former indicates that $S_u$ is bounded on $L^{\infty}(u v^{\frac{1}{m}})$ with constant $C_0=[u]_{A_1},$ that is,
\begin{equation}\label{ie19}
\|S_u f\|_{L^{\infty}(u v^{\frac{1}{m}})} \leq[u]_{A_1}\|f\|_{L^{\infty}(u v^{\frac{1}{m}})}.
\end{equation}
In order to apply the Marcinkiewicz interpolation theorem, we also need to show that
$S_u$ is bounded on $L^{p_0}(u v^{\frac{1}{m}})$ for some $1<p_0<\infty$.
Note that for $v^{1 / m} \in A_{\infty},$ there exists $t>1$ such that $v^{\frac{1}{m}} \in A_t.$ Then $A_p$ factorization theorem (\cite[Theorem 7.5.1]{gra}) tells us that there exist $v_1, v_2 \in A_1$ such that $v^{\frac{1}{m}}=v_1 v_2^{1-t}.$ Using these facts, we have
$$
u^{1-p_0} v^{\frac{1}{m}}=v_1\left(u v_2^{\frac{t-1}{p_0-1}}\right)^{1-p_0} .
$$
 Applying Lemma \ref{lem5}, for any $0<\varepsilon <\frac{1}{2^{n+2}[u]_{A_1}}$ and $v_2\in A_1,$ yields that $uv_2^\varepsilon \in A_1.$ Then we pick
$p_0=2^{n+3}(t-1)[u]_{A_1}+1$ such that $u v_2^{\frac{t-1}{p_0-1}}\in A_1,$ which further implies $$u^{1-p_0} v^{\frac{1}{m}}\in A_{p_0}.$$
Observe that
$$
\int_{\mathbb{R}^n} |S_uf(x)|^{p_0} u(x) v^{\frac{1}{m}}(x) d x=\int_{\mathbb{R}^n} (M(f u)(x))^{p_0} u^{1-p_0}(x) v^{\frac{1}{m}}(x) d x.
$$
By this observation and Buckley's theorem \cite[Theorem 3.11]{per3}, it follows that
\begin{equation}\label{ie20}
\|S_u f\|_{L^{p_0}(u v^{\frac{1}{m}})}=\|M(f u)\|_{L^{p_0}(u^{1-p_0} v^{\frac{1}{m}})} \leq c_np_0^{\prime}[u^{1-p_0} v^{\frac{1}{m}}]_{A_{p_0}}^{\frac{1}{p_0-1}}\|f\|_{L^{p_0}(uv^\frac{1}{m})},
\end{equation}
and thus $S_u: L^{p_0}(u v^{\frac{1}{m}})\rightarrow L^{p_0}(u v^{\frac{1}{m}})$ holds with
constant $C_1=c_np_0^{\prime}[u^{1-p_0} v^{\frac{1}{m}}]_{A_{p_0}}^{\frac{1}{p_0-1}}.$\par
The Marcinkiewicz interpolation theorem in \cite[Proposition A.1]{cru}, together with (\ref{ie19}) and (\ref{ie20}), yields that $S_u$ is bounded on $L^{q, 1}(uv^{\frac{1}{m}})$ for all $p_0<q<\infty$ and enjoys the property 
$$
\|S_u f\|_{L^{q, 1}(u v^{\frac{1}{m}})} \leq 2^{\frac{1}{q}}\Big(C_1(\frac{1}{p_0}-\frac{1}{q})^{-1}+C_0\Big)\|f\|_{L^{q, 1}(u v^{\frac{1}{m}})} .
$$
Notice that, if $q\geq 2p_0,$ then
$$
2^{\frac{1}{q}}\Big(C_1(\frac{1}{p_0}-\frac{1}{q})^{-1}+C_0\Big)\leq 4p_0(C_0+C_1)\leq c_np_0p_0^{\prime}([u^{1-p_0} v^{\frac{1}{m}}]_{A_{p_0}}^{\frac{1}{p_0-1}}+[u]_{A_1}).
$$
Using Lemma \ref{lem5} and $A_p$ factorization theorem again, we obtain
\begin{equation*}
\begin{aligned}
 \ [u^{1-p_0} v^{\frac{1}{m}}]_{A_{p_0}}^{\frac{1}{p_0-1}}&\leq [v_1]_{A_1}\left[u v_2^{\frac{t-1}{p_0-1}}\right]_{A_1}^{p_0-1} \leq2^{p_0-1}c_n^t[v^{\frac{1}{m}}]_{A_{t}}^2[u]_{A_1}^{p_0-1}.
\end{aligned}
\end{equation*}
Let $K_0= 4c_np_0p_0^{\prime}([u]_{A_1}+2^{p_0-1}c_n^t[v^{\frac{1}{m}}]_{A_{t}}^2[u]_{A_1}^{p_0-1})+1,$ then
$$
\|S_u f\|_{L^{q, 1}(u v^{\frac{1}{m}})} \leq K_0\|f\|_{L^{q, 1}(u v^{\frac{1}{m}})},\quad \hbox{\ for all } q\geq 2p_0 .
$$
Take $\varepsilon=\frac12{\min\{\frac{1}{2^{n+3}K_0},\frac{1}{2p_0} \}}$
 and $r=(\frac{1}{\varepsilon})^{\prime},$ which implies that $r^{\prime}>2p_0.$ This ensures that $S_u$ is bounded on $L^{r^{\prime}, 1}(u v^{\frac{1}{m}})$ with constant bounded by $K_0$.\par
 For any $0<\varepsilon<\frac{1}{2^{n+3}K_0}$ and $W_1\in A_1$ with $[W_1]_{A_1}\leq 2K_0.$ Lemma \ref{lem5} and a simple calculation yield that
$W_1W_2^{\varepsilon}\in A_t, \hbox{for all } W_2\in A_t.$
Based on the facts that $\mathcal{R} h\cdot u\in A_1$ with $[\mathcal{R} h\cdot u]_{A_1}\leq 2 K_0$
 and $v^{\frac{1}{m}}\in A_t$, we have
$$\mathcal{R} h \cdot u v^{\frac{1}{m r^{\prime}}}=\mathcal{R} h \cdot u v^{\frac{\varepsilon}{m }} \in A_{t}\subseteq A_\infty,$$
 where $\frac{1}{r^{\prime}}=\varepsilon<\frac{1}{2^{n+3}K_0}.$ In addition, it is easy to see that
\begin{equation}\label{ie21}
\begin{aligned}
\ [\mathcal{R} h \cdot u v^{\frac{1}{m r^{\prime}}}]_{A_\infty}&\leq [\mathcal{R} h \cdot u v^{\frac{1}{m r^{\prime}}}]_{A_t}
\leq2[\mathcal{R} h \cdot u]_{A_1}[v^{\frac{1}{m}}]_{A_t}^{\frac{1}{r^\prime}}\leq 4K_0[v^{\frac{1}{m}}]_{A_t}.
\end{aligned}
\end{equation}
This completes the proof of the claim. Note that the following two facts hold:
$r\leq2 $, which follows from $K_0\geq 1$ and $p_0>1.$  $r^\prime\leq2^{n+4}K_0$ since $K_0\geq p_0.$ \par
Combining these two facts and using (\ref{ie18}), (\ref{ie21}), we deduce
$$
\left\|\frac{\mathcal{T}_{\vec{b}}(\vec{f})}{v}\right\|_{L^{\frac{1}{m}, \infty}(u v^{\frac{1}{m}})} \leq C_{n,m,l} K_0^{2l+6m}[v^{\frac{1}{m}}]_{A_t}^{2l+4m}\prod_{s =1}^{l} \|b_s\|_{\mathrm{BMO}}\left\|\frac{ \mathcal{M}_{L(\log L)} (\vec{f})}{v}\right\|_{L^{\frac{1}{m}, \infty}(u v^{\frac{1}{m}})},
$$
which completes the proof of Theorem \ref{thm1.3}.
\end{proof}

\begin{proof}[Proof of Corollary $\ref{cor1.2}$]
Let $v\equiv 1,$ applying Theorem \ref{thm1.4} and Remark $\ref{remark1}$ with $v\in A_2$, Corollary  $\ref{cor1.2}$ follows easily by a simple calculation.
\end{proof}

\section{ Proofs of Theorems \ref{thm1.4} and Corollary \ref{cor1.3}}\label{Sect 6}
\begin{proof}[Proof of Theorem $\ref{thm1.4}$]
Let $\vec{\gamma}=(\overbrace{1, \ldots ,1}^{l_1}, \overbrace{2, \ldots ,2}^{l-l_1}),$ under the Hypothesis $\ref{hyp2}$ of $\mathcal{T}_{\vec{b}},$ we consider only the contribution of $\mathcal{A}_{\mathcal{S}_{j},\vec{b}}^{\vec{\gamma}}(\vec{f})$ for $j=1,\ldots,3^n.$ \par
Fix exponents $\frac{1}{p}=\frac{1}{p_1}+\cdots+\frac{1}{p_m}$ with $1<p_1, \ldots, p_m<\infty$ and weights $\vec{w}=\left(w_1, \ldots, w_m\right)$. It follows from the definition of $\mathcal{A}_{\mathcal{S}_{j},\vec{b}}^{\vec{\gamma}}(\vec{f})$ and $0<p\leq 1$ that
\begin{equation}\label{ie22}
\begin{aligned}
\| \mathcal{A}_{\mathcal{S}_{j},\vec{b}}^{\vec{\gamma}}(\vec{f}) \|_{L^p(\nu_{\vec{w}})}^p \leq
 \sum_{Q \in \mathcal{S}_{j}}& \prod_{s =1}^{l_1}{\langle|f_s|\rangle}_Q^p \int_Q \prod_{s =1}^{l_1}\left|b_s(x)-{\langle b_s \rangle}_{Q} \right|^p \nu_{\vec{w}}(x)dx \\
 &\times  \prod_{s =l_1+1}^l {\langle|(b_s-{\langle b_s \rangle}_{Q} )f_s|\rangle}_Q^p\prod_{s =l+1}^{m}{\langle|f_s|\rangle}_Q^p.
\end{aligned}
\end{equation}
Let $v_i(x)=M w_i(x)$ with $i=1,\ldots, m,$ then $ v_i(x)\geq\langle w_i\rangle_Q$ for a.e. $x\in Q$ where $Q$ is any dyadic cube contained in $\mathcal{S}_j.$
Using the H\"{o}lder's inequality, we obtain
\begin{equation}\label{ie23}
\begin{aligned}
\int_Q \prod_{s =1}^{l_1}\left|b_s(x)-{\langle b_s \rangle}_{Q} \right|^p \nu_{\vec{w}}(x)dx  \lesssim \prod_{s =1}^{l_1}\|b_s\|_{\mathrm{BMO}_{p_s}(w_s)}^p \prod_{s =1}^{m} \left(  w_s(Q) \right)^{\frac{p}{p_s}}.
\end{aligned}
\end{equation}
Pick $r,t\in \mathbb{R}$ such that $1<r<t<\min _i\{p_i\}.$ Applying the H\"{o}lder's inequality twice gives
\begin{equation}\label{ie24}
\begin{aligned}
 {\langle|(b_s-{\langle b_s \rangle}_{Q} )f_s|\rangle}_Q & \lesssim \|b_s\|_{\mathrm{BMO}}{\langle|f_s|^r\rangle}_Q^{\frac{1}{r}}  \leq \left\|b_s\right\|_{\mathrm{BMO}} {\langle|f_s|^t v_s^{\frac{t}{p_s}}\rangle}_Q^{\frac{1}{t}} {\langle w_s \rangle}_Q^{-{\frac{1}{p_s}}}.
 \end{aligned}
\end{equation}
Putting the estimates (\ref{ie22}-(\ref{ie24}) together we obtain
\begin{equation*}
\begin{aligned}
\left\| \mathcal{A}_{\mathcal{S}_{j},\vec{b}}^{\vec{\gamma}}(\vec{f}) \right\|_{L^p(\nu_{\vec{w}})}^p &\leq
 \sum_{Q \in \mathcal{S}_{j}} \prod_{s =1}^{l_1} \|b_s\|_{\mathrm{BMO}_{p_s}(w_s)}^p  \prod_{s =l_1+1}^{l} \|b_s\|_{\mathrm{BMO}}^p \prod_{s =1}^{m}  {\langle|f_s|^t v_s^{\frac{t}{p_s}}\rangle}_Q^{\frac{p}{t}} \\
 & \quad \quad\times\prod_{s =1}^{m} {\langle w_s \rangle}_Q^{-{\frac{p}{p_s}}}\prod_{s =1}^{m} w_s(Q)^{\frac{p}{p_s}}  \\
 & \lesssim   \sum_{Q \in \mathcal{S}_{j}} \left(\prod_{s =1}^{m} \inf_{x \in Q}\left(M_t(f_s v_s^{\frac{1}{p_s}}) \right)^p(x)\right)|Q| \\
 & \quad \times\prod_{s =1}^{l_1}  \|b_s\|_{\mathrm{BMO}_{p_s}(w_s)}^p \prod_{s =l_1+1}^{l} \|b_s\|_{\mathrm{BMO}}^p\\
 &\lesssim  \sum_{Q \in \mathcal{S}_{j}} \left(\prod_{s =1}^{m} \inf_{x \in E_ {Q}}\left(M_t(f_s v_s^{\frac{1}{p_s}})\right)^p(x)\right)|E_Q|\\
 &\quad \times\prod_{s =1}^{l_1}  \|b_s\|_{\mathrm{BMO}_{p_s}(w_s)}^p \prod_{s =l_1+1}^{l} \|b_s\|_{\mathrm{BMO}}^p,
\end{aligned}
\end{equation*}
where in the last inequality, we have used the sparseness property of the collection $\mathcal{S}_{j}$.\\
Note that  $1<t<\min _i\{p_i\}$ and $\{E_Q\}$ are pairwise disjoint, it follows that
\begin{equation*}
\begin{aligned}
\left\| \mathcal{A}_{\mathcal{S}_{j},\vec{b}}^{\vec{\gamma}}(\vec{f}) \right\|_{L^p(\nu_{\vec{w}})}^p &\lesssim \prod_{s =1}^{l_1} \|b_s\|_{\mathrm{BMO}_{p_s}(w_s)}^p \prod_{s =l_1+1}^{l} \|b_s\|_{\mathrm{BMO}}^p \int_{\mathbb{R}^{n}} \prod_{s =1}^{m} \left(M_t(f_s v_s^{\frac{1}{p_s}}) \right)^p(x)dx \\
& \leq \left( \|\vec{b}\|_{\mathrm{BMO}}^*\right)^p \prod_{s =1}^{m} \|M_t(f_s v_s^{\frac{1}{p_s}}) \|_{L^{p_s}(\mathbb{R}^{n})}^p  \\
& \lesssim \left( \|\vec{b}\|_{\mathrm{BMO}}^*\right)^p \prod_{s =1}^{m} \|f_s  \|_{L^{p_s}(Mw_s)}^p.
\end{aligned}
\end{equation*}
The proof of Theorem \ref{thm1.4} is finished.
\end{proof}

\begin{proof}[Proof of Corollary $\ref{cor1.3}$]
It was shown in \cite[Corollary 2.2]{omb} that, for any $ w \in A_{\infty}$ and $f \in \mathrm{BMO},$ there exists a dimensional constant $c_n$ independent of $f$ and $w$ such that for each cube $Q,$ 
\begin{equation*}
\begin{aligned}
\left( \frac{1}{w(Q)}  \int_Q \left|f(x)-{\langle f \rangle}_{Q} \right|^p w(x)dx \right)^{\frac{1}{p}} \leq c_n p[w]_{A_{\infty}}\|f\|_{\mathrm{BMO}}. \\
\end{aligned}
\end{equation*}
Therefore, if $w_s \in A_{\infty}$ and $b_s \in \mathrm{BMO}$ with $1\leq s \leq l,$ then $b_s \in {\mathrm{BMO}_{p_s}(w_s)}$ and
$$ \| b_s\|_{\mathrm{BMO}_{p_s}(w_s)}\leq c_n p_s[w_s]_{A_{\infty}}\|b_s\|_{\mathrm{BMO}},$$
which implies that
$$  \|\vec{b}\|_{\mathrm{BMO}}^* \leq C_{n ,p} \prod_{s =1}^{l}[w_s]_{A_{\infty}}\prod_{s =1}^{l} \|b_s\|_{\mathrm{BMO}}.$$
Plugging this estimate into Theorem \ref{thm1.4} we achieve the desired conclusion (1) in Corollary $\ref{cor1.3}$.

It remains to show the estimate (2) holds in Corollary $\ref{cor1.3}$. We will use the same notation as in the proof of Theorem \ref{thm1.4}. Note that
for every $x \in Q\subseteq 2Q,$ if $w_s\in A_\infty^{\text{weak}},$ then $v_s(x):=Mw_s(x)\geq {\langle w_s \rangle}_{2Q}$.
For a fixed $j\in \{1,\ldots, 3^n\},$ we need to estimate
\begin{equation}\label{ie25(1)}
\sum_{Q \in \mathcal{S}_{j}} \prod_{s =1}^{l_1}{\langle|f_s|\rangle}_Q^p \int_Q \prod_{s =1}^{l_1}\left|b_s(x)-{\langle b_s \rangle}_{Q} \right|^p \nu_{\vec{w}}(x)dx
 \prod_{s =l_1+1}^l {\langle|(b_s-{\langle b_s \rangle}_{Q} )f_s|\rangle}_Q^p\prod_{s =l+1}^{m}{\langle|f_s|\rangle}_Q^p.
 \end{equation}
A simple calculation yields that
 \begin{equation}\label{ie25.1}
\begin{aligned}
{}& \int_Q \prod_{s =1}^{l_1}\left|b_s(x)-{\langle b_s \rangle}_{Q} \right|^{p}\nu_{\vec{w}}(x)dx  \\
 & \leq \prod_{s =1}^{l_1}\left( \frac{1}{w(2Q)}\int_Q \left|b_s(x)-{\langle b_s \rangle}_{Q} \right|^{p_s} w_s(x)dx \right)^{\frac{p}{p_s}} \prod_{s =1}^{m}(w_s(2Q))^{\frac{p}{p_s}} \\
 &\lesssim  \prod_{s =1}^{l_1}\left( [w_s]_{A_{\infty }}^{weak} \right)^p \prod_{s =1}^{l_1}\|b_s\|_{\mathrm{BMO}}^p \prod_{s =1}^{m} (w_s(2Q))^{\frac{p}{p_s}},
\end{aligned}
\end{equation}
 where the last inequality follows from \cite[Corollary 2.4]{omb}.\par
 Therefore, using the same ideas as what have been used in (\ref{ie24}), we obtain
 \begin{equation}\label{ie26.1}
\begin{aligned}
{\langle|(b_s-{\langle b_s \rangle}_{Q} )f_s|\rangle}_Q &\lesssim \|b_s\|_{\mathrm{BMO}} {\langle|f_s|^t v_s^{\frac{t}{p_s}}\rangle}_Q^{\frac{1}{t}} {\langle w_s \rangle}_{2Q}^{-{\frac{1}{p_s}}}.\\
\end{aligned}
\end{equation}
This inequality, together with inqualities  (\ref{ie25(1)})-(\ref{ie26.1}), gives that
\begin{equation*}
\begin{aligned}
\left\| \mathcal{A}_{\mathcal{S}_{j},\vec{b}}^{\vec{\gamma}}(\vec{f}) \right\|_{L^p(\nu_{\vec{w}})}^p &\lesssim \prod_{s =1}^{l_1} \left( [w_s]_{A_{\infty }}^{weak} \right)^p \prod_{s =1}^{l}\|b_s\|_{\mathrm{BMO}}^p \prod_{s =1}^{m} {\langle|f_s|^t v_s^{\frac{t}{p_s}}\rangle}_Q^{\frac{p}{t}} \prod_{s =1}^{m}{\langle w_s \rangle}_{2Q}^{-{\frac{p}{p_s}}} \prod_{s =1}^{m}(w_s(2Q))^{\frac{p}{p_s}}  \\
&\lesssim  \prod_{s =1}^{l_1} \left( [w_s]_{A_{\infty }}^{weak} \right)^p \prod_{s =1}^{l}\|b_s\|_{\mathrm{BMO}}^p\prod_{s =1}^{m} \|f_s  \|_{L^{p_s}(Mw_s)}^p. \\
\end{aligned}
\end{equation*}\par
 Keeping this estimate in mind and applying the same reasoning to the other forms of $\vec{\gamma},$ we obtain
 \begin{equation*}
\begin{aligned}
\left\| \mathcal{T}_{\vec{b}}(\vec{f}) \right\|_{L^p(\nu_{\vec{w}})} &\lesssim \sum_{\vec{\gamma}\in \{1,2\}^l}(\prod_{s:\gamma_s=1}[w_s]_{A_\infty}^{\text{weak}})\prod_{s =1}^{l}\|b_s\|_{\mathrm{BMO}}\prod_{s =1}^{m} \|f_s  \|_{L^{p_s}(Mw_s)},
\end{aligned}
\end{equation*}
which finishes the proof of corollary $\ref{cor1.3}$.
\end{proof}

\section{ Proofs of Theorem \ref{thm1.5}}\label{Sect 7}
This section will be devoted to demonstrate Theorem \ref{thm1.5}. For this purpose, we first present the definition  of $N$-function.
We say a Young function $\phi$ is an $N$-function if it satisfies
$$\lim _{t \rightarrow 0^{+}} \frac{\phi(t)}{t}=0 \quad \text { and } \quad \lim _{t \rightarrow \infty} \frac{\phi(t)}{t}=\infty.$$\par
Moreover, an $N$-function is said to satisfy the sub-multiplicative property if $\phi(st)\leq \phi(s)\phi(t)$ for any $s,t\geq 0,$ 
For convenience, we need to state some properties as well as the lemmata of $\phi \in \Phi$ and it's complementary function $\bar{\phi}.$

\begin{itemize}
	\item  (Young's inequality) $st \leq \phi(s)+\bar{\phi}(t), s,t \geq 0.$
	\item When $\phi$ is an $N$-function, then $\bar{\phi}$ is also an $N$-function, and the following inequalities hold:
	\begin{equation}\label{ie6.1}
		t \leq \phi^{-1}(t) \bar{\phi}^{-1}(t) \leq 2 t, t \geq 0;
	\end{equation}
	\begin{equation}\label{ie6.2}
		\bar{\phi}\left(\frac{\phi(t)}{t}\right) \leq \phi(t), t>0.
	\end{equation}
	\item Let $\phi$ be an $N$-function, then there exists $0<\alpha<1$ such that $\phi^{\alpha}$ is quasi-convex if and only if $\bar{\phi} \in \Delta_{2},$ where $\phi^{\alpha}(t)=\phi(t)^{\alpha}.$
	\item  $\phi \in \Delta_{2}$  if and only if there exists some constant $C_1>0$ such that for any $\lambda \geq 2$,
	\begin{equation}\label{ie6.3}
		\phi(\lambda t) \leq 2^{C_1} \lambda^{C_1} \phi(t), t>0.
	\end{equation}
\end{itemize}\par

The next two technical lemmata allow us to prove Theorem \ref{thm1.5}. The first one is a modular inequality with respect to the Hardy-Littlewood maximal operator $M.$
\begin{lemma}[\cite{and1}]\label{lem7}
Let $\phi \in \Phi$ and  be quasi-convex. If $1<i_\phi<\infty$ and $w \in A_{i_\phi}$, then
$$
\int_{\mathbb{R}^n} \phi(M f(x)) w(x) d x \leq C \int_{\mathbb{R}^n} \phi\left(C[w]_{A_{i_\phi}}^{1 / i_\phi}|f(x)|\right) w(x) dx,
$$
where $C$ is an absolute constant which only depends on $\phi$ and $\alpha$.
\end{lemma}
The next lemma concerns on a modular inequality for sparse operators (see \cite[Lemma 3.13]{and1}), which plays a foundamental role in our analysis.
\begin{lemma}[\cite{and1}]\label{lem6}
Let $\phi$ be an $N$-function with $\phi \in \Delta_2,$ and $w \in A_{i_\phi}$. If $i_\phi>1$, for any dyadic grid $\mathcal{D}$ and $\mathcal{S} \subseteq \mathcal{D}$ a sparse family, we have
$$
\int_{\mathbb{R}^n} \phi\left(\mathcal{A}_{\mathcal{S}}(f)(x)\right) w(x) dx \leq C[w]_{A_{\infty}}^{1+\widetilde{C}} \int_{\mathbb{R}^n} \phi(M f(x)) w(x) dx,
$$
where $C$ is an absolute constant only depending on $\phi$ and $\widetilde{C}$ which satisfies
$\phi(\lambda t)\leq 2^{\widetilde{C}}\lambda^{\widetilde{C}}\phi(t)$ for $ \lambda\geq2$ and $ t>0.$
\end{lemma}

\begin{proof}[Proof of Theorem  $\ref{thm1.5}$]
For $1\leq j\leq 3^n$ and $\vec{\gamma}=(\overbrace{1, \ldots ,1}^{l_1}, \overbrace{2, \ldots ,2}^{l-l_1}),$ we first consider the contribution of $ \mathcal{A}_{\mathcal{S}_{j},\vec{b}}^{\vec{\gamma}}$ and try to show that there exists a constant $K$ such that
\begin{equation}\label{ie1.5(1)}
\int_{\mathbb{R}^n} \phi\left( \mathcal{A}_{\mathcal{S}_{j},\vec{b}}^{\vec{\gamma}}(\vec{f})(x)\right) w(x) d x \lesssim  \left(\prod_{i=1}^m\int_{\mathbb{R}^n} \phi^m\left(K\left|f_i(x)\right|\right) w(x) d x\right)^{\frac{1}{m}}.
\end{equation}\par
Note that if
$\mathcal{A}_{\mathcal{S}_{j},\vec{b}}^{\vec{\gamma}}(\vec{f})(x)=0$, then
$\phi( \mathcal{A}_{\mathcal{S}_{j},\vec{b}}^{\vec{\gamma}}(\vec{f})(x))=0$ since $\phi$ is an $N$-function. Define a function $h$ on $\mathbb{R}^n$ by 
\begin{equation*}
h(x)= \begin{cases}0, &  \mathcal{A}_{\mathcal{S}_{j},\vec{b}}^{\vec{\gamma}}(\vec{f})(x)=0, \\ \frac{\phi(\mathcal{A}_{\mathcal{S}_{j},\vec{b}}^{\vec{\gamma}}(\vec{f})(x))}{\mathcal{A}_{\mathcal{S}_{j},\vec{b}}
^{\vec{\gamma}}(\vec{f})(x)}, & otherwise.\end{cases}
\end{equation*}
Let $H_{l_1}(x)=\prod_{s =1}^{l_1}\big| b_s(x)-{\langle b_s \rangle}_{Q} \big| {\langle|f_s|\rangle}_Q$ and denote
\begin{equation*}
\begin{aligned}
\mathcal{A}_{\mathcal{S}_{j},\vec{b}}^{\vec{\gamma}}(\vec{f})(x)=: \sum_{Q \in \mathcal{S}_{j}} H_{l_1}(x)\prod_{s =l_1+1}^{l}{\langle|(b_s-{\langle b_s \rangle}_{Q} )f_s|\rangle}_Q \prod_{s =l+1}^{m}{\langle|f_s|\rangle}_Q \chi_Q(x), \\
\end{aligned}
\end{equation*}
Then the left-hand side of (\ref{ie1.5(1)}) can be rewritten as
\begin{equation}
\begin{aligned}
\int_{\mathbb{R}^{n}}\phi\left( \mathcal{A}_{\mathcal{S}_{j},\vec{b}}^{\vec{\gamma}}(\vec{f})(x)\right)w(x)dx=&\sum_{Q \in \mathcal{S}_{j}}\prod_{s =l_1+1}^{l}{\langle\big|(b_s-{\langle b_s \rangle}_{Q} )f_s\big|\rangle}_Q \prod_{s =l+1}^{m}{\langle|f_s|\rangle}_Q \\
&\quad \times \int_Q H_{l_1}(x)h(x)w(x)dx.
\end{aligned}
\end{equation}
Hence for any $r>1,$ Lemma \ref{lem2} with $s_1=s_2=\cdots=s_{l_1}=1,s=\frac{1}{l_1}$ yields
\begin{equation}\label{ie29}
\begin{aligned}
\int_{\mathbb{R}^{n}}\phi\left( \mathcal{A}_{\mathcal{S}_{j},\vec{b}}^{\vec{\gamma}}(\vec{f})(x)\right)w(x)dx &\lesssim \sum_{Q \in \mathcal{S}_{j}}\prod_{s =1}^{l_1}{\langle|f_s|\rangle}_Q \prod_{s =l+1}^{m}{\langle|f_s|\rangle}_Q \prod_{s =l_1+1}^{l}{\langle\big|(b_s-{\langle b_s \rangle}_{Q} )f_s\big|\rangle}_Q  \\
&\quad \quad\times \prod_{s =1}^{l_1}\left\|b_s-{\langle b_s \rangle}_{Q} \right\|_{\exp L(w),Q}\|h\|_{L(\log L)^{l_1}(w),Q} w(Q) \\
& \lesssim [w]_{A_{\infty }}^{l_1}\sum_{Q \in \mathcal{S}_{j}}\prod_{s =1}^{l_1} \|b_s\|_{\mathrm{BMO}}\prod_{s =l_1+1}^{l} {\langle\big|b_s-{\langle b_s \rangle}_{Q} \big|^{r^\prime}\rangle}_Q^{\frac{1}{r^\prime}} \\
&\quad \times \prod_{s =1}^{m}{\langle|f_s|^r\rangle}_Q^{\frac{1}{r}}\|h\|_{L(\log L)^{l_1}(w),Q} w(Q) \\
& \lesssim [w]_{A_{\infty }}^{l}\prod_{s =1}^{l}
\|b_s\|_{\mathrm{BMO}}\sum_{Q \in \mathcal{S}_{j}}\prod_{s =1}^{m}{\langle|f_s|^r\rangle}_Q^{\frac{1}{r}}\|h\|_{L(\log L)^{l_1}(w),Q} w(Q),
\end{aligned}
\end{equation}
where we have used $\langle\left|b-b_{Q}\right|^{t}\rangle_{Q}^{1 / t}\leq 2^{n+1}e^{3}  t\|b\|_{\mathrm{BMO}}$ (see \cite[p. 19]{wen}) in the last inequality.\\
 Let $\mathcal{M}_r$ be the multilinear maximal operator with power $r>1$ defined by
$$\mathcal{M}_r(\vec{f})(x):=\sup\limits _{Q \ni x} \prod\limits_{i=1}^m \left(\frac{1}{|Q|} \int_Q\left|f_i(y)\right|^r d y\right)^{\frac{1}{r}}.$$
 Using Carleson embedding theorem again, one may obtain
\begin{equation}\label{ie25}
\begin{aligned}
\int_{\mathbb{R}^{n}}\phi( \mathcal{A}_{\mathcal{S}_{j},\vec{b}}^{\vec{\gamma}}(\vec{f})(x))w(x)dx &\lesssim [w]_{A_{\infty }}^{l+1}\prod_{s =1}^{l}\|b_s\|_{\mathrm{BMO}} \int_{\mathbb{R}^{n}}\mathcal{M}_r\vec{f}(x)M_{L(\log L)^{l_1}(w)}^{\mathcal{D}_j}(h)(x)w(x)dx  \\
& \leq C'[w]_{A_{\infty }}^{l+1}\prod_{s =1}^{l}\|b_s\|_{\mathrm{BMO}} \int_{\mathbb{R}^{n}} \mathcal{M}_r\vec{f}(x)(M_w^{\mathcal{D}_j})^{l+1}(h)(x)w(x)dx.
\end{aligned}
\end{equation}\par
For convenience, we give some notation for the constants $a_1, a_2$ and $\alpha.$
First, since $\phi$ has the sub-multiplicative property, it follows that $\phi(2t)\leq \phi(2)\phi(t) $ which implies that $\phi \in \Delta _2.$ Hence, there exists some $0<\alpha < 1,$ such that $\bar{\phi}^{\alpha}$ is quasi-convex, which means that, there exists some convex function $\psi$ and $a_1> 1$ such that
\begin{equation*}
\psi(t)\leq \bar\phi^\alpha(t)\leq a_1\psi(a_1t), \ t>0.
\end{equation*}
Afterwards, for $w \in A_{\infty}$ and $\phi \in \Phi$ which satisfies that there exists $0<\alpha<1$ such that  $\phi^\alpha$ be a quasi-convex function, we recall the following modular inequality  \cite[Lemma 3.12]{and1} for the weighted maximal operator
\begin{equation}\label{ie26}
\int_{\mathbb{R}^n} \phi\left(M_w^{\mathcal{D}} f(x)\right) w(x) \mathrm{d} x \leq a_2 \int_{\mathbb{R}^n} \phi\left(a_2|f(x)|\right) w(x) \mathrm{d} x,
\end{equation}
where the constant $a_2>1$ only depends on $\phi$ and $\alpha$, and is independent of $w$. \par
Using the previous notation, we take some $\varepsilon$ such that
$$0<\varepsilon\leq \min\Big\{\frac{1}{2},\frac{1}{a_1a_2^{l+1}},\Big(\frac{1}{2C^\prime a_2^{l+1}[w]_{A_{\infty }}^{l+1}\prod_{s =1}^{l}\|b_s\|_{\mathrm{BMO}} }\Big)^\alpha\cdot\frac{1}{a_1^2a_2^{l+1}}\Big\},$$
where $C^\prime$ is determined by (\ref{ie25}).\\
Combining (\ref{ie25}) with the Young's inequality and applying (\ref{ie26}) $l+1$ times, one obtains
\begin{equation}\label{ie27}
\begin{aligned}
\int_{\mathbb{R}^{n}}\phi\left( \mathcal{A}_{\mathcal{S}_{j},\vec{b}}^{\vec{\gamma}}(\vec{f})(x)\right)w(x)dx &\leq C'[w]_{A_{\infty} }^{l+1}\prod_{s =1}^{l}\|b_s\|_{\mathrm{BMO}}\int_{\mathbb{R}^{n}}\frac{\mathcal{M}_r\vec{f}(x)}{\varepsilon}(M_w^{\mathcal{D}_j})^{l+1}
(\varepsilon h)(x)w(x)dx \\
&\leq C'[w]_{A_{\infty} }^{l+1}\prod_{s =1}^{l}\|b_s\|_{\mathrm{BMO}}\bigg[\int_{\mathbb{R}^{n}}\phi\Big(\frac{\mathcal{M}_r\vec{f}(x)}{\varepsilon} \Big )w(x)dx\\
& \hspace{0.5cm} +\int_{\mathbb{R}^{n}}\bar{\phi}\left((M_w^{\mathcal{D}_j})^{l+1}
(\varepsilon h)(x) \right)w(x)dx\bigg] \\
&\leq  C'[w]_{A_{\infty} }^{l+1}\prod_{s =1}^{l}\|b_s\|_{\mathrm{BMO}} \cdot 2^{C_1}\varepsilon^{-C_1}\int_{\mathbb{R}^{n}} \phi \left(\mathcal{M}_r\vec{f}(x)\right)w(x)dx \\
& \hspace{0.5cm} +  C'[w]_{A_{\infty} }^{l+1}\prod_{s =1}^{l}\|b_s\|_{\mathrm{BMO}}  \cdot a_2^{l+1}\int_{\mathbb{R}^{n}} \bar{\phi} \left(a_2^{l+1}\varepsilon h(x)\right)w(x)dx.
\end{aligned}
\end{equation}
Consider to estimate $\int_{\mathbb{R}^{n}}\bar\phi(a_2^{l+1}\varepsilon h(x))w(x)dx.$
Noting that $\bar\phi^{\alpha}$ is a quasi-convex function, it follows that
\begin{equation*}
\begin{aligned}
\bar{\phi}^\alpha \left(a_2^{l+1}\varepsilon h(x)\right) \leq a_1\psi\left(a_1a_2^{l+1}\varepsilon h(x)\right) \leq a_1^2a_2^{l+1}\varepsilon\psi\left(h(x)\right)\leq a_1^2a_2^{l+1}\varepsilon\bar{\phi}^\alpha\left(h(x)\right),
\end{aligned}
\end{equation*}
where in the above inequality  we have used $a_1a_2^{l+1}\varepsilon\leq 1$ and $ \psi(\lambda t)\leq\lambda\psi(t)(0\leq \lambda\leq1).$
 Thus, the definition of $h$ together with the fact that $\bar{\phi}\left(\frac{\phi(t)}{t}\right) \leq \phi(t)$ for $t>0,$ gives that
\begin{equation*}
\begin{aligned}
\int_{\mathbb{R}^{n}} \bar{\phi} \left(a_2^{l+1}\varepsilon h(x)\right)w(x)dx &\leq \left(a_1^2a_2^{l+1}\varepsilon  \right)^{\frac{1}{\alpha}}\int_{\mathbb{R}^{n}} \bar{\phi}\bigg(\frac{\phi( \mathcal{A}_{\mathcal{S}_{j},\vec{b}}^{\vec{\gamma}}(\vec{f})(x))}{\mathcal{A}_{\mathcal{S}_{j},
\vec{b}}^{\vec{\gamma}}(\vec{f})(x)}\bigg)w(x)dx \\
& \leq \left(a_1^2a_2^{l+1}\varepsilon  \right)^{\frac{1}{\alpha}}\int_{\mathbb{R}^{n}} \phi\left(\mathcal{A}_{\mathcal{S}_{j},
\vec{b}}^{\vec{\gamma}}(\vec{f})(x) \right)w(x)dx.
\end{aligned}
\end{equation*}
Plugging the above estimate into (\ref{ie27}), we obtain
\begin{equation*}
\begin{aligned}
\int_{\mathbb{R}^{n}}\phi\left( \mathcal{A}_{\mathcal{S}_{j},\vec{b}}^{\vec{\gamma}}(\vec{f})(x)\right)w(x)dx &\leq C'[w]_{A_{\infty} }^{l+1}\prod_{s =1}^{l}\|b_s\|_{\mathrm{BMO}}\cdot 2^{C_1}\varepsilon^{-C_1}\int_{\mathbb{R}^{n}} \phi \left(\mathcal{M}_r(\vec{f})(x)\right)w(x)dx \\
&\quad +  C'[w]_{A_{\infty} }^{l+1}\prod_{s =1}^{l}\|b_s\|_{\mathrm{BMO}}\cdot a_2^{l+1} \\
& \quad \quad \times \left(a_1^2a_2^{l+1}\varepsilon  \right)^{\frac{1}{\alpha}}\int_{\mathbb{R}^{n}} \phi\left(\mathcal{A}_{\mathcal{S}_{j},
\vec{b}}^{\vec{\gamma}}(\vec{f})(x) \right)w(x)dx,\\
\end{aligned}
\end{equation*}
which further implies that for $1\leq j\leq 3^n,$
\begin{equation}\label{ie28}
\begin{aligned}
\int_{\mathbb{R}^{n}}\phi\left( \mathcal{A}_{\mathcal{S}_{j},\vec{b}}^{\vec{\gamma}}(\vec{f})(x)\right)w(x)dx \lesssim \prod_{s =1}^{l}\|b_s\|_{\mathrm{BMO}} \cdot [w]_{A_{\infty} }^{(l+1)(\alpha {C_1}+1)}\int_{\mathbb{R}^{n}} \phi \left(\mathcal{M}_r(\vec{f})(x)\right)w(x)dx.
\end{aligned}
\end{equation}\par
We will give another proof of the modular inequality of $ \mathcal{A}_{\mathcal{S}_{j},\vec{b}}^{\vec{\gamma}},$
which has the advantage to give the best possible range of $i_\phi.$ \par
Similarly as the argument for (\ref{ie1}), one can verfy that there exists a  sparse family $\tilde{{\mathcal{S}}}$ such that for any $ l_1+1 \leq t \leq l,$
\begin{equation*}
\begin{aligned}
\big |b_t(x)-{\langle b_t \rangle}_{Q} \big| \leq 2^{n+2}\sum_{R \in\tilde{{\mathcal{S}}} , R\subseteq Q} {\langle \left|b_t-{\langle b_t \rangle}_{R} \right| \rangle}_R\chi_R(x). \\
\end{aligned}
\end{equation*}
Therefore
\begin{equation*}
\begin{aligned}
\prod_{s =l_1+1}^{l}{\langle\big|(b_s-{\langle b_s \rangle}_{Q} )f_s\big|\rangle}_Q &\lesssim\prod_{s =l_1+1}^{l}\frac{1}{|Q|} \int_Q \sum_{R \in\tilde{{\mathcal{S}}} , R\subseteq Q} {\langle \left|b_s-{\langle b_s \rangle}_{R} \right| \rangle}_R\chi_R(x)|f_s(x)|dx \\
&\lesssim \prod_{s =l_1+1}^{l}\|b_s\|_{\mathrm{BMO}}\sum_{R \in\tilde{{\mathcal{S}}} , R\subseteq Q}\frac{1}{|Q|}\int_R |f_s(x)|dx \\
&\lesssim \prod_{s =l_1+1}^{l}\|b_s\|_{\mathrm{BMO}}\frac{1}{|Q|}  \int_Q \mathcal{A}_{\tilde{\mathcal{S}}}(f_s)(x)dx.
\end{aligned}
\end{equation*}
It then follows from the above estimate and (\ref{ie29}) that
\begin{equation*}
\begin{aligned}
\int_{\mathbb{R}^{n}}\phi\left( \mathcal{A}_{\mathcal{S}_{j},\vec{b}}^{\vec{\gamma}}(\vec{f})(x)\right)w(x)dx &\lesssim \sum_{Q \in\tilde{\mathcal{S}}_j}\prod_{s =1}^{l_1}{\langle|f_s|\rangle}_Q \prod_{s =l+1}^{m}{\langle|f_s|\rangle}_Q\prod_{s =l_1+1}^{l}{\langle\big|(b_s-{\langle b_s \rangle}_{Q} )f_s\big|\rangle}_Q \\
&\quad \quad \times \prod_{s =1}^{l_1}\left\|b_s-{\langle b_s \rangle}_{Q} \right\|_{\exp L(w),Q}\|h\|_{L(\log L)^{l_1}(w),Q} w(Q) \\
&\lesssim [w]_{A_{\infty} }^{l_1}\prod_{s =1}^{l}\|b_s\|_{\mathrm{BMO}}\prod_{s =1}^{l_l}{\langle|f_s|\rangle}_Q \prod_{s =l+1}^{m}{\langle|f_s|\rangle}_Q \prod_{s =l_1+1}^{l}{\langle \mathcal{A}_{\tilde{\mathcal{S}}}f_s\rangle}_{Q}\\
&\quad \times  \|h\|_{L(\log L)^{l_1}(w),Q} w(Q) \\
&\lesssim [w]_{A_{\infty} }^{l+1}\prod_{s =1}^{l}\|b_s\|_{\mathrm{BMO}} \int_{\mathbb{R}^{n}}  \mathcal{M}(\vec{f^*})(x)M_{L(\log L)^{l_1}(w)}^{\mathcal{D}_j}(h)(x)w(x)dx.
\end{aligned}
\end{equation*}
where $\vec{f^*}=\left(f_1,\ldots,f_{l_1},\mathcal{A}_{\tilde{\mathcal{S}}}(f_{l_1+1}),\ldots,
\mathcal{A}_{\tilde{\mathcal{S}}}(f_l),f_{l+1},\ldots, f_m\right).$\par
The same reasoning as what we have done with the case for $\mathcal{M}_r$ then gives 
\begin{equation}\label{ie29.1}
\begin{aligned}
\int_{\mathbb{R}^{n}}\phi\left( \mathcal{A}_{\mathcal{S}_{j},\vec{b}}^{\vec{\gamma}}(\vec{f})(x)\right)w(x)dx &\leq C' [w]_{A_{\infty} }^{(1+l)(\alpha {C_1}+1)}\prod_{s =1}^{l}\|b_s\|_{\mathrm{BMO}} \int_{\mathbb{R}^{n}} \phi \left( \mathcal{M}(\vec{f^*})(x)\right)w(x)dx.
\end{aligned}
\end{equation}
Hence, by (\ref{ie28}) and (\ref{ie29.1}), we can dominate the left side of \ref {ie29.1} by a constant times
\begin{equation}\label{ie30}
\begin{aligned}{}&
	 [w]_{A_{\infty} }^{(1+l)(\alpha {C_1}+1)}\prod_{s =1}^{l}\|b_s\|_{\mathrm{BMO}} \min \left\{ \int_{\mathbb{R}^{n}} \phi \left( \mathcal{M}_r(\vec{f})(x)\right)w(x)dx,\int_{\mathbb{R}^{n}} \phi \left( \mathcal{M}(\vec{f^*})(x)\right)w(x)dx \right\}.
\end{aligned}
\end{equation}\par
Having obtained the above estimate, we are in a position to finish our proof. First, for any $r\geq 1,$ $w\in A_q $ with $1<q<\frac{i_{\phi}}{r}$, in \cite[Lemma 5.3] {tan2}, it was proved that there exists a constant $a_3>1$ such that
\begin{equation}\label{ie31}
\begin{aligned}
\int_{\mathbb{R}^{n}} \phi \left( \mathcal{M}_r(\vec{f})(x)\right)w(x)dx \leq  a_3 \left( \prod_{i =1}^{m} \int_{\mathbb{R}^{n}} \phi^m \big(a_3[w]_{A_q}^{\frac{1}{q_r}}|f_i(x)|\big)w(x)dx\right)^{\frac{1}{m}}.
\end{aligned}
\end{equation}
In particular, when $r=1$ then it holds for every $w\in A_q $ with $1<q<i_{\phi}$ that
\begin{equation*}
\begin{aligned}
\int_{\mathbb{R}^{n}} \phi \left( \mathcal{M}(\vec{f^*})(x)\right)w(x)dx &\leq a_3 \Bigg(\prod_{i \in \{1,\cdots, m\} \atop i \notin \{l_1+1,\cdots, l\}} \int_{\mathbb{R}^{n}} \phi^m \Big(a_3[w]_{A_q}^{\frac{1}{q}}|f_i(x)|\Big )w(x)dx\Bigg)^{\frac{1}{m}}\\
&\quad \times \Bigg( \prod_{i =l_1+1}^{l} \int_{\mathbb{R}^{n}} \phi^m \left(a_3[w]_{A_q}^{\frac{1}{q}}\mathcal{A}_{\tilde{\mathcal{S}}}(f_i)(x)\right)w(x)dx\Bigg)^{\frac{1}{m}}.
\end{aligned}
\end{equation*}
In order to apply $\int_{\mathbb{R}^{n}} \phi^m \big(a_3[w]_{A_q}^{\frac{1}{q}}\mathcal{A}_{\tilde{\mathcal{S}}}(f_i)(x)\big)w(x)dx$ $(i=l_1+1,\cdots,l)$ to Lemma \ref{lem6}, we need the following observation: for any $m\in \mathbb{N}^*,$ $\phi^m$  is an $N$-function and $I_{\phi^m}<\infty.$
First we show that $\phi^m$ is a convex function. To see this, it suffices to prove that if $f,g$ are $N$-functions then
$fg$ is a convex function, this means for every $x\geq y>0, \lambda\in (0,1),$
\begin{equation}\label{ie31.1}
f(\lambda x+(1-\lambda) y) g(\lambda x+(1-\lambda) y) \leq \lambda f(x) g(x)+(1-\lambda) f(y) g(y).
\end{equation}
 In deed, we can dominate $f(\lambda x+(1-\lambda) y) g(\lambda x+(1-\lambda) y)$ by
\begin{equation*}
\begin{aligned}
\lambda^2 f(x) g(x)+(1-\lambda)^2 f(y) g(y)+\lambda(1-\lambda)(f(x) g(y)+f(y) g(x)).
\end{aligned}
\end{equation*}
Then, a simple calculation gives
\begin{equation*}
\begin{aligned}
f(\lambda x+(1-\lambda) y) g(\lambda x+(1-\lambda) y) \leq\lambda f(x) g(x)+(1-\lambda)f(y) g(y),
\end{aligned}
\end{equation*}
where in the last inequality we have used the fact that $f,g$ are $N$-functions. This proves (\ref{ie31.1}).\\
On the one hand, it is easy to verify that
$$\lim _{t \rightarrow 0^{+}} \frac{\phi^m(t)}{t}=0.$$
On the other hand, using $ \lim _{t \rightarrow \infty}\phi(t)=\infty,$ we obtain that there exists $M>0$ such that
$\phi(t)>1$ holds for any $t>M.$ Therefore,
$$ \frac{\phi^m(t)}{t} \geq \frac{\phi(t)}{t}\rightarrow \infty (t\rightarrow\infty),$$
which implies that $\phi^m$ is an $N$-function. \\
It remains to prove that $I_{\phi^m}<\infty$ is valid. Recall that $ h_{\phi}(t)=\sup _{s>0} \frac{\phi(s t)}{\phi(s)}, t>0,$ and
$$I_{\phi}=\lim _{t \rightarrow \infty} \frac{\log h_{\phi}(t)}{\log t}=\inf _{1<t<\infty} \frac{\log h_{\phi}(t)}{\log t}.$$ Thus, we have $h_{{\phi}^m}(t)=h_{\phi}^m(t), $
which indicates that
$$ I_{{\phi}^m}=\inf _{1<t<\infty} \frac{\log h_{\phi}^m(t)}{\log t}=mI_{\phi}<\infty.$$\par
In virtue of the preceding observation, applying Lemma \ref{lem6} with $\tilde{C}=mC_1$, we have
\begin{equation*}
\begin{aligned}
{}&\Bigg(\prod_{i =l_1+1}^{l} \int_{\mathbb{R}^{n}} \phi^m \left(a_3[w]_{A_q}^{\frac{1}{q}}\mathcal{A}_{\tilde{\mathcal{S}}}(f_i)(x)\right)w(x)dx\Bigg)^{\frac{1}{m}}\\
&\quad\quad\quad\quad\lesssim [w]_{A_\infty}^{1+mC_1} \left( \prod_{i =l_1+1}^{l} \int_{\mathbb{R}^{n}} \phi^m \Big(a_3^2[w]_{A_q}^{\frac{2}{q}}|f_i(x)|\Big)w(x)dx\right)^{\frac{1}{m}}.
\end{aligned}
\end{equation*}
Therefore
\begin{equation*}
\int_{\mathbb{R}^{n}} \phi \left( \mathcal{M}(\vec{f^*})(x)\right)w(x)dx \lesssim [w]_{A_\infty}^{1+mC_1}\left(\prod_{i=1 }^m \int_{\mathbb{R}^{n}} \phi^m \left([w]_{A_q}^{\frac{2}{q}}|f_i(x)|\right)w(x)dx\right)^{\frac{1}{m}}.
\end{equation*}
This inequality, together with (\ref{ie30}) and (\ref{ie31}), yields that
\begin{enumerate}[(i)]
	\item if $r<i_{\phi}<\infty,$ then for
every~$1<q<\frac{i_{\phi}}{r}$ and ~$w\in A_q ,$
	 \begin{equation}\label{ie32}
\begin{aligned}
		\int_{\mathbb{R}^n} \phi\left( \mathcal{A}_{\mathcal{S}_{j},\vec{b}}^{\vec{\gamma}}(\vec{f})(x)\right) w(x) d x \lesssim & [w]_{A_\infty}^{(1+l)(\alpha C_1+1)}\prod_{s=1}^l\|b_s\|_{\mathrm{BMO}}^{1+\alpha C_1}\\
&\times\left(\prod_{i=1}^m\int_{\mathbb{R}^n} \phi^m\left([w]_{A_q}^{\frac{1}{qr}}\left|f_i(x)\right|\right) w(x) d x\right)^{\frac{1}{m}};
\end{aligned}
	\end{equation}
	\item if $1<i_{\phi}\leq r<\infty,$ then for
every~$1<q<i_{\phi}$ and ~$w\in A_q ,$
	 \begin{equation}\label{ie33}
\begin{aligned}
		\int_{\mathbb{R}^n} \phi\left( \mathcal{A}_{\mathcal{S}_{j},\vec{b}}^{\vec{\gamma}}(\vec{f})(x)\right) w(x) d x \lesssim & [w]_{A_\infty}^{(l+1)(\alpha C_1+1)+1+mC_1}\prod_{s=1}^l\|b_s\|_{\mathrm{BMO}}^{1+\alpha C_1}\\
&\times\left(\prod_{i=1}^m\int_{\mathbb{R}^n} \phi^m\left([w]_{A_q}^{\frac{2}{q}}\left|f_i(x)\right|\right) w(x) d x\right)^{\frac{1}{m}}.
\end{aligned}
	\end{equation}
\end{enumerate}
Finally, applying the convexity of $\phi$ and replacing $\mathcal{A}_{\mathcal{S}_{j},\vec{b}}^{\vec{\gamma}}(\vec{f})$ with $\mathcal{T}_{\vec{b}}(\vec{f})$, inequalities (\ref{ie32}) and (\ref{ie33}) still hold, which completes the proof of Theorem \ref{thm1.5}.

\end{proof}

\section{ applications}\label{Sect 8}

In this section, we present some applications of the results obtained in Section \ref{sub1.3}. We
will see that the Hypothesis \ref{hyp1} or Hypothesis \ref{hyp2} hold for multilinear $\omega$-Calder\'{o}n-Zygmund operators, multilinear pseudo-differential operators, higher order Calder\'{o}n commutators, and Stein's square functions. Furthermore, we will establish weighted modular estimates for them.

\subsection{ Multilinear $\omega$-Calder\'{o}n-Zygmund  operators} \label{S_1}
We recall the definition of multilinear Calder\'{o}n-Zygmund operator of type $\omega$.
\begin{definition}[\textbf{Multilinear $\omega$-Calder\'{o}n-Zygmund  operator}]\label{def8.1}
 Let $\omega(t):[0, \infty) \rightarrow[0, \infty)$ be a nondecreasing function. A locally integrable function $K\left(x, y_1, \ldots, y_m\right)$, defined away from the diagonal $x=y_1=$ $\cdots=y_m$ in $\left(\mathbb{R}^n\right)^{m+1}$, is called an $m$-linear Calder\'{o}n-Zygmund kernel of type $\omega$ if, for some constants $0<\tau<1$, there exists a constant $A>0$ such that
$$
\left|K\left(x, y_1, \ldots, y_m\right)\right| \leq \frac{A}{\left(\left|x-y_1\right|+\cdots+\left|x-y_m\right|\right)^{m n}}
$$
for all $\left(x, y_1, \ldots, y_m\right) \in\left(\mathbb{R}^n\right)^{m+1}$ with $x \neq y_j$ for some $1\leq j \leq m$, and
$$
\begin{aligned}
& \left|K\left(x, y_1, \ldots, y_m\right)-K\left(x^{\prime}, y_1,\ldots, y_m\right)\right| \\
& \qquad \leq \frac{A}{\left(\left|x-y_1\right|+\cdots+\left|x-y_m\right|\right)^{m n}} \omega\left(\frac{\left|x-x^{\prime}\right|}{\left|x-y_1\right|+\cdots+\left|x-y_m\right|}\right)
\end{aligned}
$$
whenever $\left|x-x^{\prime}\right| \leq \tau \max _{1 \leq j \leq m}\left|x-y_j\right|$, and for all $1\leq i \leq m$
$$
\begin{aligned}
& \left|K\left(x, y_1, \ldots, y_i, \ldots, y_m\right)-K\left(x, y_1, \ldots, y_i^{\prime}, \ldots, y_m\right)\right| \\
& \qquad\leq \frac{A}{\left(\left|x-y_1\right|+\cdots+\left|x-y_m\right|\right)^{m n}} \omega\left(\frac{|y_i-y_i^{\prime}|}{\left|x-y_1\right|+\cdots+\left|x-y_m\right|}\right)
\end{aligned}
$$
whenever $\left|y_i-y_i^{\prime}\right| \leq \tau \max _{1 \leq j \leq m}\left|x-y_j\right|$. Particularly, when $\omega(t)=t^\delta$ with $\delta \in(0,1], K$ is called an $m$-linear standard Calder\'{o}n-Zygmund kernel.\\
We say $T$: $\mathscr{S}\left(\mathbb{R}^n\right) \times \cdots \times \mathscr{S}\left(\mathbb{R}^n\right) \rightarrow \mathscr{S}^{\prime}\left(\mathbb{R}^n\right)$ is an $m$-linear operator with an $m$-linear Calder\'{o}n-Zygmund kernel of type $\omega, K\left(x, y_1, \ldots, y_m\right)$, if
$$
T\left(f_1, \ldots, f_m\right)(x)=\int_{(\mathbb{R}^n)^m} K\left(x, y_1, \ldots, y_m\right) f_1\left(y_1\right) \cdots f_m\left(y_m\right) d y_1 \cdots d y_m
$$
whenever $x \notin \bigcap_{j=1}^m \operatorname{supp} f_j$ and each $f_j$ with $j=1, \ldots, m$ is a bounded function with compact support.
If $T$ can be extended to a bounded multilinear operator from $L^{q_1}\left(\mathbb{R}^n\right) \times \cdots \times$ $L^{q_m}\left(\mathbb{R}^n\right)$ to $L^{q, \infty}\left(\mathbb{R}^n\right)$ for some $1 / q=1 / q_1+\cdots+1 / q_m$ with $1 \leq q_1, \ldots, q_m<\infty$, then $T$ is called an $m$-linear Calder\'{o}n-Zygmund operator of type $\omega$.
\end{definition}

\begin{definition}[\textbf{$\log$-Dini condition }]\label{def8.2}
 Let $\omega(t):[0, \infty) \rightarrow[0, \infty)$ be a nondecreasing function. For $a>0, m\in \mathbb{N}$, we say that $\omega$ satisfies the $\log$-Dini$(a,m)$ condition, denote $\omega \in$ $\log$-Dini$(a,m)$, if
$$
\|\omega\|_{\log\mathrm{-Dini}(a,m)}:=\int_0^1 \frac{\omega^a(t)}{t} \left(1+\log \frac{1}{t}\right)^m d t<\infty .
$$
\end{definition}
It is worth mentioning that, in 2014, Lu and Zhang \cite{zha} obtained the weighted inequalities for commutators of multilinear Calder\'{o}n-Zygmund operators $T$ of type $\omega$ with $\omega \in$ $\log$-Dini$(1,m).$ Very recently, Cao et al. \cite{cao1} proved the local exponential decay, mixed weak type estimate for $T$ with $\omega \in$ $\log$-Dini$(1,0),$ and obtained the weighted compactness for commutators.\par
As was shown in \cite[Theorem 2.5]{cao1}, both Hypothesis \ref{hyp1} and Hypothesis  \ref{hyp2} hold. By Theorems \ref{thm1.1}, \ref{thm1.3} and \ref{thm1.4}, we obtain
\begin{theorem}\label{thm8.1}
Let $I=\{i_1,\ldots,i_l\}=\{1,\ldots,l\}\subseteq \{1,\ldots,m\},$ $T$ be an $m$-linear $\omega$-Calder\'{o}n-Zygmund operator with $\omega \in$ $\log$-Dini$(1,m).$ If $\vec{b} \in \mathrm{BMO}^l,$ then we have
\begin{enumerate}[(a).]
		\item Let $w\in A_{\infty}^{\text{weak}},$ $Q_0$ be a cube and $f_s \in L_c^{\infty}\left(\mathbb{R}^n\right)$ such that $\operatorname{supp}\left(f_s\right) \subset Q_0$ for $1 \leq s\leq m.$ Then there are constants $\alpha, c>0$ independent of $w$ such that
\begin{equation*}
\begin{aligned}
&w\left(\left\{x \in Q_0:\left|T_{\vec{b}}(\vec{f})(x)\right|>t \mathcal{M}_{L(\log L)}(\vec{f})(x)\right\}\right) \\
&\hspace{5cm} \qquad\leq c e^{-{\frac{\alpha}{[w]_{A_\infty}^{weak}+1}}\left({\frac{t}{\prod_{s =1}^{l} \|b_s\|_{\mathrm{BMO}}}}\right)^{\frac{1}{l+1}}} w(2Q_0), \quad t>0.
\end{aligned}
\end{equation*}
      \item Let $\vec{w}=\left(w_1, \ldots, w_m\right)$ and $u=\prod_{i=1}^m w_i^{1 / m}$. If $\vec{w} \in A_{\vec{1}}$ and $v \in A_{\infty},$ then there exists $t>1 $ depending only on $v$, such that
$$
\left\|\frac{T_{\vec{b}}(\vec{f})}{v}\right\|_{L^{\frac{1}{m}, \infty}(u v^{\frac{1}{m}})} \lesssim K_0^{2l+6m}[v^{\frac{1}{m}}]_{A_t}^{2l+4m}\prod_{s =1}^{l} \|b_s\|_{\mathrm{BMO}}\left\|\frac{ \mathcal{M}_{L(\log L)} (\vec{f})}{v}\right\|_{L^{\frac{1}{m}, \infty}(u v^{\frac{1}{m}})},
$$
where $K_0=4C_np_0p_0^{\prime}([u]_{A_1}+2^{p_0-1}C_n^t[v^{\frac{1}{m}}]_{A_t}^2[u]_{A_1}^{p_0-1})+1$ with $p_0=2^{n+3}(t-1)[u]_{A_1}+1.$
      \item Let $1<p_1, \ldots, p_m<\infty$ and $\frac{1}{p}=\frac{1}{p_1}+\cdots+\frac{1}{p_m}.$ Assume that $m\geq 2$ and for all weights $\vec{w}=\left(w_1, \ldots, w_m\right), \nu_{\vec{w}}=\prod_{s=1}^m w_s^{p / p_s},$ $b_s \in \mathrm{BMO}_{p_s}(w_s)\cap \mathrm{BMO}$ with $1\leq s\leq l.$ If $0<p\leq1,$ then

$$
\left\|T_{\vec{b}}(\vec{f})\right\|_{L^p\left(\nu_{\vec{w}}\right)} \leq C \|\vec{b}\|_{\mathrm{BMO}}^*\prod_{s=1}^m\left\|f_s\right\|_{L^{p_s}(M w_s)},
$$
where $C$ is independent of $\vec{w}$ and $\vec{b},$ and $$\|\vec{b}\|_{\mathrm{BMO}}^*=\max_{\vec{\gamma}\in \{1,2\}^l}\{\prod_{s:\gamma_s=1}\|b_s\|_{\mathrm{BMO}_{p_s}(w_s)}\prod_{s:\gamma_s=2}\|b_s\|_{\mathrm{BMO}}\}.$$
\end{enumerate}

\end{theorem}
\subsection{ Multilinear maximal singular integral operators}
In this section, we will consider the weighted estimates of the commutators of multilinear maximal operator $T^*$, which is defined as
$$
T^*(\vec{f})(x):=\sup _{\delta>0}\left|\int_{\sum_{i=1}^m\left|y_i-x\right|^2>\delta^2} K\left(x, y_1, \ldots, y_m\right) f_1\left(y_1\right) \cdots f_m\left(y_m\right) d \vec{y}\right|,
$$
for $x \notin \bigcap_{j=1}^m \operatorname{supp} f_j$ and each $f_j \in L_c^{\infty}(j=1, \ldots, m)$, where $d \vec{y}=d y_1 \cdots d y_m$ and $K(x,\vec{y})$ is an $m$-linear Calder\'{o}n-Zygmund kernel of type $\omega,$ which is defined in Subsection \ref{S_1}.
In this subsection, we consider the following commutators of multilinear maximal singular integral operators:
$$
T_{\vec{b}}^*(\vec{f})(x):=\sum_{j=1}^m \sup _{\delta>0}\left|\int_{\sum_{i=1}^m\left|y_i-x\right|^2>\delta^2}\left(b_j(x)-b_j\left(y_j\right)\right) K\left(x,\vec{y}\right) f_1\left(y_1\right) \cdots f_m\left(y_m\right) d \vec{y}\right|.
$$
Now, we use \cite[Theorems 1.8]{zhan}, Theorem \ref{thm1.2}-\ref{thm1.3} and Theorem \ref{thm1.5} to conclude the following results.
\begin{theorem}\label{thm8.1}
	Let $T$ be an $m$-linear $\omega$-Calder\'{o}n-Zygmund operator with $\omega \in$ $\log$-Dini$(1,0).$ If $\vec{b} \in \mathrm{BMO}^m,$ then we have:
	\begin{enumerate}[(a).]
		\item Let $\vec{w}=\left(w_1, \ldots, w_m\right)$ and $u=\prod_{i=1}^m w_i^{1 / m}$. If $\vec{w} \in A_{\vec{1}}$ and $v \in A_{\infty},$ then there exists $t>1 $ depending only on $v$, such that
		$$
		\left\|\frac{T^*_{\vec{b}}(\vec{f})}{v}\right\|_{L^{\frac{1}{m}, \infty}(u v^{\frac{1}{m}})} \lesssim K_0^{2+6m}[v^{\frac{1}{m}}]_{A_t}^{2+4m}\|\vec{b}\|_{\mathrm{BMO}}\left\|\frac{ \mathcal{M}_{L(\log L)} (\vec{f})}{v}\right\|_{L^{\frac{1}{m}, \infty}(u v^{\frac{1}{m}})},
		$$
		where $K_0=4C_np_0p_0^{\prime}([u]_{A_1}+2^{p_0-1}C_n^t[v^{\frac{1}{m}}]_{A_t}^2[u]_{A_1}^{p_0-1})+1$ with $p_0=2^{n+3}(t-1)[u]_{A_1}+1$ and $\|\vec{b}\|_{\mathrm{BMO}}:=\sup _{1 \leq j \leq m}\left\|b_j\right\|_{\mathrm{BMO}}.$
		\item For any $0<p<\infty, w\in A_\infty,$
		\begin{equation*}
			\int_{\mathbb{R}^n}\left|T^*_{\vec{b}}(\vec{f})(x)\right|^pw(x)dx \lesssim \|\vec{b}\|_{\mathrm{BMO}}[w]_{A_\infty}^{p}
			[w]_{A_\infty}^{\max\{2,p\}}\int_{\mathbb{R}^n}\left(\mathcal{M}_{L(\log L)}(\vec{f})(x)\right)^pw(x)dx,
		\end{equation*}
		where $\|\vec{b}\|_{\mathrm{BMO}}:=\sup _{1 \leq j \leq m}\left\|b_j\right\|_{\mathrm{BMO}}.$
		\item Let~$\phi$ be a~$N$-function with sub-multiplicative property. For any $1<r<\infty,$
		\begin{enumerate}[(1)]
			\item if $r<i_{\phi}<\infty,$ then there exists constant~$\alpha$ such that for
			every~$1<q<\frac{i_{\phi}}{r}$ and ~$w\in A_q ,$
			\begin{equation*}
				\begin{aligned}
					\int_{\mathbb{R}^n} \phi\left(T^*_{\vec{b}}(\vec{f})(x)\right) w(x) d x \lesssim & [w]_{A_\infty}^{2\alpha C_1+1}\|\vec{b}\|_{\mathrm{BMO}}^{1+\alpha C_1}\\
					&\times\left(\prod_{i=1}^m\int_{\mathbb{R}^n} \phi^m\left([w]_{A_q}^{\frac{1}{qr}}\left|f_i(x)\right|\right) w(x) d x\right)^{\frac{1}{m}};
				\end{aligned}
			\end{equation*}
			\item if $1<i_{\phi}\leq r,$ then there exists constant~$\alpha$ such that for
			every~$1<q<i_{\phi}$ and ~$w\in A_q ,$
			\begin{equation*}
				\begin{aligned}
					\int_{\mathbb{R}^n} \phi\left(T^*_{\vec{b}}(\vec{f})(x)\right) w(x) d x \lesssim & [w]_{A_\infty}^{2\alpha C_1+mC_1+2}\|\vec{b}\|_{\mathrm{BMO}}^{1+\alpha C_1}\\
					&\times\left(\prod_{i=1}^m\int_{\mathbb{R}^n} \phi^m\left([w]_{A_q}^{\frac{2}{q}}\left|f_i(x)\right|\right) w(x) d x\right)^{\frac{1}{m}}.
				\end{aligned}
			\end{equation*}
			
		\end{enumerate}
	\end{enumerate}
	
\end{theorem}

\subsection{ Multilinear pseudo-differential operators}

Given a function $\sigma$ on $\mathbb{R}^n \times \mathbb{R}^{n m}$, the $m$-linear pseudo-differential operator $T_\sigma$ is defined by
$$
T_\sigma(\vec{f})(x):=\int_{(\mathbb{R}^{n})^ {m}} \sigma(x, \vec{\xi}) e^{2 \pi i x \cdot\left(\xi_1+\cdots+\xi_m\right)} \widehat{f_1}\left(\xi_1\right) \cdots \widehat{f_m}\left(\xi_m\right) d \vec{\xi}
$$
for all $f_i \in \mathscr{S}\left(\mathbb{R}^n\right), i=1, \ldots, m$, where $d\vec{\xi}=d\xi_1\cdots d\xi_m,$ and $\widehat{f}$ is the Fourier transform  of the function $f$ defined by
 $$\widehat{f}(\xi)=\int_{\mathbb{R}^n} f(x) e^{-2\pi i x\cdot \xi} dx.$$

Given $m \in \mathbb{N},$ $r \in \mathbb{R}$ and $0\leq \rho, \delta \leq 1.$
We say a smooth function $\sigma$ belongs to the H\"{o}rmander class $S_{\rho, \delta}^r(n, m)$ if for each triple of multi-indices $\alpha:=\left(\alpha_1, \ldots, \alpha_n\right)$ and $\beta_1, \ldots, \beta_m$, there exists a constant $C_{\alpha, \beta}$ such that
$$
\left|\partial_x^\alpha \partial_{\xi_1}^{\beta_1} \cdots \partial_{\xi_m}^{\beta_m} \sigma(x, \vec{\xi})\right| \leq C_{\alpha, \beta}\left(1+\left|\xi_1\right|+\cdots+\left|\xi_m\right|\right)^{r-\rho \sum_{j=1}^m\left|\beta_j\right|+\delta|\alpha|}.
$$
It was shown in \cite{wen1} that if $\sigma\in S_{\rho, \delta}^r(n, m)$, then the iterated commutator of the pseudo-differential operator $T_\sigma$ is weighted bounded from $L^{p_1}\times \cdots \times L^{p_m}$ to $L^{p}$ with $1 / p=1 / p_1+\cdots+1 / p_m$. Using Theorem \ref{thm1.1}, \ref{thm1.2} and Theorem \ref{thm1.5}, combined with Theorem 1.1 in \cite{wen1}, we know that both Hypothesis \ref{hyp1} and Hypothesis  \ref{hyp2} hold. Therefore, we obtain
\begin{theorem}\label{thm8.2}
Let $T_\sigma$ be an m-linear pseudo-differential operator, $\sigma \in S_{\rho, \delta}^r(n$, $m)$ with $0 \leq \rho, \delta \leq 1$ and $r<2 n(\rho-1)$.
Let $I=\{i_1,\ldots,i_l\}=\{1,\ldots,l\}\subseteq \{1,\ldots,m\}.$  If $\vec{b} \in \mathrm{BMO}^l,$ then the following hold:
\begin{enumerate}[(a).]
		\item Let $Q_0$ be a cube and $f_s \in L_c^{\infty}\left(\mathbb{R}^n\right)$ such that $\operatorname{supp}\left(f_s\right) \subset Q_0$ for $1 \leq s\leq m.$ If $w\in A_{\infty},$ then there are constants $\alpha, c>0$ independent of $w$ such that
\begin{equation*}
\begin{aligned}
&w\left(\left\{x \in Q_0:\left|T_{\sigma,\vec{b}}(\vec{f})(x)\right|>t \mathcal{M}_{L(\log L)}(\vec{f})(x)\right\}\right) \\
&\hspace{5cm} \qquad\leq c e^{-{\frac{\alpha}{[w]_{A_\infty}}}\left({\frac{t}{\prod_{s =1}^{l} \|b_s\|_{\mathrm{BMO}}}}\right)^{\frac{1}{l+1}}} w(Q_0), \quad t>0.
\end{aligned}
\end{equation*}

      \item For any $0<p<\infty, w\in A_\infty,$
\begin{equation*}
\int_{\mathbb{R}^n}\left|T_{\sigma,\vec{b}}(\vec{f})(x)\right|^pw(x)dx \lesssim \prod_{s =1}^{l} \|b_s\|_{\mathrm{BMO}}[w]_{A_\infty}^{pl}
[w]_{A_\infty}^{\max\{2,p\}}\int_{\mathbb{R}^n}\left(\mathcal{M}_{L(\log L)}(\vec{f})(x)\right)^pw(x)dx.
\end{equation*}

      \item Let~$\phi$ be a~$N$-function with sub-multiplicative property. For any $1<r<\infty,$
\begin{enumerate}[(1)]
	\item if $r<i_{\phi}<\infty,$ then there exists constant~$\alpha$ such that for
every~$1<q<\frac{i_{\phi}}{r}$ and ~$w\in A_q ,$
	 \begin{equation*}
\begin{aligned}
		\int_{\mathbb{R}^n} \phi\left(T_{\sigma,\vec{b}}(\vec{f})(x)\right) w(x) d x \lesssim & [w]_{A_\infty}^{(l+1)(\alpha C_1+1)}\prod_{s=1}^l\|b_s\|_{\mathrm{BMO}}^{1+\alpha C_1}\\
&\times\left(\prod_{i=1}^m\int_{\mathbb{R}^n} \phi^m\left([w]_{A_q}^{\frac{1}{qr}}\left|f_i(x)\right|\right) w(x) d x\right)^{\frac{1}{m}};
\end{aligned}
	\end{equation*}
	\item if $1<i_{\phi}\leq r,$ then there exists constant~$\alpha$ such that for
every~$1<q<i_{\phi}$ and ~$w\in A_q ,$
	 \begin{equation*}
\begin{aligned}
		\int_{\mathbb{R}^n} \phi\left(T_{\sigma,\vec{b}}(\vec{f})(x)\right) w(x) d x \lesssim & [w]_{A_\infty}^{(l+1)(\alpha C_1+1)+1+mC_1}\prod_{s=1}^l\|b_s\|_{\mathrm{BMO}}^{1+\alpha C_1}\\
&\times\left(\prod_{i=1}^m\int_{\mathbb{R}^n} \phi^m\left([w]_{A_q}^{\frac{2}{q}}\left|f_i(x)\right|\right) w(x) d x\right)^{\frac{1}{m}}.
\end{aligned}
	\end{equation*}

\end{enumerate}
\end{enumerate}
\end{theorem}

\subsection{ Higher order Calder\'{o}n commutators}
 In this subsection, we apply our results to derive the quantitative weak $A_\infty$ decay estimates,  mixed weak type estimates and  Fefferman-Stein inequalities with arbitrary weights of higher order Calder\'{o}n commutators on $\mathbb{R}$. These operators and its higher-order counterpart first appeared in the investigation of Cauchy integrals along Lipschitz curves, in the proof of the $L^2$ boundedness of the latter. \par
 For our purpose, we first present some definitions.
Given functions $A_1, \ldots, A_m$ defined on $\mathbb{R},$ let $a_j=\frac{dA_j}{dt}, j=1, \ldots, m$. The higher order Calder\'{o}n commutators are defined by
$$
\mathcal{C}_{m+1}\left(a_1, \ldots, a_{m}, f\right)(x):=\text { p.v. } \int_{\mathbb{R}} \frac{\prod_{j=1}^{m}\left(A_j(x)-A_j(y)\right)}{(x-y)^{m}} f(y) d y .
$$
Using the method in \cite[p. 2106]{duo}, we can rewrite $\mathcal{C}_{m+1}$ in the form of the multilinear singular integral as follows.
$$
\mathcal{C}_{m+1}\left(a_1, \ldots, a_{m}, f\right)(x)=\int_{\mathbb{R}^{m+1}} K(x, y_1,\ldots,y_{m+1}) f\left(y_{m+1}\right)\prod_{j=1}^{m} a_j\left(y_j\right)  dy_1\cdots dy_{m+1},
$$
where the kernel
\begin{equation}\label{c1}
 K(x, y_1,\ldots,y_{m+1})=\frac{(-1)^{m e\left(y_{m+1}-x\right)}}{\left(x-y_{m+1}\right)^{m+1}} \prod_{j=1}^{m} \chi_{\left(x \wedge y_{m+1}, x \vee y_{m+1}\right)}\left(y_j\right)
\end{equation}
with $ x \wedge y=\min \{x, y\}$, $x \vee y=\max \{x, y\},$ and
$$
e(x)= \begin{cases}1, & x>0, \\ 0, & x<0 .\end{cases}
$$
Whenever $\left|x-x^{\prime}\right| \leq \frac{1}{8} \min _{1 \leq j \leq m+1}\left|x-y_j\right|$, it was shown in \cite{hu} that
$$
|K(x, y_1,\ldots,y_{m+1})| \lesssim \frac{1}{\left(\sum_{j=1}^{m+1}\left|x-y_j\right|\right)^{m+1}}
$$
and
$$
\left|K(x, y_1,\ldots,y_{m+1})-K(x^\prime, y_1,\ldots,y_{m+1})\right| \lesssim \frac{\left|x-x^{\prime}\right|}{\left(\sum_{j=1}^{m+1}\left|x-y_j\right|\right)^{m+2}}.
$$
It is natural to generalize $\mathcal{C}_{m+1}$ to multilinear version as following.

\begin{equation}\label{c2}
\mathscr{C}(\vec{f})(x):=\int_{\mathbb{R}^{m+1}}K(x, y_1,\ldots,y_{m+1}) \prod_{j=1}^{m+1} f_j\left(y_j\right) dy_1\cdots dy_{m+1},
\end{equation}
where the kernel $K$ is given in (\ref{c1}).\par
With \cite[Theorem 2.24 ]{cao1} in hand, both Hypothesis \ref{hyp1} and Hypothesis  \ref{hyp2} hold. One can obtain the following results by Theorems \ref{thm1.1}, \ref{thm1.3} and \ref{thm1.4}.
\begin{theorem}\label{thm8.3}
Let $\mathscr{C}$ be the operator in (\ref{c2}) with the kernel $K$ given by (\ref{c1}). If $I=\{i_1,\ldots,i_l\}=\{1,\ldots,l\}\subseteq \{1,\ldots,m+1\},$  $\vec{b} \in \mathrm{BMO}^l,$ then the following statement are true
\begin{enumerate}[(a).]
		\item Let $w\in A_{\infty}^{\text{weak}},$ $Q_0$ be an interval and $f_s \in L_c^{\infty}\left(\mathbb{R}\right)$ such that $\operatorname{supp}\left(f_s\right) \subset Q_0$ for $1 \leq s\leq m+1.$ Then there are constants $\alpha, c>0$ independent of $w$ such that
\begin{equation*}
\begin{aligned}
&w\left(\left\{x \in Q_0:\left|\mathscr{C}_{\vec{b}}(\vec{f})(x)\right|>t \mathcal{M}_{L(\log L)}(\vec{f})(x)\right\}\right) \\
&\hspace{5cm} \qquad\leq c e^{-{\frac{\alpha}{[w]_{A_\infty}^{weak}+1}}\left({\frac{t}{\prod_{s =1}^{l} \|b_s\|_{\mathrm{BMO}}}}\right)^{\frac{1}{l+1}}} w(2Q_0), \quad t>0.
\end{aligned}
\end{equation*}
      \item Let $\vec{w}=\left(w_1, \ldots, w_{m+1}\right)$ and $u=\prod_{i=1}^{m+1} w_i^{1 / (m+1)}$. If $\vec{w} \in A_{\vec{1}}$ and $v \in A_{\infty},$ then there exists $t>1 $ depending only on $v$, such that
$$
\left\|\frac{\mathscr{C}_{\vec{b}}(\vec{f})}{v}\right\|_{L^{\frac{1}{m+1}, \infty}(u v^{\frac{1}{m+1}})} \lesssim \mathcal{K}\prod_{s =1}^{l} \|b_s\|_{\mathrm{BMO}}\left\|\frac{ \mathcal{M}_{L(\log L)} (\vec{f})}{v}\right\|_{L^{\frac{1}{m+1}, \infty}(u v^{\frac{1}{m+1}})},
$$
where
$\mathcal{K}=K_0^{2l+6(m+1)}[v^{\frac{1}{m+1}}]_{A_t}^{2l+4(m+1)}$ with
$$K_0=Cp_0p_0^{\prime}([u]_{A_1}+2^{p_0-1}C^t[v^{\frac{1}{m+1}}]_{A_t}^2[u]_{A_1}^{p_0-1})+1, p_0=16(t-1)[u]_{A_1}+1.$$
      \item Let $1<p_1, \ldots, p_{m+1}<\infty$ and $\frac{1}{p}=\frac{1}{p_1}+\cdots+\frac{1}{p_{m+1}}.$ Assume that $m\geq 1$ and for all weights $\vec{w}=\left(w_1, \ldots, w_{m+1}\right), \nu_{\vec{w}}=\prod_{s=1}^{m+1} w_s^{p / p_s},$ $b_s \in \mathrm{BMO}_{p_s}(w_s)\cap \mathrm{BMO}$ with $1\leq s\leq l.$ If $0<p\leq1,$ then

$$
\left\|\mathscr{C}_{\vec{b}}(\vec{f})\right\|_{L^p\left(\nu_{\vec{w}}\right)} \leq C \|\vec{b}\|_{\mathrm{BMO}}^*\prod_{s=1}^{m+1}\left\|f_s\right\|_{L^{p_s}(M w_s)},
$$
where $C$ is independent of $\vec{w}$ and $\vec{b},$ and $$\|\vec{b}\|_{\mathrm{BMO}}^*=\max_{\vec{\gamma}\in \{1,2\}^l}\{\prod_{s:\gamma_s=1}\|b_s\|_{\mathrm{BMO}_{p_s}(w_s)}\prod_{s:\gamma_s=2}\|b_s\|_{\mathrm{BMO}}\}.$$
\end{enumerate}

\end{theorem}

\subsection{ Stein's square functions}
The Stein's square function $G_\alpha$ is defined by
$$
G_\alpha f(x)=\left(\int_0^{\infty}\left|\frac{\partial}{\partial t} B_\alpha^t f(x)\right|^2 t d t\right)^{1 / 2}, \qquad \hbox{for }  \alpha>0,
$$
where $B_\alpha^t$ is the Bochner-Riesz multiplier
$
\widehat{B_\alpha^t f}(\xi)=\left(1-\frac{|\xi|^2}{t^2}\right)_{+}^\alpha \widehat{f}(\xi).
$
A simple calculation gives that
$$
\frac{\partial}{\partial t} B_\alpha^t f(x)=\frac{2 \alpha}{t} \int_{\mathbb{R}^n} \frac{|\xi|^2}{t^2}\left(1-\frac{|\xi|^2}{t^2}\right)_{+}^{\alpha-1} \widehat{f}(\xi) e^{2 \pi i x \xi} d \xi.
$$
Let $\widehat{K_t^\alpha}(\xi)=\frac{|\xi|^2}{t^2}\left(1-\frac{|\xi|^2}{t^2}\right)_{+}^{\alpha-1}$. Then, $G_\alpha$ can be rewritten as
$$
G_\alpha f(x)=\left(\int_0^{\infty}\left|K_t^\alpha * f(x)\right|^2 \frac{d t}{t}\right)^{1 / 2}.
$$
The function $G_\alpha$ was first introduced by Stein \cite{ste} to study $L^2$ properties of the maximal Bochner-Riesz operator and deduce almost everywhere convergence for Bochner-Riesz means of Fourier series.
 Invoking \cite[Theorem 1.1]{car1} and Theorems \ref{thm1.3} and \ref{thm1.5}, we know Hypothesis  \ref{hyp2} hold and thus we may obtain the following results.
\begin{theorem}\label{thm8.4}
Let $\alpha > \frac{n+1}{2},$ then the following hold:
\begin{enumerate}[(a).]
		\item If $w \in A_1$ and $v \in A_{\infty}$, then
$$
\left\|\frac{G_\alpha f}{v}\right\|_{L^{1, \infty}\left(w v\right)} \lesssim \left\|f\right\|_{L^1\left( w\right)} .
$$
      \item Let $\phi$ be an $N$-function belonging to $\Delta_2$ and $w \in A_{i_\phi}$. If $i_\phi>1$, then there exists $C_0>0$ such that
$$
\int_{\mathbb{R}^n} \phi\left(\left|G_\alpha f(x)\right|\right) w(x) dx \lesssim C(\phi, w) \int_{\mathbb{R}^n} \phi(|f(x)|) w(x) dx,
$$
where
$$
C(\phi, w)= \begin{cases}[w]_{A_{\infty}}^{1+\alpha C_1}, & C_0[w]_{A_{i_\phi}}^{1 / i_\phi}<2, \\ [w]_{A_{\infty}}^{1+\alpha C_1}\left([w]_{A_{i_\phi}}^{1 / i_\phi}\right)^{C_1}, & C_0[w]_{A_{i_\phi}}^{1 / i_\phi} \geq 2.\end{cases}
$$
\end{enumerate}

\end{theorem}

\end{document}